\documentclass[11pt]{amsart}

\usepackage{amssymb, amsmath, amsthm, amsxtra, bbm}
\usepackage{slashed, verbatim, dsfont}

\usepackage{scalerel,stackengine}
\stackMath
\newcommand\reallywidehat[1]{
\savestack{\tmpbox}{\stretchto{
  \scaleto{
    \scalerel*[\widthof{\ensuremath{#1}}]{\kern-.6pt\bigwedge\kern-.6pt}
    {\rule[-\textheight/2]{1ex}{\textheight}}
  }{\textheight}
}{0.5ex}}
\stackon[1pt]{#1}{\tmpbox}
}

\newtheorem{theorem}{Theorem}
\numberwithin{theorem}{section}

\newtheorem{lemma}[theorem]{Lemma}

\newtheorem{corollary}[theorem]{Corollary}

\numberwithin{equation}{section}

\newcommand{\BMO}{{\mathit {BMO}}}
\newcommand{\Lip}{\mathrm{Lip}}
\newcommand{\loc}{{\mathrm{loc}}}
\newcommand{\comp}{{\mathrm{c}}}
\newcommand{\supp}{{\mathrm{supp}}}
\newcommand{\Rd}{{{\mathbb R}^d}}
\newcommand{\re}{{\mathbb R}}
\newcommand{\sph}{{\mathbb S}}
\newcommand{\R}{{\mathrm R}}
\newcommand{\Ric}{{\mathrm{Ric}}}

\newcommand{\N}{{\mathbb N}}
\newcommand{\e}{{\varepsilon}}
\newcommand{\g}{{\rm g}}

\newcommand{\I}{\mathrm{I}}

\newcommand{\hf}{\frac 12}
\newcommand{\thf}{\tfrac 12}

\newcommand{\la}{\langle}
\newcommand{\ra}{\rangle}
\newcommand{\dist}{\mathrm{dist}}

\newcommand{\tpsi}{{\tilde\psi}}
\newcommand{\tp}{{\tilde p}}
\newcommand{\tP}{{\tilde P}}
\newcommand{\tB}{{\tilde B}}
\newcommand{\tW}{{\tilde W}}
\newcommand{\tE}{{\tilde E}}
\newcommand{\tK}{{\tilde K}}
\newcommand{\ta}{{\tilde a}}

\newcommand{\tq}{{\tilde q}}
\newcommand{\tr}{{\tilde r}}
\newcommand{\tkhf}{\bigl(t^{\hf}\ts 2^{-\frac k2}\bigr)}
\newcommand{\tkhfinv}{\bigl(t^{-\hf}\ts 2^{\frac k2}\bigr)}
\newcommand{\parsl}{\slashed{\partial}}
\newcommand{\barx}{\bar x}
\newcommand{\barxi}{\bar \xi}

\newcommand{\one}{{\mathds 1}}

\newcommand{\hatf}{{\hat{f\,}\!}}
\newcommand{\ts}{\hspace{.06em}}

\newcommand{\bfr}{\mathbf{r}}

\begin{document}
 
\title[The wave equation on manifolds of bounded curvature]{Dispersive estimates for the wave equation on \\Riemannian manifolds of bounded curvature}
\author[Y. Chen]{Yuanlong Chen}
\email{ylchen88@uw.edu}
\author[H. Smith]{Hart F. Smith}
\email{hfsmith@uw.edu}
\address{Department of Mathematics, University of Washington, Seattle, WA 98195-4350, USA}

\thanks{This material is based upon work supported by 
the National Science Foundation under Grant DMS-1500098}

  

\begin{abstract}
We establish space-time dispersive estimates for solutions to the wave equation on compact Riemannian manifolds with bounded sectional curvature, with the same exponents as for $C^\infty$ metrics. The estimates are for bounded time intervals, so by finite propagation velocity the results apply also on non-compact manifolds under appropriate uniform conditions. We assume a priori that in local coordinates the metric tensor components satisfy ${\rm g}_{ij}\in W^{1,p}$ for some $p>d$, which ensures that the curvature tensor is well defined in the weak sense, but this can be relaxed to any assumption that suffices for the local harmonic coordinate calculations in the paper.
\end{abstract}

\maketitle

\section{Introduction}
We assume throughout this paper that $(M,\g)$ is a $d$-dimensional Riemannian manifold of $C^1$ structure with the following property: there exists $r_0>0$, $C_0<\infty$, and $p\in(d,\infty]$, such that for each $z\in M$ there is a $C^1$ coordinate chart $\Phi_z:B_{r_0}\rightarrow M$ with $\Phi_z(0)=z$, in which the induced metric $\g_{ij}$ on $B_{r_0}\subset \Rd$ satisfies
$$
\g_{ij}(0)=\delta_{ij},\qquad \sup_{ij}\|\g_{ij}\|_{W^{1,p}}\le C_0.
$$
As shown in \cite[Chapter 3 \S 9]{Tay} or Section \ref{sec:reductions} of this paper, the Riemannian curvature tensor components $\R_{ijkl}$ are then well defined as distributions in $W^{-1,p}(B_{r_0})$. We make the assumption that the $\R_{ijkl}$ are measurable functions, and that for some $C_0$ uniform over the coordinate charts,
$$
\sup_{ijkl}\,\|\R_{ijkl}\|_{L^\infty(B_{r_0})}\le C_0.
$$
In Theorem \ref{spectral} we show that the Sobolev spaces $H^s(M)$ for $-2\le s\le 2$ defined using local harmonic coordinates are equivalent to those defined using fractional powers of $-\Delta_\g$ via the spectral calculus. For $-1\le s\le 2$ the following Cauchy problem for the wave equation on $(M,\g)$ can then be solved using the spectral decomposition for $\Delta_\g$ and Duhamel's formula,
\begin{equation}\label{eqn:wave}
\begin{aligned}
\big(\partial_t^2 - \Delta_{\rm{g}} \big) u(t,x) &= F(t,x) \in L^1([-T,T]; H^{-2}(M)),\\
u(0,x) &= f(x) \in H^s(M),\\
\partial_t u(0,x) &= g(x) \in H^{s-1}(M).
\end{aligned}
\end{equation}
In this paper we prove two types of dispersive estimates on the solution $u$, under the above assumptions on $(M,\g)$. Recall that a triple $(s,q,r)$ with $2\le q,r\le\infty$ is said to be \textit{admissible} for the wave equation if
$$
\frac 1q + \frac dr = \frac d2 - s, \qquad\quad
\frac 1q \le \frac{d-1}2 \Bigl( \frac 12 - \frac 1r \Bigr).
$$

\begin{theorem}[\textbf{Strichartz estimates}]\label{thm:strichartz}
If $(s,q,r)$ and $(1-s,\tilde q,\tilde r)$ are admissible, and $r,\tilde r<\infty$,
then for a positive $T$ depending only on $(M,\g)$, solutions to \eqref{eqn:wave} defined using the spectral decomposition of $\Delta_\g$ satisfy
\begin{multline*}
\|u\|_{L^q([-T,T]; L^r(M))}+\|u\|_{L^\infty([-T,T]; H^s(M))}+\|\partial_t u\|_{L^\infty([-T,T]; H^{s-1}(M))} \\
\le C \Bigl(\, \|f\|_{H^s(M)} + \|g\|_{H^{s-1}(M)} + \|F\|_{L^{\tq'}([-T,T];L^{\tr'}(M))}\,\Bigr).
\end{multline*}
\end{theorem}
Note that under these assumptions $0\le s\le 1$, and since $\tilde q\ge 2$ we see that $H^{\frac 32}(M)\subset L^{\tr}(M)$, hence $F\in L^1\bigl([-T,T];H^{-\frac 32}(M)\bigr)$.

The next estimate is due in the smooth case to Mockenhaupt-Seeger-Sogge \cite{MSS}. Here we consider only the critical exponent $q_d$, but similar results with $s_d\le s\le 2$ hold by Sobolev embedding.

\begin{theorem}[\textbf{Squarefunction estimate}]\label{thm:sqfn}
Let $q_d=\frac{2(d+1)}{d-1}$, and $s_d=\frac 1{q_d}$. Then for a positive $T$ depending only on $(M,\g)$, solutions to \eqref{eqn:wave} satisfy
\begin{equation*}
\|u\|_{L^{q_d}(M;L^2([-T,T]))} \le 
C \bigl(\|f\|_{H^{s_d}(M)}+\|g\|_{H^{s_d-1}(M)}+\|F\|_{L^1([-T,T];H^{s_d-1}(M))}\bigr).
\end{equation*}
\end{theorem}
A straightforward consequence of the squarefunction estimate are the following $L^2\rightarrow L^q$ bounds for unit-width spectral projection operators, which were originally established for smooth metrics by Sogge \cite{Sog1}.
\begin{corollary}\label{cor:specproj}
Suppose that $\lambda\ge 0$, and let $\Pi_{[\lambda,\lambda+1]}$ denote the $L^2(M)$ projection onto the span of eigenfunctions $\{\phi_j\}$ such that $-\Delta_\g\phi_j=\lambda_j^2\phi_j$ with $\lambda_j\in[\lambda,\lambda+1]$. Then for some $C$ depending only on $(M,\g)$,
$$
\|\Pi_{[\lambda,\lambda+1]}f\|_{L^q(M)}\le 
C\,\lambda^{d\bigl(\frac 12-\frac 1q\bigr)-\frac 12}\|f\|_{L^2(M)},\qquad q_d\le q\le\infty.
$$
\end{corollary}
Corollary \ref{cor:specproj} is proven for $q=q_d$ from Theorem \ref{thm:sqfn}, and for $q>q_d$ it follows by Sobolev embedding. See \cite{Sm2} for details. It is shown there that the $q=\infty$ case, which is related to the spectral counting remainder estimates of Avakumovi\'c-Levitan-H\"ormander, holds more generally on compact manifolds with metrics $\g$ of Lipschitz regularity.

The first version of Strichartz estimates was obtained globally on $\mathbb{R}^{d+1}$ by Strichartz in \cite{Str0}, \cite{Str1}, for $s = \frac 12$ and $q=r=\frac{2(d+1)}{d-1}$. The results were subsequently extended to other values of the exponents, and to the setting of smooth Riemannian manifolds using a Fourier integral representation of the fundamental solution. More details can be found in \cite{Sog0}, \cite{GV}, \cite{KT}, and \cite{LS}. Of particular interest are the critical indices, when equality holds in the second admissibility condition.

For a non-smooth metric $\g$, the standard constructions of the fundamental solution do not work. However, the second author used paradifferential techniques and wave packet parametrices in \cite{Sm0} to prove homogeneous Strichartz estimates in dimensions $d=2,3$ under the condition that the metric $\g$ is $C^{1,1}$. 
For all dimensions this is the minimal regularity condition on $\g$ in the context of H\"older spaces that implies the Strichartz estimates. Indeed, Smith and Sogge in \cite{SS} produced explicit examples of $C^{1,\alpha}$ metrics for which the homogeneous Strichartz estimates fail, for each $0 < \alpha < 1$.

The key idea in handling non-smooth metrics is to introduce a para-differential approximation $P$ to $\sqrt{-\Delta_{\g}}$, in that $P^2+\Delta_g$ behaves as a first order operator on a suitable range of Sobolev spaces. By energy estimates it then suffices to establish the bounds of Theorems \ref{thm:strichartz} and \ref{thm:sqfn} when $\Delta_\g$ is replaced by $-P^2$ in \eqref{eqn:wave}. The operator $P$ has symbol of class $S^1_{1,\hf}$ and is obtained by mollifying the coefficients of $\g$ over scale $2^{-\frac k2}$ when acting on functions at frequency scale $2^k$.

One then seeks a construction of the evolution operator $e^{-itP}$ for which the desired dispersive bounds can be proven. In \cite{Sm0}, an approximation $E(t)$ to $e^{-itP}$ was obtained by working in a frame of dyadic-parabolic wave packets (curvelets). A key property of such wave packets is that the action of $e^{-itP}$ on each element of the frame is well approximated by rigid translation of the packet along the Hamiltonian flow of $P$, and $E(t)$ was defined as this rigid motion. 
This operator $E(t)$ failed to satisfy the unitary group property $E(t)E(s)^*=E(t-s)$, however, which is a crucial requirement for the established proofs of dispersive bounds such as in \cite{KT}.
This limited the results of \cite{Sm0} to low dimensions. The Strichartz estimates of Theorem \ref{thm:strichartz} for $C^{1,1}$ metrics and general dimensions were subsequently established by Tataru in \cite{Tat0}, \cite{Tat1}, \cite{Tat2},  where space-time bounds on the FBI transform were used. The paper \cite{Sm1} of Smith used a modified FBI transform to translate the problem to phase-space, and $e^{-itP}$ was approximated on the transform side by the Hamiltonian flow map. This forms a unitary group, and the estimates in Theorems \ref{thm:strichartz} and \ref{thm:sqfn} (with $F=0$ in Theorem \ref{thm:strichartz}) were established for $C^{1,1}$ metrics, in all dimensions.

For metrics of bounded curvature the paradifferential construction of the self-adjoint operator $P$ goes through as above, provided one works in harmonic coordinates on $(M,\g)$. In such coordinates the metric $\g$ has second derivatives belonging to $\BMO$, which is sufficient to show that $P^2+\Delta_\g$ maps $H^s\rightarrow H^{s-1}$ for a range of $s$. The wave packet methods fail to give a useful construction of $e^{-itP}$, however, since the error estimates for the rigid translation or Hamiltonian flow approximations depend explicitly on pointwise bounds on $\partial_x^2\g^{ij}(x)$. On the other hand, by the Jacobi variation formula $L^\infty$ bounds on the Riemannian curvature tensor imply that the geodesic and Hamiltonian flows are bilipschitz. A consequence is that the solution to the eikonal equation in any local harmonic coordinate system has bounded second derivatives, the same regularity as for $C^{1,1}$ metrics. 

This naturally leads us in this paper to imitate the Lax parametrix construction for $e^{-itP}$. It turns out that solving the transport equations for the amplitude produces no further improvement beyond setting the amplitude to be identically one, as all terms in the expansion of the amplitude would be symbols of order zero, due to the fact that the symbol of $P$ is of class $S^1_{1,\hf}$. On the other hand, to have a unitary group we need work with the exact operator $e^{-itP}$. We achieve this by producing $e^{-itP}$ exactly as an iterative expansion of the Lax approximation, which we show converges uniformly on finite time intervals in the $H^s$ operator norm for every $s\in\re$.

To prove the dispersive estimates of Theorems \ref{thm:strichartz} and \ref{thm:sqfn} we establish bounds on the integral kernel of $e^{-itP}$ localized dyadically in frequency. These bounds capture the pointwise decay of the fundamental solution away from the light cone, and are of the exact same form as for smooth metrics. An advantage of this proof is that we can obtain the inhomogeneous estimates stated in Theorem \ref{thm:strichartz}.
We establish the kernel bounds using a version of the wave packet frame of \cite{Sm0} rescaled by time $t$. This method is well adapted to handle the multiple products arising in the iterative expression for $e^{-itP}$, since the bounds can be phrased in terms of operator bounds in certain weighted norm spaces.

The proof of Theorems \ref{thm:strichartz} and \ref{thm:sqfn} is composed of multiple distinct steps, and we divide it up into sections as follows. A more detailed summary of each section is included at its beginning.

In Section \ref{sec:reductions}, we present the details of harmonic coordinates on $(M,\g)$ and the regularity results for $\g$ in such coordinates. The procedure is similar to that in Taylor \cite{Tay}, Chapter 3 \S 9. We then reduce matters to working with a compact perturbation of the Euclidean metric on $\Rd$. We introduce the paradifferential operator approximation $P$, and equate the estimates of Theorems \ref{thm:strichartz} and \ref{thm:sqfn} to Lebesgue space mapping properties for $e^{-itP}$.

In Section \ref{sec:Hamflow}, we use the Jacobi variation formula to study the regularity of the geodesic flow for the metric $\g_k$ that is obtained by mollifying $\g$ at scale $2^{-\frac k2}$. The estimates on the derivatives of the geodesic flow are exactly those obtained in the case $\g\in C^{1,1}$. 

In Section \ref{sec:phasefunction}, we use the results derived in Section 3 and a dilation argument to prove symbol type estimates on the solution $\varphi_k(t,x,\eta)$ of the eikonal equation for $\g_k$. A key result is obtaining better estimates for small $t$, which is crucial to proving the dispersive estimates on the kernel of $e^{-itP}$ when $|t|\ll 1$.

In Section \ref{sec:parametrix}, we introduce an approximation $W(t)$ to $e^{-itP}$, which is a sum over $k$ of terms
$$
\bigl(W_k(t)f\bigr)(x)=\frac 1{(2\pi)^d}\int e^{i\varphi_k(t,x,\eta)}\,\psi_k(\eta)\hatf(\eta)\,d\eta,
$$
where $\psi_k$ is a Littlewood-Paley partition of unity. We show that
$$
\big(\partial_t + iP_k\big)\bigl(W_k(t)f\bigr) = B_k(t)f
$$
where $B_k(t)$ is an oscillatory integral operator with phase $\varphi_k$, and symbol $b_k(t,x,\eta)$ of order $0$
that satisfies derivative bounds similar to those for $\varphi_k$.

Section \ref{sec:energyflow} is concerned with energy flow properties of iterated compositions of $W(t)$ and $B(t)$, which arise in the expansion of $e^{-itP}$. In particular, we show that multiple compositions preserve dyadic localization in frequency up to smoothing errors. Thus, in proving dispersive estimates for $e^{-itP}$ we need only handle the composition of terms $W_k$ and $B_k$ all of which are localized at the same dyadic scale. We also prove ``sideways'' energy estimates that arise in the proof of Theorem \ref{thm:sqfn}.

In Section \ref{sec:wavepackets} we prove that, for small $t$, the kernel $K_k(t,x,y)$ of $e^{-itP}\psi_k(D)$ satisfies, modulo a smoothing operator, the same bounds as for smooth metrics:
$$
|K_k(t,x,y)|\le C_N\,2^{kd}\bigl(1+2^k|t|\bigr)^{-\frac{d-1}2}\bigl(1+2^k\dist(x,S_t(y))\bigr)^{-N},
$$
where $S_t(y)$ is the geodesic sphere centered at $y$ and $\dist(\cdot,\cdot)$ the geodesic distance for $\g_k$.
Together with standard arguments these estimates  yield Theorems \ref{thm:strichartz} and \ref{thm:sqfn}. The proof of this estimate proceeds, for a given value of $t$, by representing $e^{-itP}\psi_k(D)$ in a wave packet frame that is obtained by scaling by $|t|$ the dyadic-parabolic frame from \cite{Sm0}. The kernel estimates follow by showing that the operator $e^{-itP}\psi_k(D)$ maps a frame element at time $0$ to a similar function translated along the Hamiltonian flow through its center. This fact is deduced from showing the same result for the terms $W_k(s)$ and $B_k(s)$ for $0\le s\le t$ that arise in the iterative formula for $e^{-itP}$.

\section{Preliminaries and Reduction to the Model Operator}\label{sec:reductions}
In this section we establish regularity estimates for the metric $\g$ in local harmonic coordinate charts.
We then consider Sobolev spaces on $M$, and define the wave group for $\sqrt{-\Delta_\g}$ using the orthonormal basis for $L^2(M)$ consisting of eigenfunctions of $\Delta_g$. We conclude by reducing the proof of Theorem \ref{thm:strichartz} to estimates for the evolution group $e^{-itP}$ of the self-adjoint first order pseudodifferential operator $P$ on $\Rd$, where $P$ is an extension to $\Rd$ of a paradifferential approximation to $\sqrt{-\Delta_\g}$ in one of a finite cover by $M$ of local harmonic coordinate charts.

\subsection{Harmonic coordinates on $(M,\g)$}
We start with the assumption that $(M,\g)$ is a Riemannian manifold of $C^1$ structure with the following condition: there exists $r_0>0$, $C_0<\infty$, and $p\in(d,\infty]$, and for each $z\in M$ a coordinate chart $\Phi_z:B_{r_0}\rightarrow M$ with $\Phi_z(0)=z$, so that the induced metric $\g$ on $B_{r_0}\subset \Rd$ satisfies
$$
\g_{ij}(0)=\delta_{ij},\qquad \sup_{ij}\|\g_{ij}\|_{W^{1,p}}\le C_0.
$$
Since $W^{1,p}$ functions are of H\"older regularity $1-\frac dp>0$, by shrinking $r_0$ if needed we may additionally assume that, given $c_0>0$ to be determined,
$$
\sup_{x\in B_{r_0}}|\g_{ij}(x)-\delta_{ij}|\le c_0.
$$
Following Taylor \cite{Tay}, Chapter 3 \S 9, in particular \cite[ch.\ 3, Prop.\ 9.1]{Tay} and the comments following \cite[ch.\ 3, (9.39)]{Tay}, after replacing $r_0$ by $\rho_0=\rho_0(d,p,C_0,c_0)$, we may assume that the induced coordinate functions, $f_z^i:\Phi_z(B_{\rho_0})\rightarrow \re$,  are harmonic functions with respect to the Laplace-Beltrami operator of $\g$, and that overlapping harmonic coordinate charts have transition functions of regularity $W^{2,p}$ on their overlaps. The harmonic coordinates are related to the coordinate functions of $\Phi_z$ by a $W^{2,p}$ change of coordinates over $B_{r_0}$, and it follows that the original coordinates were necessarily of regularity $W^{2,p}\subset C^{1,1-\frac dp}$ on their overlaps. Consequently, $M$ is a manifold with $W^{2,p}$ structure. This is consistent with the fact that a metric $\g$ maintains its $W^{1,p}$ regularity under a $W^{2,p}$ change of coordinates, which can be seen by \eqref{prodest} below.

For every integer $m\ge 0$, there is a continuous linear extension operator of $W^{m,p}(B_{\rho_0})$ to $W^{m,p}(\Rd)$; see e.g. \cite[ch.\ VI \S 3 Thm.\ 5]{St}. We may thus apply \cite[ch.\ 2 Prop.\ 1.1]{Tay}, together with the inclusions 
$$
W^{1,p}(\Rd)\subset L^\infty(\Rd),\qquad H^1(\R^d)=W^{1,2}(\Rd)\subset L^{\frac{2p}{p-2}}(\Rd),
$$
to see that the following hold, both on $\Rd$ and $B_{\rho_0}$,
\begin{equation}\label{prodest}
\|fg\|_{W^{1,p}}\le C\,\|f\|_{W^{1,p}}\|g\|_{W^{1,p}},\qquad
\|fg\|_{H^1}\le C\,\|f\|_{W^{1,p}}\|g\|_{H^1}.
\end{equation}

The Riemannian curvature tensor $\R$ for $\g$ is given in coordinates by
$$
\R_{ijkl}=
\frac12\left[\frac{\partial^2\g_{ik}}{\partial x_j\partial x_\ell}+
\frac{\partial^2 \g_{j\ell}}{\partial x_i\partial x_k}-
\frac{\partial^2 \g_{i\ell}}{\partial x_j\partial x_k}-
\frac{\partial^2 \g_{jk}}{\partial x_i\partial x_\ell}\right]
+Q(\g,\partial\g),
$$
where $Q(\g,\partial\g)$ is a quadratic form in first order derivatives of $\g_{ij}$, with coefficients given by a combination of coefficients of $\g$, hence $Q(\g,\partial\g)\in L^{\frac p2}$ when $\g\in W^{1,p}$ with $p>d$. Then $\R$ is defined as a distribution, and our key assumption is that $\R_{ijkl}$ is a bounded measurable function, such that uniformly in the local coordinates $F_z$,
$$
\sup_{ijkl}\,\|\R_{ijkl}\|_{L^\infty(B_{\rho_0})}\le C_0.
$$
This is implied by assuming that $\R$ is a measurable function, together with the geometric condition that for all continuous vector fields $v_j$,
$$
\|\la R(v_1,v_2)v_3,v_4\ra\|_{L^\infty(M)}\le C_0 \quad\text{if}\quad\|\g(v_j)\|_{L^\infty(M)}\le 1.
$$

In harmonic coordinates, the Ricci tensor $\Ric$ can be written, see for example \cite{DK}, in the form
$$
\Ric_{ij}=\sum_{mn}\partial_{x_m}\bigl(\g^{mn}\partial_{x_n} \g_{ij}\bigr)+Q(\g,\partial \g).
$$
Since $\Ric_{ij}\in L^\infty(B_{\rho_0})$, following \cite[ch.\ 3 \S 10]{Tay} we conclude $\g_{ij}\in W^{2,q}(B_\rho)$ for all $\rho<\rho_0$ and all $q<\infty$, hence $\g_{ij}\in \Lip(B_{.9\rho_0})$.

Take $\phi\in C^\infty_c(B_{.8\rho_0})$ with $\phi=1$ on $B_{.7\rho_0}$, and $\chi\in C^\infty_c(B_{.9\rho_0})$ with $\chi=1$ on $B_{.8\rho_0}$, and assume $\phi$ and $\chi$ take values in $[0,1]$.

We form a Riemannian metric $\tilde\g_{ij}=\phi\,\g_{ij}+\bigl(1-\phi\bigr)\delta_{ij}$ on $\Rd$, and uniformly elliptic coefficients $a^{ij}=\chi\,\g^{ij}+\bigl(1-\chi\bigr)\delta^{ij}$ on $\Rd$. Note that
$Q(\g,\partial\g)\in L^\infty(B_{.9\rho_0})$ since $\g\in\Lip(B_{.9\rho_0})$. Then the following holds globally on $\Rd$,
$$
\sum_{m,n=1}^d\,\partial_{x_m}\bigl(a^{mn}\partial_{x_n}\tilde\g_{ij}\bigr)
\in L^\infty_{\comp}.
$$
Since the $a^{mn}$ are globally Lipschitz, from \cite[ch.\ 3, Prop.\ 10.3]{Tay} we conclude that 
$\partial_x^2\tilde \g_{ij}\in \BMO_{\comp}(\Rd)$; more precisely $\partial_x^2\tilde \g_{ij}$ belongs to $\BMO(\Rd)$ and is supported in $B_{.8\rho_0}$.

Note that the Riemannian curvature tensor $\tilde\R$ of $\tilde\g$ belongs to $L^\infty_\comp(\Rd)$, where we use that $\tilde\g$ is Lipschitz, so $\tilde\R=\phi\R$ modulo products of $\g$ and $\partial_x\g$ and functions in $C_c^\infty(B_{.9\rho_0})$.
After shrinking $\rho_0$ by a factor of 2, we conclude

\begin{lemma}\label{lem:harmcoords}
Given $c_0>0$, there exists $\rho_0>0$ and $C_0<\infty$ so that
for each $z\in M$ there is a harmonic coordinate chart $\Phi_z:B_{\rho_0}\rightarrow M$, with 
$\Phi_z(0)=z$, such that the induced metric on $B_{\rho_0}$ agrees with the restriction of a metric 
$\g$ defined on $\Rd$ that satisfies $\g_{ij}=\delta_{ij}\;\;\text{if}\;\;|x|>2\rho_0$, and
$$
\|\partial_x^2\g_{ij}\|_{\BMO}+\|\g_{ij}\|_{\Lip}+\|\R_{ijkl}\|_{L^\infty}\le C_0,\qquad 
\|\g_{ij}-\delta_{ij}\|_{L^\infty}\le c_0.
$$
In particular, $\g_{ij}-\delta_{ij}$ belongs to $W^{2,q}_\comp(\Rd)$ for all $q<\infty$.
\end{lemma}

We now cover $M$ by a finite collection of harmonic coordinate charts 
$\Phi_j\equiv \Phi_{z_j}:B_{\rho_0}\rightarrow M$, each of which satisfies the conditions of 
Lemma \ref{lem:harmcoords}, such that there is a partition of unity $\chi_j$ on $M$ with 
$\supp(\chi_j)\subset \Phi_j(B_{\rho_0/3})$ and $\chi_j\circ \Phi_i\in W^{2,p}(B_{\rho_0})$ for each $i,j$. In particular, $\chi_j\circ \Phi_j\in W^{2,p}_\comp(B_{\rho_0})$.

By \eqref{prodest}, multiplication by $\chi_j\circ\Phi_j$ maps $H^s_{\loc}(B_{\rho_0})$ into $H^s_\comp(B_{\rho_0})$ for $s=0,1,2$. By interpolation this holds for $0\le s\le 2$. We may then introduce Sobolev spaces $H^s(M)\subset L^2(M)$, for $0\le s\le 2$, by the condition
\begin{equation}\label{Hsdef}
\begin{split}
f\in H^s(M)\,&\Leftrightarrow\;f\circ \Phi_j\in H^s_\loc(B_{\rho_0})\;\;\forall j,\\
\|f\|_{H^s(M)}\,&=\,\sum_j\|(\chi_j f)\circ \Phi_j\|_{H^s(\Rd)}.
\end{split}
\end{equation}
If $g\in H^s_\comp(B_{\rho_0})$ then
$$
\|(\chi_j\cdot g\circ\Phi_i^{-1})\circ\Phi_j\|_{H^s}\le C\|g\|_{H^s},
$$
for $C$ depending on the support of $g$.
This holds for $s=0,1$ since $\Phi_i^{-1}\circ \Phi_j$ is a $C^1$ diffeomorphism. It holds for $s=2$ since $D(\Phi_i^{-1}\circ \Phi_j)\in W^{1,p}$ is a multiplier on $H^1$ by \eqref{prodest}. It then holds by interpolation for $0\le s\le 2$. Consequently, there are natural continuous inclusions $H^s_\comp(B_{\rho_0})\rightarrow H^s(M)$ for $0\le s\le 2$ given by $g\rightarrow g\circ \Phi_j^{-1}$, and one may identify $H^s(M)$ with a closed subspace of the finite direct sum over $j$ of $H^s(B_{\rho_0})$.

An element of $(H^s)^*$ thus induces an element of $H^{-s}_\loc(B_{\rho_0})$, and if we identify $H^{-s}(M)$ with $(H^s)^*$ for $0\le s\le 2$, then the condition \eqref{Hsdef} holds for $-2\le s\le 2$, with approximate equality for the norm.

We observe here the following regularity property for $\Delta_g$ in harmonic coordinates, which follows, for example, from \cite[Thm.\ 8.9]{GT}. Suppose that $u\in H^1(B_{\rho_0})$ is a weak solution to $\Delta_\g u=f$, where $f\in L^2(B_{\rho_0})$. Then $u\in H^2(B_\rho)$ for all $\rho<\rho_0$, and
\begin{equation}\label{ellipticreg}
\|u\|_{H^2(B_\rho)}\le C_\rho\,\bigl(\|u\|_{H^1(B_{\rho_0})}+\|f\|_{L^2(B_{\rho_0})}\bigr).
\end{equation}

The Sobolev spaces for $|s|\le 2$ can also be characterized using the spectral decomposition of $\Delta_\g$ on $L^2(M)$. Consider the quadratic form on $H^1(M)$ given by
$$
Q(u,v)=-\int \overline {u}\,(\Delta_\g v)\,dm_\g=\int\g(d\overline u,dv)\,dm_\g.
$$
Then $Q$ is symmetric, nonnegative, and coercive. By the Rellich compactness theorem there is a complete orthonormal basis $\{v_j\}$ of $L^2(M,dm_\g)$ that diagonalizes $Q$, in that for $f,g\in H^1(M)$
$$
Q(f,g)=\sum_j\lambda_j^2\,\overline{c_j(f)}\,c_j(g),\qquad c_j(f)=\int_M \overline{v_j}\,f\,dm_\g,
$$
and $0=\lambda_0\le \lambda_1\le \cdots$ is a sequence of real numbers converging to $\infty$. The $v_j$ are weak solutions in $H^1(M)$ to $-\Delta_\g v_j=\lambda_j^2\,v_j$, hence \eqref{ellipticreg} gives $\|v_j\|_{H^2(M)}\le C\,\lambda_j^2$. It follows that $c_j(f)$ can be defined for $f\in H^s(M)$ when $-2\le s\le 0$ as the action of $f$ on $\overline{v_j}$.

The operator $(1-\Delta_\g)$ is equivalent to multiplication by $(1+\lambda_j^2)$ in the basis $\{v_j\}$, and the following theorem then gives a more natural definition of $H^s(M)$.

\begin{theorem}\label{spectral}
For $-2\le s\le 2$, the mapping $f\rightarrow \{c_j(f)\}_{j=0}^\infty$ defines a homeomorphism of $H^s(M)$ with the space $\ell^2\bigl(\N,(1+\lambda_j^2)^{s}\bigr)$. In particular, uniformly over $-2\le s\le 2$, we have
$$
\|f\|_{H^s(M)}^2\approx\sum_{j=0}^\infty \bigl(1+\lambda_j^2\ts\bigr)^s|c_j(f)|^2\,,\qquad c_j(f)=\int_M f\,\overline{v_j}\,dm_\g,
$$
and $\sum_{j=0}^\infty c_j(f)\,v_j$ converges to $f$ in the topology of $H^s(M)$.
\end{theorem}
\begin{proof} The theorem holds for $s=0$ by orthonormality, and for $s=1$ since
$\|f\|_{H^1}^2\approx \|f\|_{L^2}^2+Q(f,f)$.
For $s=2$, we note that the partial sums
$$
\sum_{j=0}^N c_j\bigl((1-\Delta_\g) f\bigr)\,v_j
=\sum_{j=0}^N \bigl(1+\lambda_j^2\bigr)c_j(f)\,v_j
=(1-\Delta_\g)\sum_{j=0}^N c_j(f)\,v_j
$$ 
converge in $L^2(M)$ to $(1-\Delta_\g)f$ if $f\in H^2(M)$. It follows by elliptic regularity that
$\sum_{j}c_j(f)\,v_j$ converges in $H^2(M)$ to $f$. Surjectivity onto $\ell^2\bigl(\N,(1+\lambda_j^2)^2\bigr)$ follows similarly.
The theorem follows for $0\le s\le 2$ by interpolation, and for $-2\le s\le 0$ by duality.
\end{proof}
We note that the proof also shows that $-\Delta_\g$ conjugates to multiplication by $\{\lambda_j^2\}$ in the basis $\{v_j\}$, as a map from $H^s(M)\rightarrow H^{s-2}(M)$, provided $0\le s\le 2$.


\subsection{The wave equation on $(M,\g)$}
For data $(f,g)\in L^2(M)\oplus H^{-1}(M)$ and $F\in L^1_t\bigl[-T,T];H^{-2}(M)\bigr)$ we define the solution of the Cauchy problem \eqref{eqn:wave} to be
\begin{multline}\label{eqn:specrep}
u(t,x)=\sum_{j=0}^\infty 
\biggl(\cos(t\lambda_j)\,c_j(f)+\lambda_j^{-1}\sin(t\lambda_j)\,c_j(g)
\\
+\int_0^t\lambda_j^{-1}\sin((t-s)\lambda_j)\,c_j(F(s,\cdot))
\biggr)v_j(x)
\end{multline}
where we set $0^{-1}\sin(0\ts t)=t$. We show here that Theorem \ref{thm:strichartz} can be deduced from the following assertion:

\noindent\textit{Assume that $u\in C^0(H^s(M))\cap C^1(H^{s-1}(M))$, and that $u$ is given by \eqref{eqn:specrep}. Then for $s,q,\tq,r,\tr$ as in Theorem \ref{thm:strichartz}, the following estimate holds,}
\begin{multline*}
\|u\|_{L_t^q([-T,T]; L^r(M))}
\le C \Bigl(\,\|u\|_{L_t^\infty([-T,T]; H^s(M))}
+\|\partial_t u\|_{L_t^\infty([-T,T]; H^{s-1}(M))}\\
+\|F\|_{L_t^{\tq'}([-T,T];L^{\tr'}(M))}\,\Bigr).
\end{multline*}
To see that this result implies Theorem \ref{thm:strichartz}, consider first the case $F=0$. Then by the spectral representation of $u$ we have
$$
\|u\|_{L_t^\infty([-T,T]; H^s(M))}
+\|\partial_t u\|_{L_t^\infty([-T,T]; H^{s-1}(M))}\approx \|f\|_{H^s(M)}+\|g\|_{H^{s-1}(M)},
$$
and Theorem \ref{thm:strichartz} follows from the assertion. We apply this to the triple $(1-s,\tq,\tr)$ and use duality to see that, when $f=g=0$,
$$
\|u\|_{L_t^\infty([-T,T]; H^s(M))}
+\|\partial_t u\|_{L_t^\infty([-T,T]; H^{s-1}(M))}\le C\,\|F\|_{L_t^{\tq'}([-T,T];L^{\tr'}(M))}.
$$
The continuity of $u$ and $\partial_t u$ follows by translation continuity,
and Theorem \ref{thm:strichartz} then follows from the assertion for the case $F\ne 0$.

As a result we may assume that
$$
u\in C^0(H^s(M))\cap C^1(H^{s-1}(M))\cap C^2(H^{s-2}(M)),
$$
and in particular, $\partial_t^2 u=\Delta_\g u$ in the weak sense on $B_{\rho_0}$ in each of the local harmonic coordinate charts $\Phi_j$.

If the data $(f,g,F)$ is localized in $\Phi_j(B_{\rho_0/3})$, then finite propagation velocity shows that $u(t)$ is supported in $\Phi_j(B_{2\rho_0/3})$ if $|t|\le\rho_0/6$, where we use $W^{2,p}$ regularity of $\g$ for all $p<\infty$, and closeness of $\g_{ij}$ to $\delta_{ij}$ for $c_0$ small.

Using the partition of unity $\chi_j$, we can thus reduce the proof of Theorem \ref{thm:strichartz} to the case that the Cauchy data is supported in $\Phi_j(B_{\rho_0/3})$, and thus work on $\Rd$ with a metric satisfying the conditions of Lemma \ref{lem:harmcoords}.
After rescaling space and time by a factor $R\ge 1$, where
$R^{-1}C_0\le c_d$, we can reduce Theorems \ref{thm:strichartz} and \ref{thm:sqfn} with $T= R^{-1}\rho_0/6$ to the following Theorem \ref{thm:strichartz'}. The constant $c_d$ will be fixed depending only on the dimension, and in particular will be small enough to rule out conjugate points for $|t|\le 1$.

\begin{theorem}\label{thm:strichartz'}
Assume $\g$ is a Riemannian metric on $\Rd$,
such that for a prescribed constant $c_d$ depending on the dimension $d$,
$$
\|\R_{ijkl}\|_{L^\infty}+\|\g_{ij}-\delta_{ij}\|_\Lip+\|\partial_x^2\g_{ij}\|_{BMO}\le c_d.
$$
Assume that $(s,q,r)$ and $(1-s,\tilde q,\tilde r)$ are admissible with $r,\tilde r<\infty$, and
let $u\in C^0\bigl([0,1];H^s(\Rd)\bigr)\cap C^1\bigl([0,1];H^{s-1}(\Rd)\bigr)$ be a weak solution to 
$$
(\partial_t^2-\Delta_\g)u=F,\quad u(0,\cdot)=f,\quad \partial_tu(0,\cdot)=g.
$$
Then there is a constant $C<\infty$ depending only on $d$, so that
\begin{multline*}
\|u\|_{L^q([0,1]; L^r(\Rd))}\le C \Bigl(\,\|u\|_{L^\infty([0,1]; H^s(\Rd))}+\|\partial_t u\|_{L^\infty([0,1]; H^{s-1}(\Rd))} \\
 +\|F\|_{L^{\tq'}([0,1];L^{\tr'}(\Rd))}\,\Bigr).
\end{multline*}
If $q_d=\frac {2(d+1)}{d-1}$ and $s=s_d=q_d^{-1}$, then
$$
\|u\|_{L^{q_d}(\Rd;L^2([0,1]))} \le 
C \bigl(\|f\|_{H^{s_d}(\Rd)}+\|g\|_{H^{s_d-1}(\Rd)}+\|F\|_{L^1([0,1];H^{s_d-1}(\Rd))}\bigr).
$$
\end{theorem}

\subsection{The model operator $P$}
We construct here the paradifferential approximation to $\sqrt{-\Delta_\g}$, where we will assume that $\g$ is a metric on $\Rd$ that satisfies the conditions of Theorem \ref{thm:strichartz'}.

We fix a family of dyadically supported functions $\beta_k(\xi)$ for $k\ge 0$, such that $\beta_k(\xi)=\beta_1(2^{1-k}\xi)$ if $k\ge 1$, and such that $\psi_k(\xi)=\beta_k(\xi)^2$ gives a Littlewood-Paley partition of unity. We will assume that
$$
\supp(\beta_1)\subset \{\tfrac 9{10}\le |\xi|\le \tfrac{20}9\}\,,\qquad \beta_0(\xi)^2+\sum_{k=1}^\infty\beta_k(\xi)^2=1.
$$
We introduce a family of metrics $\g_k(x)$ that are mollifications of $\g(x)$ on spatial scale $2^{-\frac k2}$.
Precisely, fix a radial function $\chi\in C_c^\infty(B_1)$, so that 
$$\int \chi(x)\,dx=1,\qquad\int x^\alpha \chi(x)\,dx=0\quad\text{if}\quad 1\le|\alpha|\le 3.
$$
For $k\ge 1$ define a smooth metric $\g_k$ on $\Rd$ by
$$
(\g_k)^{ij}(x)=2^{\frac {kd}2 }\int\chi\bigl(2^{\frac k 2}(x-y)\bigr)\,\g^{ij}(y)\,dy.
$$
From the conditions on $\g$ in Theorem \ref{thm:strichartz'} it follows that $\|\g_k-\I\|_{\Lip}\le c_d$. Also,
\begin{equation*}
\|\partial_x^\beta \g^{ij}_k\|_{L^\infty}\le C_\alpha
\begin{cases}
1+\log(k),& |\beta|=2,\\
2^{\frac k2(|\beta|-2)},& |\beta|\ge 3.
\end{cases}
\end{equation*}
The estimate for $|\beta|=2$ holds when $k=1$ since $\partial_x^2(\chi*\g)=
(\partial_x\chi)*(\partial_x\g)$. For $k\ge 2$ we use that $\chi(2^{\frac 12}\cdot)-\chi(\cdot)$ is an $H^1$-atom, and $\partial_x^2\g\in \BMO(\Rd)$.
The estimate for $|\beta|\ge 3$ follows by writing
$$
\partial_x^\beta(\g_k)^{ij}(x)=2^{\frac k2 (d+|\beta|-2)}\int(\partial_x^{\beta-2}\chi)\bigl(2^{\frac k 2}(x-y)\bigr)\,\partial_y^2\g^{ij}(y)\,dy,
$$ 
and using that $\partial_x^\theta\chi$ is an $H^1$-atom, with norm $C_\alpha$, when $|\theta|\ge 1$.

We also note here the following bounds:
\begin{equation}\label{gkdiffbounds}
\bigl\|\partial_x^\beta(\g_k-\g_{k-1})\bigr\|_{L^\infty}\le C_\beta\,2^{-k+\frac 12 |\beta|}\,.
\end{equation}
For this, write
$$
\chi(\xi)-\chi(2^{\hf}\xi)=|\xi|^2\rho(\xi),\qquad\rho\in\mathcal{S}(\Rd)\,,\;\;\rho(0)=0.
$$
Then, setting $\rho_k(\xi)=\rho(2^{-\frac k2}\xi)$, we have
$$
\g_k-\g_{k-1}=2^{-k}\widehat{\rho_k}*(\Delta \g).
$$
The bound \eqref{gkdiffbounds} then follows from \cite[IV.1.1.4]{St1} as above.

Many of the steps in subsequent estimates use only the weaker estimates that follow from 
the Lipschitz bounds on $\g$,
\begin{equation}\label{gkest'}
\|\partial_x^\beta \g^{ij}_k\|_{L^\infty}\le C_\alpha
\begin{cases}
1,& |\beta|\le 1,\\
2^{\frac k2(|\beta|-1)},& |\beta|\ge 2.
\end{cases}
\end{equation}

Define $\,p_k(x,\xi)=\Bigl(\sum_{i,j=1}^d\g^{ij}_k(x)\,\xi_i\,\xi_j\Bigr)^{\frac 12}$,  
so that $p_k(x,\xi)$ is homogeneous of degree 1 in $\xi$. Then by \eqref{gkest'} and the conditions of 
Theorem \ref{thm:strichartz'} 
\begin{equation}\label{eqn:pkest}
\begin{split}
&\bigl|p_k(x,\xi)-|\xi|\bigr|+|\partial_xp_k(x,\xi)|\le c_d\ts|\xi|,\\
&|\partial_\xi^\alpha\partial_x^\beta p_k(x,\xi)|\le C_{\alpha,\beta}2^{\frac k2\max(0,|\beta|-1)}\,|\xi|^{1-|\alpha|}.\rule{0pt}{15pt}
\end{split}
\end{equation}
Hence, $\partial_x^\beta p_k(x,\xi)\psi_k(\xi)\in S^1_{1,\hf}$, uniformly over $k\ge 1$, if $|\beta|\le 1$. Similarly, by \eqref{gkdiffbounds} we see that 
\begin{equation}\label{pkdiff}
(p_{k\pm 1}-p_k)\psi_k\in S^0_{1,\hf}\quad\text{uniformly over}\; k.
\end{equation}

Define
\begin{equation*}
P=\beta_0(D)^2+\frac 12\sum_{k=1}^\infty \beta_k(D)\bigl(p_k(x,D)+p_k(x,D)^*\bigr)\beta_k(D),
\end{equation*}
and let $p(x,\xi)$ be the symbol of $P$. Then $P$ is self-adjoint, and the $S_{1,\hf}$ pseudodifferential calculus shows that 
$$
p(x,\xi)-\sum_{k=1}^\infty p_k(x,\xi)\psi_k(\xi)\,\in\,S^0_{1,\hf}.
$$
In particular,
$$
\partial_x^\beta p\in S^{1}_{1,\hf}\;\;\text{for}\;\;|\beta|\le 1.
$$
We note for future use that the Garding inequality for $P$ follows easily. It can be verified by letting
$$
b(x,\xi)=\Bigl(\psi_0(\xi)+\sum_{k=1}^\infty p_k(x,\xi)\psi_k(\xi)\Bigr)^{\hf}.
$$
Then $b(x,D)^*b(x,D)-P\in \mathrm{Op}(S^0_{1,\hf})$, hence for $f\in H^{\hf}$, and some real $C_1$
\begin{equation}\label{garding}
\la Pf,f\ra\ge-C_1\,\|f\|_{L^2}^2.
\end{equation}

\begin{lemma}\label{lem:paraest} The following holds for $0\le s\le 2$,
\begin{equation*}
\|P^2 u+\Delta_\g u\|_{H^{s-1}(\Rd)}\le C\,\|u\|_{H^s(\Rd)}\,.
\end{equation*}
\end{lemma}
\begin{proof}
By \eqref{eqn:pkest}, we deduce that $\partial_x^\beta p_k(x,\xi)\beta_k(\xi)\in S^1_{1,\hf}$ for $|\beta|\le 1$, with uniform bounds over $k$. Furthermore, $\beta_k$ has disjoint support from $\beta_j$ if $|j-k|>1$. The composition calculus together with \eqref{pkdiff} thus show that
$$
P^2=\sum_{k=0}^\infty \Bigl(\sum_{i,j=1}^d \g^{ij}_k(x) D_i D_j\Bigr)\psi_k(D)+r(x,D)\,,\quad r(x,\xi)\in S^1_{1,\hf},
$$
and in particular $r(x,D):H^s\rightarrow H^{s-1}$ for all $s$. We next write
$$
-\Delta_\g=\sum_{i,j=1}^d \g^{ij}(x) D_i D_j+\det(\g)^{-\hf}\bigl(D_i\bigl(\,\det(\g)^\hf\,\g^{ij}\bigr)\bigr)D_j.
$$
By \eqref{prodest} we see that $\det(\g)^{-\hf}\bigl(D_i\bigl(\,\det(\g)^\hf\,\g^{ij}\bigr)\bigr)\in W^{1,p}$ is a multiplier on $H^s$ for $|s|\le 1$, so the second term maps $H^s\rightarrow H^{s-1}$ for $0\le s\le 2$.

We thus need establish that, for each $i,j,$ we have
\begin{equation}\label{paraerror}
\Bigl\|\,\sum_{k=0}^\infty \bigl(\g^{ij}(x)-\g^{ij}_k(x)\bigr) \psi_k(D) D_i u\,\Bigr\|_{H^s}\le C\,\|u\|_{H^s}\quad\text{if}\;\; -1\le s\le 1.
\end{equation}

By the vanishing moment condition on the radial function $\chi\in C_c^\infty$, we can write
$$
1-\hat\chi(\xi)=|\xi|^2 h(\xi)\,,\quad \text{where}\quad
|\partial^\alpha h(\xi)|\le C_\alpha
\begin{cases}
\min\bigl(1,|\xi|^{2-|\alpha|}\bigr),&|\xi|\le 1,\\
|\xi|^{-2-|\alpha|},&|\xi|\ge 1.
\end{cases}
$$
For $j,k\ge 0$, we let $h_{j,k}(\xi)=\psi_j(\xi)h(2^{-\frac k2}\xi)$ and then have
\begin{equation}\label{hjkest}
|\partial_\xi^\alpha h_{j,k}(\xi)|\le C_\alpha 2^{-|2j-k|}\,2^{-j|\alpha|}.
\end{equation}
That is, $\bigl\{2^{|2j-k|}h_{j,k}\bigr\}_{j=0}^\infty$ satisfies the derivative estimates and localization properties of a Littlewood-Paley partition of unity in $j$, uniformly over $k$. 
We then write
$$
\g-\g_k=2^{-k}\sum_{j=0}^\infty \g_{j,k}\,,\quad\text{where}\quad \g_{j,k}=-(2\pi)^{-n}\,\widehat{h_{j,k}}*(\Delta \g).
$$
We observe that
$$
\supp\bigl(\widehat{\g_{j,k}}\bigr)\subset \{ 2^{j-1}\le |\xi|\le 2^{j+2}\},\qquad\|\g_{j,k}\|_{L^\infty}\le 2^{-|2j-k|}.
$$
For the second estimate we use that $\|\widehat{h_{j,k}}*(\Delta \g)\|_{L^\infty}\le C\,2^{-|2j-k|}\|\Delta\g\|_{\BMO}.$ This follows for $j\ne 0$ from dilation invariance of $\BMO$ and the bound
$$
\int h_{j,k}(x)\,dx=0,\qquad
|h_{j,k}(x)|\le C\,2^{-|2j-k|}\,2^{jn}\bigl(1+2^{j}|x|\bigr)^{-n-1}.
$$
See for example \cite[IV.1.1.4]{St1}.
For $j=0$ we write $\g_{0,k}=(\nabla \widehat{h_{0,k}})*(\nabla \g)$.

If $j<k-1$, the function $\g_{j,k}\,\psi_k(D)u$ has Fourier transform supported in $\{2^{k-1}\le |\xi|\le 2^{k+2}\}$, so we can use orthogonality to estimate the corresponding terms in \eqref{paraerror} over $j< k-1$,
\begin{align*}
\Bigl\|\,\sum_{k=0}^\infty\sum_{j=0}^{k-2}2^{-k}\g_{j,k}\ts\psi_k(D)\ts Du\,\Bigr\|_{H^s}^2&\le
C\,\sum_{k=0}^\infty\,\Bigl\|\,\sum_{j=0}^{k-2}2^{-k}\g_{j,k}\ts\psi_k(D)\ts Du\,\Bigr\|_{H^s}^2\\
&\le
C\,\sum_{k=0}^\infty\,\Bigl(\,\sum_{j=0}^{k-2}2^{-|2j-k|}\|\psi_k(D)\ts u\|_{H^s}\,\Bigr)^2\\
&\le
C\,\sum_{k=0}^\infty\,\|\psi_k(D)u\|_{H^s}^2\le C\,\|u\|_{H^s}^2.
\end{align*}
If $j> k+1$, then $\g_{j,k}\psi_k(D)u$ is frequency supported in $\{2^{j-1}\le|\xi|\le 2^{j+2}\}$, and we estimate the corresponding terms in \eqref{paraerror} over $j>k+1$,
\begin{align*}
\Bigl\|\,\sum_{k=0}^\infty\,\sum_{j=k+2}^{\infty}2^{-k}\g_{j,k}\psi_k(D)\ts Du\,\Bigr\|_{H^s}
&\le
\sum_{k=0}^\infty\,\sum_{j=k+2}^{\infty}\, 2^{-k}\|\g_{j,k}\ts\psi_k(D)\ts Du\,\|_{H^s}\\
&\le
C\sum_{k=0}^\infty\,\sum_{j=k+2}^{\infty}\, 2^{-k+js}\|\g_{j,k}\ts\psi_k(D)\ts Du\|_{L^2}\\
&\le
C\sum_{k=0}^\infty\,\sum_{j=k+2}^{\infty}\, 2^{k(1-s)+j(s-2)}\|\psi_k(D)\ts u\|_{H^s}\\
&\le
C\sum_{k=0}^\infty\, 2^{-k}\|\psi_k(D)u\|_{H^s}\le C\,\|u\|_{H^s}.
\end{align*}
It remains to handle the case $|j-k|\le 1$. For this, we note that, by \eqref{hjkest}, the function $a_k(\xi):= 
2^k\sum_{|j-k|\le 1} h_{j,k}(\xi)$ satisfies the properties of a Littlewood-Paley partition of unity, as does $2^{-k}\psi_k(D)D:=\tilde\psi_k(D)$. We rewrite the remaining terms in \eqref{paraerror} as
$$
\Bigl\|\,\sum_{k=0}^\infty\, 2^{-k}\bigl(a_k(D)\Delta\g\bigr)\bigr(\tilde\psi_k(D)u\bigr)\,\Bigr\|_{H^s}.
$$
For $-1\le s\le 0$, we dominate this by
$$
\Bigl\|\,\sum_{k=0}^\infty\, \bigl(a_k(D)\Delta\g\bigr)\bigr(2^{-k}\tilde\psi_k(D)u\bigr)\,\Bigr\|_{L^2}\le C\,\|\Delta\g\|_{\BMO}\|u\|_{H^{-1}},
$$
where we use the paraproduct estimate of Carleson \cite{Ca} and Fefferman-Stein \cite{FS}; for a proof, see \cite[II.2.4 and IV.4.3]{St1}.

For $s\ge 0$, we use that $\bigl(a_k(D)\Delta\g\bigr)\bigr(\tilde\psi_k(D)u\bigr)$ is frequency supported in $|\xi|\le 2^{k+3}$, and bound
\begin{align*}
\Bigl\|\,\sum_{k=0}^\infty\, 2^{-k}\bigl(a_k(D)\Delta\g\bigr)\bigr(\tilde\psi_k(D)u\bigr)\,\Bigr\|_{H^s}
&\le 
\sum_{k=0}^\infty\,2^{k(s-1)}\bigl\|\bigl(a_k(D)\Delta\g\bigr)\bigr(\tilde\psi_k(D)u\bigr)\bigr\|_{L^2}
\\
&\le 
C\,\sum_{k=0}^\infty\,2^{k(s-1)}\|\Delta\g\|_{\BMO}\|\tilde\psi_k(D)u\|_{L^2}\\
&\le 
C\,\sum_{k=0}^\infty\,2^{-k}\|\Delta\g\|_{\BMO}\|\tilde\psi_k(D)u\|_{H^s}\\
&\le
C\,\|\Delta\g\|_{\BMO}\|u\|_{H^s}.
\end{align*}

\end{proof}


\subsection{Reduction to a first order equation} Write
$
(\partial_t^2+P^2)u=F+G,
$
where $G=(P^2+\Delta_\g)u$. By Lemma \ref{lem:paraest},
$$
\|G\|_{L^\infty_t\bigl([0,1];H^{s-1}(\Rd)\bigr)}\le C\,\|u\|_{L^\infty_t\bigl([0,1];H^{s}(\Rd)\bigr)}.
$$
If $v$ solves $(\partial_t^2+P^2)v=G$ with Cauchy data set to $0$, then by the Duhamel formula and energy estimates we can deduce 
$$
\|v\|_{L^q_tL^r_x}+\|v\|_{L^\infty_t H^s_x}+\|\partial_t v\|_{L^\infty_t H^{s-1}_x}\le C\,\|u\|_{L^\infty_t H^s_x},
$$ 
provided that we prove homogeneous Strichartz estimates for $\partial_t^2+P^2$. 
By splitting $u=v+(u-v)$, the Strichartz estimates of Theorem \ref{thm:strichartz'} can thus be reduced to the same estimates with $-\Delta_\g$ replaced by $P^2$; that is, by proving that the following holds on $[0,1]\times\Rd$, provided $u\in C^0H^s\cap C^1H^{s-1}$,
\begin{equation}\label{eqn:pstrichartz}
\|u\|_{L^q_t L^r_x}\le C\,\Bigl(
\|u\|_{L^\infty_t H^s_x}
+\|\partial_t u\|_{L^\infty_t H^{s-1}_x}
+\|(\partial_t^2+P^2)u\|_{L^{\tq'}_t L^{\tr'}_x}\Bigr).
\end{equation}


We replace $u(t,\cdot)$ by $\la D\ra^{-s}u(t,\cdot)$, where $\la D\ra=(1-\Delta)^{\frac 12}$, 
and note that
$$
(\partial_t^2+P^2)\la D\ra^{-s}u=[P^2,\la D\ra^{-s}]u+\la D\ra^{-s}(\partial_t^2+P^2)u.
$$
The $S_{1,\hf}$ calculus shows that $[P^2,\la D\ra^{-s}]\in S^{1-s}_{1,\frac 12}$, where we also use that 
$\partial_x p(x,\xi)\in S^1_{1,\hf}$.
Consequently, using Duhamel's principle as above we see that \eqref{eqn:pstrichartz} is equivalent to showing that, for $u\in C^0L^2\cap C^1H^{-1}$, we have
\begin{equation*}
\|\la D\ra^{-s}u\|_{L^q_t L^r_x}\le C\,\Bigl(
\|u\|_{L^\infty_t L^2_x}
+\|\partial_t u\|_{L^\infty_t H^{-1}_x}
+\|\la D\ra^{-s}(\partial_t^2+P^2)u\|_{L^{\tq'}_t L^{\tr'}_x}\Bigr).
\end{equation*}
By \eqref{garding}, with $\mu=1+C_1$ we have
$$
\la (P+\mu)f,f\ra\ge \|f\|_{L^2}^2\quad\Rightarrow\quad \|(P+\mu)f\|_{L^2}\ge\|f\|_{L^2}
\;\;\text{when}\;\; f\in H^1.
$$
By elliptic estimates we have $\|(P+\mu)f\|_{L^2(\Rd)}\ge \|f\|_{H^1(\Rd)}$, consequently $(P+\mu)^{-1}$ exists as a map from $L^2(\Rd)\rightarrow H^1(\Rd)$. 
One can show that $(P+\mu)^{-1}\in \mathrm{Op}(S^{-1}_{1,\hf})$, for example by \cite{Bo}. 

Note that since $(P+\mu)^2-P^2\in \mathrm{Op}(S^1_{1,\hf})$, the estimate remains unchanged if we replace $P$ by $P+\mu$. We will therefore assume $P$ is invertible, with $P^{-1}\in \mathrm{Op}(S^{-1}_{1,\hf})$.

The remainder of this paper is devoted to constructing the exact evolution group $E(t)=\exp({-itP})$ for the self-adjoint operator $P$, and proving dispersive estimates for its kernel. The group $E(t)$ will satisfy following properties:
\begin{itemize}
\item
$E(t)$ is a strongly continuous 1-parameter unitary group on $L^2(\Rd)$.

\medskip

\item
$E(t)$ is strongly continuous with respect to $t$ on $H^s(\Rd)$ for all $s\in \re$.

\medskip

\item
$\partial_t E(t)$ is strongly continuous with respect to $t$ from $H^s(\Rd)$ into $H^{s-1}(\Rd)$ for all $s\in \re$.

\medskip

\item
$E(0)f=f$, and $\partial_t E(t)f=-iPE(t)f=-iE(t)Pf$ if $f\in H^s(\Rd)$ for some $s\in\re$.
\end{itemize}

The second and third condition imply that $E(t)f\in C^0(H^s)\cap C^1(H^{s-1})$ if $f\in H^s(\Rd)$. For $s<0$ we understand this to mean that $E(t)$ extends continuously to such an operator from $L^2(\Rd)$. It follows from the third and fourth conditions that $E(t)f\in C^j(H^{s-j})$ for all $s\in \re$ and all $j\in\N$. We now let
$$
C(t)=\tfrac 12\bigl(E(t)+E(-t)\bigr)\,,\qquad
S(t)=\tfrac 12\bigl(E(t)-E(-t)\bigr)P^{-1}.
$$
The solution $u$ to the Cauchy problem with Sobolev data
$$
(\partial_t^2+P^2)u=F\,,\qquad u(0)=f,\quad\partial_t u(0)=g,
$$
is then given by
$$
u(t)=C(t)f+S(t)g+\int_0^t S(t-s)F(s)\,ds.
$$
The Strichartz estimates in Theorem \ref{thm:strichartz'} are thus reduced to showing that, for $s,q,\tq,r,\tr$ as in the statement of Theorem \ref{thm:strichartz},
\begin{equation}\label{fundest}
\begin{split}
\|\la D\ra^{-s}E(t)f\|_{L^q_t L^r_x([0,1]\times\Rd)}&\le C\,\|f\|_{L^2(\Rd)}\\
\|\la D\ra^{-s}\!\int_0^tE(t-s)F(s,\cdot)\|_{L^q_t L^r_x([0,1]\times\Rd)}&\le 
C\,\|\la D\ra^{1-s} F\|_{L^{\tq'}_t L^{\tr'}_x([0,1]\times\Rd)}
\end{split}
\end{equation}
Here we have used that $\la D\ra^{1-s} P^{-1}\la D\ra^{s}$ is bounded on $L^{\tr'}(\Rd)$ since it is a Calder\'on-Zygmund operator.

Similar steps apply to the squarefunction estimate. For that estimate it will be more convenient to work with smooth cutoffs of the solution. We fix $\phi\in C_c^\infty\bigl((-\hf,\hf)\bigr)$ with $\phi(t)=1$ if $|t|\le\frac 13$ . By energy conservation the squarefunction estimates of Theorem \ref{thm:strichartz'} are then reduced to showing
\begin{equation}\label{fundsqfnest}
\|\phi(t)\la D\ra^{-s_d}E(t)f\|_{L^{q_d}_x L^2_t(\Rd\times [0,1])}\le C\,\|f\|_{L^2(\Rd)}.
\end{equation}

\section{Regularity of the geodesic and Hamiltonian flows}\label{sec:Hamflow}
In this section we establish estimates for derivatives of all order on the geodesic and Hamiltonian flows of the metrics $\g_k$, as well as for spatial dilates $\g_k(\varepsilon\,\cdot)$ for $\varepsilon\le 1$.
To operate in a general context we will consider a family of metrics $\g_M$ on $\Rd$ that satisfy derivative estimates depending on the parameter $M\in[1,\infty)$. 

For a sufficiently small constant $c_d$ to be chosen depending only on the dimension $d$, we will assume a smallness condition
\begin{equation}\label{cond0}
\|\R_{ijkl}\|_{L^\infty}+\|(\g_M)_{ij}-\delta_{ij}\|_\Lip+\|\nabla_x^2(\g_M)_{ij}\|_{BMO}\le c_d.
\end{equation}
Here, $\R_{ijkl}$ is the Riemann curvature tensor of $\g_M$. This tensor, as well as the Christoffel symbols $\Gamma^n_{ij}$, depends on $M$, but to simplify notation we suppress the subscript $M$.

We additionally assume that, for constants $C_\beta$ independent of $M$, 
\begin{align}\label{cond1}
&\|\partial_x^\beta\g_M^{ij}\|_{L^\infty}\le C_\beta\,M^{|\beta|-1}\,,\qquad |\beta|\ge 1,\\
\label{cond2}
&\|\partial_x^\beta\R_{ijkl}\|_{L^\infty}\le C_\beta\,M^{|\beta|}\,,\qquad |\beta|\ge 0.
\rule{0pt}{17pt}
\end{align}

Let $\gamma(t,y,w)$ be the geodesic for $\g_M$ with initial conditions $(y,w)$:
$$
\partial_t^2 \gamma^n=\sum_{ij}\Gamma_{ij}^n(\gamma)\dot\gamma^i\dot \gamma^j,\qquad
\gamma(0,y,w)=y,\quad \dot\gamma(0,y,w)=w,
$$
where $\dot\gamma\equiv \partial_t\gamma$.
Note that by \eqref{cond0}--\eqref{cond1} we have
\begin{equation}\label{christoffel}
\|\Gamma^n_{ij}\|_{L^\infty}\lesssim c_d,\qquad
\|\partial_x^\beta\Gamma^n_{ij}\|_{L^\infty}\le C_\beta\,M^{|\beta|},\quad |\beta|\ge 1,
\end{equation}
where in this section $a\lesssim b$ means that $a\le C\, b$, where $C$ depends only on the dimension $d$.

\begin{theorem}\label{thm:geodflow}
Suppose that $\g_M$ satisfies \eqref{cond0}--\eqref{cond2}, for a suitably small constant $c_d$. Then
there are constants $C_{\alpha,\beta}$, depending only on the constants $C_\beta$ in \eqref{cond1}-\eqref{cond2}, so that over the set $\frac 12\le|w|\le 2$ and $|t|\le 1$,
\begin{equation}\label{smallvar}
|\partial_y\gamma-\I\ts|+|\partial_y\dot\gamma\ts|+
|\partial_w\dot\gamma-\I\ts|
\lesssim c_d,\quad |\partial_w\gamma-t\,\I\ts|\lesssim c_d \, |t|,
\end{equation}
and
$$
|\partial_y^\beta\partial_w^\alpha \gamma(t,y,w)|
+|\partial_y^\beta\partial_w^\alpha \dot\gamma(t,y,w)|
\le C_{\alpha,\beta}\,M^{|\alpha|+|\beta|-1},\qquad |\alpha|+|\beta|\ge 1.
$$
Additionally,
$$
|\partial_y^\beta\partial_w^\alpha \gamma(t,y,w)|\le 
C_{\alpha,\beta}\,|t|\,M^{|\alpha|+|\beta|-1},
\qquad\text{if}\;\; |\alpha|\ge 1\;\;\text{or}\;\;|\beta|\ge 2.
$$
\end{theorem}
\begin{proof}
We produce a (not necessarily orthonormal) frame $\{V_m\}_{m=1}^d$ along $\gamma(t,y,w)$ by parallel translation of the standard frame $\{\partial_m\}_{m=1}^d$. We label the resulting vector fields $V_m(t,y,w)=\sum_n v_m^n(t,y,w)\partial_n$. The dual frame $\{V^n\}_{n=1}^d$ under $\g_M$ is obtained by parallel translating $\sum_m \g^{nm}_M(y)\partial_m$ along $\gamma$, so $v^{n,l}(t,y,w)=\sum_m\g^{nm}_M(y)v_m^l(t,y,w)$, and derivative estimates for the functions $v^{n,l}$ will follow directly from those for $v_m^l$.
We have
\begin{equation}\label{veqn}
\partial_tv^n_m=-\Gamma^n_{ij}(\gamma) \ts \dot \gamma^i \ts v^j_m\,,\qquad v_m^n(0,t,w)=\delta_m^n.
\end{equation}
We expand the variation of the flow in the initial parameters using the frame,
\begin{equation}\label{xvar}
\begin{split}
\partial_{y^k} \gamma=\sum_m f^m_k(t,y,w)\ts V_m=\sum_{mj}f^m_k(t,y,w)\ts v_m^j(t,y,w)\ts\partial_j,
\\
\partial_{w^k} \gamma=\sum_m h^m_k(t,y,w)\ts V_m=\sum_{mj}h^m_k(t,y,w)\ts v_m^j(t,y,w)\ts\partial_j.
\end{split}
\end{equation}
By \eqref{veqn} we then have
\begin{equation}\label{xdotvar}
\begin{split}
\partial_{y^k} \dot \gamma^n=\partial_t \partial_{y^k} \gamma^n=
\sum_m (\partial_t f^m_k)\ts v^n_m-\sum_{ij}\Gamma_{ij}^n(\gamma)\,\dot \gamma^i f^m_k v^j_m,\\
\partial_{w^k} \dot \gamma^n=\partial_t \partial_{w^k} \gamma^n=
\sum_m (\partial_t h^m_k)\ts v^n_m-\sum_{ij}\Gamma_{ij}^n(\gamma)\,\dot \gamma^i h^m_k v^j_m.
\end{split}
\end{equation}
Since $D_t^2\partial_{y^k}\gamma=\sum_m(\partial_t^2 f^m_k)\,V_m$, with $D_t$ covariant differention in $t$, the Jacobi variation formula yields
\begin{equation}\label{Jacobi}
\partial_t^2 f^m_k=\sum_n\Bigl(\sum_{ijlp}\R_{ijlp}(\gamma)\dot \gamma^i v_n^j \dot \gamma^l v^{m,p}\Bigr)  f^n_k,
\end{equation}
with the following initial conditions, where the second holds by \eqref{xdotvar},
$$
f^m_k(0,y,w)=\delta^m_k,\qquad
\partial_t f^m_k(0,y,w)=\sum_i\Gamma^m_{ik}(y)\ts w^i.
$$
The equation \eqref{Jacobi} holds with $f$ replaced by $h$, with initial conditions
$$
h^m_k(0,y,w)=0,\qquad
\partial_t h^m_k(0,y,w)=\delta^m_k.
$$
The bound $|\dot \gamma|\lesssim 1$ and $|v|\lesssim 1$, together with \eqref{christoffel}, yield for $|t|\le 1$,
$$
|v^n_j-\delta^n_j|+|f^m_k-\delta^m_k|+|\partial_t f^m_k|+|\partial_t h^m_k-\delta^m_k|
\,\lesssim\, c_d,\quad
|h^m_k-t\delta^m_k|\lesssim c_d|t|.
$$
Together with \eqref{christoffel} and \eqref{xvar}--\eqref{xdotvar} these yield the bound \eqref{smallvar}.

Assume we have shown the following for $|\alpha|+|\beta|\le N-1$, where $N\ge 1$,
\begin{equation}\label{inductionbound}
|\partial_y^\beta\partial_w^\alpha (v^n_j,f^m_k,h^m_k)|
+|\partial_y^\beta\partial_w^\alpha (\partial_t f^m_k,\partial_t h^m_k)|\le\, C_{\alpha,\beta}\,M^{|\alpha|+|\beta|}.
\end{equation}
Using \eqref{xvar}, \eqref{christoffel}, and \eqref{xdotvar}, we conclude that if $1\le|\alpha|+|\beta|\le N$,
$$
|\partial_y^\beta\partial_w^\alpha\gamma^n|+
|\partial_y^\beta\partial_w^\alpha\dot \gamma^n|\,
\le\,C_{\alpha,\beta}\,M^{|\alpha|+|\beta|-1}.
$$
By \eqref{veqn} and the Leibniz rule, for $|\alpha|+|\beta|= N$ we then can write
$$
\partial_t\partial_y^\beta\partial_w^\alpha v^n_m=-\Gamma^n_{ij}(\gamma) \ts \dot \gamma^i \ts \partial_y^\beta\partial_w^\alpha v^j_m+\mathcal{O}\bigl(M^{|\alpha|+|\beta|}\bigr).
$$
Similarly, by \eqref{Jacobi}, for $|\alpha|+|\beta|=N$ we have
$$
\partial_t^2 \partial_y^\beta\partial_w^\alpha f^m_k=\sum_n\Bigl(\sum_{ijlp}\R_{ijlp}(\gamma)\dot \gamma^i v_n^j \dot \gamma^l v^{m,p}\Bigr) \partial_y^\beta\partial_w^\alpha f^n_k+\mathcal{O}(M^{|\alpha|+|\beta|}).
$$
and the same for $f$ replaced by $h$. By the initial conditions, we have
$$
\partial_y^\beta\partial_w^\alpha (v^m_n,f^m_k,h^m_k,\partial_t h^m_k)\bigr|_{t=0}=0,\qquad
|\partial_t \partial_y^\beta\partial_w^\alpha f^m_k(0,y,w)|\le C_{\alpha,\beta}\, M^{|\alpha|+|\beta|}.
$$
An application of Gronwall's lemma then yields, for $|\alpha|+|\beta|=N$,
$$
|\partial_y^\beta\partial_w^\alpha (v^m_n,f^m_k,h^m_k)|+
|\partial_y^\beta\partial_w^\alpha (\partial_t f^m_k,\partial_t h^m_k)| \le C_{\alpha,\beta}\, M^{|\alpha|+|\beta|},
$$
and \eqref{inductionbound} follows for $|\alpha|+|\beta|=N$ by \eqref{xvar} and \eqref{xdotvar}, hence all $\alpha,\beta$ by induction. As above, this implies the desired bounds for 
$\partial_y^\beta\partial_w^\alpha(\gamma,\dot \gamma)$.

The last estimate of the theorem follows from the bound on $|\partial_t\partial_y^\beta\partial_w^\alpha \gamma|$, since $\partial_y^\beta\partial_w^\alpha \gamma(0,y,w)=0$ if either $|\alpha|\ge 1$ or $|\beta|\ge 2$.
\end{proof}


We now consider the related Hamiltonian flow. Let 
$$
p_M(x,\eta)=\Bigl(\,\sum_{ij}\g^{ij}_M(x)\ts\eta_i\eta_j\Bigr)^{\frac 12},
$$
and consider the solution $\bigl(x(t,y,\eta),\xi(t,y,\eta)\bigr)$ to Hamilton's equations,
$$
\dot x=(\nabla_\xi p_M)(x,\xi),
\quad
\dot \xi=-(\nabla_x p_M)(x,\xi),
\quad
x(0,y,\eta)=y,\;\;\xi(0,y,\eta)=\eta.
$$
These are related to the geodesic flow by the following,
\begin{align*}
x^i(t,y,\eta)&=\gamma^i\bigl(t,y,w(y,\eta)\bigr)\,,\\
\xi_j(t,y,\eta)&=p_M(y,\eta)\,\sum_j \g_{M,ij}(\gamma)\,
\dot\gamma^j\bigl(t,y,w(y,\eta)\bigr),\rule{0pt}{15pt}
\end{align*}
where
$$
w^i(y,\eta)=\frac 1{p_M(y,\eta)}\;\sum_j \g_M^{ij}(y)\,\eta_j.
$$
It follows from \eqref{cond0} that 
\begin{equation*}
|\partial_y w|+|w-|\eta|^{-1}\eta\ts|+|\partial_\eta w-(\ts\I-|\eta|^{-2}\eta\otimes\eta)\ts|\lesssim c_d,
\end{equation*}
and from \eqref{cond1} and homogeneity that
$$
\bigl|\partial_y^\beta\partial_\eta^\alpha w(y,\eta)\bigr|\le C_{\alpha,\beta}\,M^{|\beta|-1}\,|\eta|^{-|\alpha|}.
$$
Observe that $\,\I-|\eta|^{-2}\eta\otimes\eta=\Pi_\eta^\perp,\,$ the projection onto the plane perpendicular to $\eta$.
We use this to deduce the following corollary of Theorem \ref{thm:geodflow}.

\begin{corollary}\label{cor:geodflow}
Suppose that $\g_M$ satisfies \eqref{cond0}--\eqref{cond2}, for a suitably small constant $c_d$. Then
there are constants $C_{\alpha,\beta}$, depending only on the constants $C_\beta$ in \eqref{cond1}--\eqref{cond2}, so that for $|t|\le 1$
$$
|\partial_y x-\I\ts|+|\partial_\eta \xi-\I\ts|\lesssim c_d,\quad
|\partial_y \xi\ts|+|\xi-\eta|\lesssim c_d\,|\eta|,\quad
|\partial_\eta x-t\,\Pi_\eta^\perp\ts|\lesssim c_d\,|t|,
$$
and when $|\alpha|+|\beta|+m\ge 1$,
$$
|\eta|\ts|\partial_\eta^\alpha\partial_y^\beta\partial_t^m x(t,y,\eta)|
+|\partial_\eta^\alpha\partial_y^\beta\partial_t^m \xi(t,y,\eta)|
\le C_{\alpha,\beta}\,M^{|\alpha|+|\beta|+m-1}\,|\eta|^{1-|\alpha|}.
$$
Additionally,
$$
|\partial_y^\beta\partial_\eta^\alpha x(t,y,\eta)|\le 
C_{\alpha,\beta}\,|t|\,M^{|\alpha|+|\beta|-1}\,|\eta|^{-|\alpha|},
\qquad |\alpha|\ge 1\;\;\text{or}\;\;|\beta|\ge 2.
$$
\end{corollary}
\begin{proof} The estimates other than those involving derivatives in $t$ follow from Theorem \ref{thm:geodflow}. Estimates on derivatives in $t$ follow by induction using Hamilton's equations and the following consequence of \eqref{cond1},
$$
\bigl|\partial_x^\beta\partial_\xi^\alpha \bigl(\nabla_{\xi}p_M\bigr)\bigr|
+
|\xi|^{-1}\bigl|\partial_x^\beta\partial_\xi^\alpha \bigl(\nabla_x p_M\bigr)\bigr|
\le C_{\alpha,\beta}\,M^{|\beta|}\,|\xi|^{-|\alpha|}.
$$
\end{proof}


For the generating function $\varphi_k(t,x,\eta)$, we need consider the function $y(t,x,\eta)$ that is the inverse of the map $y\rightarrow x(t,y,\eta)$.

\begin{theorem}\label{thm:yest}
Suppose that $\g_M$ satisfies \eqref{cond0}--\eqref{cond2}, for a suitably small constant $c_d$. Then
there are constants $C_{\alpha,\beta}$, depending only on the constants $C_\beta$ in \eqref{cond1}--\eqref{cond2}, so that if $|t|\le 1$ and $\eta\ne 0$ the map $y\rightarrow x(t,y,\eta)$ is invertible. The inverse map $y(t,x,\eta)$ satisfies $|\partial_x y-\I\ts|\lesssim c_d$, and
$$
|\partial_x^\beta\partial_\eta^\alpha y(t,x,\eta)|
\le C_{\alpha,\beta}\,M^{|\alpha|+|\beta|-1}\,|\eta|^{-|\alpha|},\quad |\alpha|+|\beta|\ge 1.
$$
Additionally,
$$
|\partial_x^\beta\partial_\eta^\alpha y(t,x,\eta)|\le 
C_{\alpha,\beta}\,|t|\,M^{|\alpha|+|\beta|-1}\,|\eta|^{-|\alpha|},
\qquad |\alpha|\ge 1\;\;\text{or}\;\;|\beta|\ge 2.
$$
Also, for the function $\xi(t,x,\eta):=\xi(t,y(t,x,\eta),\eta)$,
$$
|\partial_x^\beta\partial_\eta^\alpha \xi(t,x,\eta)|
\le C_{\alpha,\beta}\,M^{|\alpha|+|\beta|-1}\,|\eta|^{1-|\alpha|},\quad |\alpha|+|\beta|\ge 1.
$$
\end{theorem}
\begin{proof}
We have $|x(t,y,\eta)-y|\lesssim |t|$, so
for each $\eta\ne 0$ and $|t|\le 1$ 
the map $y\rightarrow x$ is proper and
hence a closed mapping. Since $|\partial_y x-\I\ts|\lesssim c_d$ it is an open mapping,  
hence onto and one-to-one  by connectivity and simple connectivity of $\Rd$.
Thus $y\rightarrow x(t,y,\eta)$ is a diffeomorphism of $\Rd$, with inverse satisfying $|\partial_x y-\I|\lesssim c_d$.
The estimates of the theorem are then a consequence of the
inverse function theorem and Corollary \ref{cor:geodflow}.
\end{proof}

\section{Estimates for solutions of the eikonal equation}\label{sec:phasefunction}
In this section we establish estimates on derivatives of the solution to the eikonal equation for $\g_k$.
For simplicity we consider $0\le t\le 1$.
Let $\g_k$ be the mollification of $\g$ at spatial scale $2^{-\frac k2}$ from Chapter 2, and let $\varphi_k$
be the solution to the eikonal equation
$$
\partial_t\varphi_k(t,x,\eta)=-p_k\bigl(x,\nabla_x\varphi_k(t,x,\eta)\bigr)\,,\qquad
\varphi_k(0,x,\eta)=\langle x,\eta\rangle.
$$
Then $\varphi_k(t,x,\eta)=\sum_j \eta_j \ts y_j(t,x,\eta)$, where $y(t,x,\eta)$ is as in Theorem \ref{thm:yest}, and the estimates of that theorem hold with $M=2^k$.
Furthermore, 
$$
\partial_{\eta_j}\varphi_k(t,x,\eta)=y_j(t,x,\eta),\qquad
\partial_{x_j}\varphi_k(t,x,\eta)=\xi_j(t,x,\eta).
$$
We then easily read off the following from Theorem \ref{thm:yest},
\begin{align}
\label{dxphi}
\bigl|\partial_x^\beta \varphi_k(t,x,\eta)\bigr|&\le C_\beta\,2^{\frac k2(|\beta|-2)}|\eta|,\qquad |\beta|\ge 2,\\
\label{detaphi}
\bigl|\partial_x^\beta \partial_\eta\varphi_k(t,x,\eta)\bigr|&\le C_\beta\,t\,2^{\frac k2(|\beta|-1)},\qquad |\beta|\ge 2,\\
\label{deta2phi}
\bigl|\partial_x^\beta\partial_\eta^\alpha \varphi_k(t,x,\eta)\bigr|&\le C_{\alpha,\beta}\,t\,2^{\frac k2(|\alpha|+|\beta|-2)}|\eta|^{1-|\alpha|},\qquad |\alpha|\ge 2.
\end{align}
Additionally,
\begin{equation}\label{dxdetaphi}
\bigl|\partial_x\partial_\eta\varphi_k(t,x,\eta)|\le C.
\end{equation}
The following shows that some estimates can be improved for derivatives in $\eta$, which is key to controlling the evolution operators for small $t$.

\begin{theorem}\label{phikest}
Assume that $|\alpha|\ge 2$ or $|\beta|\ge 2$. Then when $2^{-k}\le t\le 1$,
$$
\bigl|\la\eta,\partial_\eta\ra^j\partial_x^\beta\partial_\eta^\alpha\varphi_k(t,x,\eta)|\le
C_{j,\alpha,\beta}\,\bigl(t^{\hf}2^{\frac k2}\bigr)^{|\alpha|}
2^{\frac k2(|\beta|-2)}|\eta|^{1-|\alpha|},
$$
and when $0\le t\le 2^{-k}$,
$$
\bigl|\la\eta,\partial_\eta\ra^j\partial_x^\beta\partial_\eta^\alpha\varphi_k(t,x,\eta)|\le
C_{j,\alpha,\beta}\,
2^{\frac k2(|\beta|-2)}|\eta|^{1-|\alpha|}.
$$
\end{theorem}
\begin{proof} By homogeneity it suffices to consider the case $j=0$.
If $|\alpha|\le 1$, the estimates for all $0\le t\le 1$ follow from \eqref{dxphi}--\eqref{detaphi}. To handle $|\alpha|\ge 2$,
we take a parameter $\e$ with $2^{-k/2}\le\e\le 1$. Let $\g_{\e,k}(x)=\g_k(\e x)$, where
$\g_k$ is the localization of $\g$ to frequency $2^{k/2}$. Similarly, let $p_{\e,k}(x,\xi)=p_k(\e x,\xi)$. Let $\varphi_{\e,k}$ be the solution to
$$
\partial_t\varphi_{\e,k}(t,x,\eta)=-p_{\e,k}\bigl(x,\nabla_x\varphi_{\e,k}(t,x,\eta)\bigr)\,,\qquad
\varphi_{\e,k}(0,x,\eta)=\langle x,\eta\rangle.
$$
Then by homogeneity we have
\begin{equation}\label{ephase}
\varphi_k(t, x,\eta)=\e\ts\varphi_{\e, k}(\e^{-1}t,\e^{-1}x,\eta)\,.
\end{equation}
The metric $\g_{\e,k}(x)$ is mollification of $\g(\e x)$ at scale
$\e^{-1} 2^{-\frac k2}\le 1$. Since $\g(\e x)$ is Lipschitz with bounded curvature, uniformly over $\e\in[0,1]$, we can apply estimates \eqref{dxphi}--\eqref{deta2phi} with $2^{\frac k2}$ replaced by $M=\e 2^{\frac k2}$.

For $2^{-\frac k2}\le t\le 1$ we take $\e=t$ in \eqref{ephase}, and apply \eqref{deta2phi} with $M=t\ts 2^{-\frac k2}$ to get
$$
\bigl|\partial_x^\beta\partial_\eta^\alpha\varphi_k(t,x,\eta)|\le C_{\alpha,\beta}\,t^{|\alpha|-1}\,2^{\frac k2(|\alpha|+|\beta|-2)}\,|\eta|^{1-|\alpha|}.
$$
For $|\alpha|\ge 2$ this implies the desired estimate.

For $0\le t\le 2^{-\frac k2}$ we take $\e=2^{-\frac k2}$ in \eqref{ephase}, and apply \eqref{deta2phi} with $2^{\frac k2}$ replaced by $1$ to get
$$
\bigl|\partial_x^\beta\partial_\eta^\alpha\varphi_k(t,x,\eta)|\le C_{\alpha,\beta}\,t\,2^{\frac k2|\beta|}\,|\eta|^{1-|\alpha|}.
$$
Since $t\le t^{\frac{|\alpha|}2}2^{\frac k2(|\alpha|-2)}$ for $t\ge 2^{-k}$ and $|\alpha|\ge 2$, and $t\ts 2^{\frac k2|\beta|}\le 2^{\frac k2(|\beta|-2)}$ for $0\le t\le 2^{-k}$,
this concludes the theorem for $0\le t\le 2^{-\frac k2}$.
\end{proof}

As a corollary we obtain the estimates we need for linearizing the phase function, and showing the symbols are slowly varying, for $\eta$ in an appropriate conical region. Given a unit vector $\nu$, and $2^{-k}\le t\le 1$, we define the dyadic/conic region
\begin{equation}\label{rknudef}
\Omega^\nu_{k,t}=\bigl\{\eta:\tfrac 23 2^{k-1}\le|\eta|\le \tfrac 32 2^{k+2},\;|\nu-|\eta|^{-1}\eta|\le\tfrac 1{16} t^{-\hf}2^{-\frac k2}\,\bigr\}.
\end{equation}
Note that on this region, since $t^{-\hf}2^{-\frac k2}\le 1$,
$$
|\eta|\ge \la\nu,\eta\ra\ge\tfrac 34|\eta|,\qquad \bigl|\Pi_{\nu^\perp}\eta\bigr|\le t^{-\hf}2^{\frac k2},
$$
where $\Pi_{\nu^\perp}$ is projection onto the hyperplane perpendicular to $\nu$.

\begin{corollary}\label{phikest'}
The following estimates hold if $\eta\in \Omega^\nu_{k,t}$ and $2^{-k}\le t\le 1$.
\begin{equation}\label{linphase}
\bigl|\la\nu,\partial_\eta\ra^j\partial_\eta^\alpha\partial_x^\beta
(\partial_\eta^2\varphi_k)(t,x,\eta)\bigr|\le 
C_{j,\alpha,\beta}\,t\ts 2^{-k}\,2^{-kj}\,\tkhf^{|\alpha|}\,2^{\frac k2 |\beta|},
\end{equation}
\begin{multline}\label{linsymbol}
\bigl|\la\nu,\partial_\eta\ra^j\partial_\eta^\alpha\partial_x^\beta
(\partial_\eta\partial_x\varphi_k)(t,x,\eta)\bigr|+
2^{-k}\bigl|\la\nu,\partial_\eta\ra^j\partial_\eta^\alpha\partial_x^\beta
(\partial^2_x\varphi_k)(t,x,\eta)\bigr|\\
\le C_{j,\alpha,\beta}\,2^{-kj}\,\tkhf^{|\alpha|}\,2^{\frac k2 |\beta|},
\end{multline}
and
\begin{multline}\label{linphase'}
\bigl|\la\nu,\partial_\eta\ra^j\partial_\eta^\alpha\partial_x^\beta
\bigl(\varphi_k(t,x,\eta)-\eta\cdot\nabla_\eta\varphi_k(t,x,\nu)\bigr)\bigr|\\
\le C_{j,\alpha,\beta}\,2^{-kj}\,\tkhf^{|\alpha|}\,2^{\frac k2 |\beta|}.
\end{multline}
For $0\le t\le 2^{-k}$ these hold for $\eta$ in the dyadic shell $\frac 23 2^{k-1}\le|\eta|\le \frac 32 2^{k+2}$ if $t$ is replaced by $2^{-k}$ on the right hand side.
\end{corollary}
\begin{proof}
We consider the estimate \eqref{linphase}. Theorem \ref{phikest} gives the following,
$$
\bigl|\la\eta,\partial_\eta\ra^j\partial_\eta^\alpha\partial_x^\beta
(\partial_\eta^2\varphi_k)(t,x,\eta)\bigr|\le C_{j,\alpha,\beta}\,t\ts 2^{-k}\,(t^{\hf}2^{-\frac k2})^{|\alpha|}\,2^{\frac k2 |\beta|}.
$$
After rotation we may assume that $\nu=(1,0,\ldots,0)$. We proceed by induction in $j$, the case $j=0$ being the same as above.
Suppose then that \eqref{linphase} holds for $j<j_0$. We expand
$$
\la\eta,\partial_\eta\ra^{j_0}=\eta_1^{j_0}\partial_{\eta_1}^{j_0}+\sum_{\substack{j+|\alpha|\le j_0 \\ j<j_0}} c_{j_0,j,\alpha}\,\eta_1^j\ts\eta'^\alpha\partial_{\eta_1}^j\partial_{\eta'}^\alpha
$$
Since $\eta_1\le \frac 32 2^{k+2}$ and $|\eta'|\le t^{-\hf}2^{\frac k2}$ on $\Omega^\nu_{k,t}$, the induction hypothesis yields
$$
\bigl|\eta_1^{j_0}\partial_{\eta_1}^{j_0}\partial_\eta^\alpha\partial_x^\beta
(\partial_\eta^2\varphi_k)(t,x,\eta)\bigr|
\le C_{j,\alpha,\beta}\,t\ts2^{-k}\,\tkhf^{|\alpha|}\,2^{\frac k2 |\beta|},
$$
which establishes \eqref{linphase} for $j=j_0$, since $\eta_1\ge 2^{k-2}$ on $\Omega^\nu_{k,t}$. Similar steps establish \eqref{linsymbol}.

The estimate \eqref{linphase'} follows from \eqref{linphase} if $|\alpha|\ge 2$, so it suffices to consider $|\alpha|\le 1$. The proof for $|\beta|\ne 0$ will follow from the proof for $\beta=0$ with $\varphi_k$ replaced by $\partial_x^\beta\varphi_k$, so we assume $\beta=0$.
We then rotate to assume that $\nu=e_1$, in which case by homogeneity the estimate becomes
\begin{multline*}
\bigl|\partial_{\eta_1}^j\partial_{\eta'}^\alpha
\bigl(\varphi_k(t,x,\eta_1,\eta')-\varphi_k(t,x,\eta_1,0)-\eta'\cdot\nabla_{\eta'}\varphi_k(t,x,\eta_1,0)\bigr)\bigr|\\
\le C_{j,\alpha}\,2^{-kj}\,\tkhf^{|\alpha|}.
\end{multline*}
This estimate follows from a Taylor expansion argument together with \eqref{linphase}, since $|\eta'|\le t^{-\frac 12}2^{\frac k2}$ on $\Omega_{k,t}^{e_1}$.

For $0\le t\le 2^{-k}$ the desired estimates follow easily from Theorem \ref{phikest}.
\end{proof}

We also record estimates for time derivatives of $\varphi_k$, which will be used in establishing space-time energy estimates.

\begin{corollary}\label{phikest''}
Assume that $2^{-k}\le t\le 1$. If $|\alpha|\ge 1$, then
$$
\bigl|\la\eta,\partial_\eta\ra^j\partial_\eta^\alpha\partial_t\varphi_k(t,x,\eta)\bigr|\le 
C_{j,\alpha}\bigl(t^{\hf}2^{\frac k2}\bigr)^{|\alpha|-1}|\eta|^{1-|\alpha|},
$$
and if $m+|\beta|\ge 2$,
$$
\bigl|\la\eta,\partial_\eta\ra^j\partial_x^\beta\partial_\eta^\alpha\partial_t^m\varphi_k(t,x,\eta)|\le
C_{j,m,\alpha,\beta}
\bigl(t^{\hf}2^{\frac k2}\bigr)^{|\alpha|}2^{\frac k2(m+|\beta|-2)}|\eta|^{1-|\alpha|}.
$$
If $0\le t\le 2^{-k}$, both of these estimates hold with $t$ replaced by $2^{-k}$ on the right hand side.
\end{corollary}
\begin{proof}
By homogeneity we may assume $j=0$.
The estimates that involve no derivatives in $t$, the second estimate with $m=0$, hold by Theorem \ref{phikest}. We assume both estimates hold for derivatives up to order $m\ge 0$ in $t$, and prove they hold for derivatives of order $m+1$ in $t$. Write $\partial_t\varphi_k=p_k(x,\nabla_x\varphi_k)$, and observe that $\partial_x^\beta\partial_\eta^\alpha\partial_t^{m+1}\varphi_k$ can be written as a sum of terms of the form
$$
\bigl(\partial_x^{\beta_0}\partial_\xi^\gamma p_k\bigr)(x,\nabla_x\varphi_k)\bigl(\partial_x^{\beta_1}\partial_\eta^{\alpha_1}\partial_t^{m_1}\nabla_x\varphi_k\bigr)\cdots
\bigl(\partial_x^{\beta_{|\gamma|}}\partial_\eta^{\alpha_{|\gamma|}}\partial_t^{m_{|\gamma|}}\nabla_x\varphi_k\bigr),
$$
where $\sum_{j=0}^{|\gamma|}\beta_j=\beta$, $\;\sum_{j=1}^{|\gamma|}\alpha_j=\alpha$, $\;\sum_{j=1}^{|\gamma|}m_j=m$. If $|\beta|=m=0$, we must have $|\alpha_j|\ge 1$ for all $j$, and the first estimate of the corollary is a result of the following bounds from \eqref{eqn:pkest} and Theorem \ref{phikest},
\begin{align*}
\bigl|\bigl(\partial_\xi^\gamma p_k\bigr)(x,\nabla_x\varphi_k)\bigr|&\le 
C_\gamma\,|\eta|^{1-|\gamma|},
\\
|\partial_\eta^{\alpha_j}\nabla_x\varphi_k|&\le 
C_{\alpha_j}\,\bigl(t^{\hf}2^{\frac k2}\bigr)^{|\alpha_j|-1}|\eta|^{1-|\alpha_j|}.
\end{align*}
Assume that $|\beta|+m\ge 1$. If $|\beta_0|\ge 1$, then the second estimate of the corollary is a result of the following bounds from \eqref{eqn:pkest} and the induction assumption,
\begin{align*}
\bigl|\bigl(\partial_x^{\beta_0}\partial_\xi^\gamma p_k\bigr)(x,\nabla_x\varphi_k)\bigr|&\le 
C_{\gamma,\beta_0}\,2^{\frac k2(|\beta_0|-1)}|\eta|^{1-|\gamma|},\\
\bigl|\partial_x^{\beta_j}\partial_\eta^{\alpha_j}\partial_t^{m_j}\nabla_x\varphi_k\bigr|&\le 
C_{\alpha_j,\beta_j,m_j}\,2^{\frac k2(|\beta_j|+m_j)}\bigl(t^{\hf}2^{\frac k2}\bigr)^{|\alpha_j|}|\eta|^{1-|\alpha_j|}.
\end{align*}
Finally, if $|\beta_0|=0$ then we may assume $|\beta_1|+m_1\ge 1$, and use the bounds
\begin{align*}
\bigl|\bigl(\partial_\xi^\gamma p_k\bigr)(x,\nabla_x\varphi_k)\bigr|&\le 
C_{\gamma}\,|\eta|^{1-|\gamma|},\\
\bigl|\partial_x^{\beta_1}\partial_\eta^{\alpha_1}\partial_t^{m_1}\nabla_x\varphi_k\bigr|&\le 
C_{\alpha_1,\beta_1,m_1}\,2^{\frac k2(|\beta_1|+m_1-1)}\bigl(t^{\hf}2^{\frac k2}\bigr)^{|\alpha_1|}|\eta|^{1-|\alpha_1|}\\
\bigl|\partial_x^{\beta_j}\partial_\eta^{\alpha_j}\partial_t^{m_j}\nabla_x\varphi_k\bigr|&\le 
C_{\alpha_j,\beta_j,m_j}\,2^{\frac k2(|\beta_j|+m_j)}\bigl(t^{\hf}2^{\frac k2}\bigr)^{|\alpha_j|}|\eta|^{1-|\alpha_j|}.
\end{align*}
\end{proof}

\section{Parametrix for the dyadically localized equation}\label{sec:parametrix}
In this section, we use the eikonal solution $\varphi_k$ to produce an approximation to the wave group for $P$ with data at frequency scale $2^k$. In the next section we will use these approximations to produce the exact evolution group for $P$ by iteration. For $k\ge 2$ we define
$$
\tP_k=\frac 12\sum_{j=k-1}^{k+1}\beta_j(D)\bigl(p_j(x,D)+p_j(x,D)^*\bigr)\beta_j(D).
$$
Let $\tp_k(x,\eta)$ denote the symbol of $\tP_k$. Recalling that $\beta_j^2=\psi_j$, then
\begin{equation}\label{pkerror}
\tp_k(x,\eta)=\sum_{j=k-1}^{k+1}p_j(x,\eta)\psi_j(\eta)+\sum_{j=k-1}^{k+1}q_j(x,\eta)\beta_j(\eta),
\end{equation}
where $q_j\in S^0_{1,\frac 12}$, uniformly over $j$.
For $|\eta|\in\bigl[\frac 34\ts 2^{k},\frac 43\ts 2^{k+1}\bigr]$ we define
$$
b_k(t,x,\eta)=e^{-i\varphi_k(t,x,\eta)}\bigl(\partial_t+i\tP_k\bigr)e^{i\varphi_k(t,x,\eta)}
$$
where $\tP_k$ acts on $x$.

We then define $W_k(t)$ for $k\ge 2$ by
\begin{equation}\label{wkdef}
\bigl(W_k(t)f\bigr)(x)=\frac 1{(2\pi)^d}\int e^{i\varphi_k(t,x,\eta)}\,\psi_k(\eta)\hatf(\eta)\,d\eta\,.
\end{equation}
It follows that $\bigl(\partial_t+i\tP_k\bigr)W_k(t)=B_k(t)$, where
\begin{equation}\label{bkdef}
\bigl(B_k(t)f\bigr)(x)=\frac 1{(2\pi)^d}\int e^{i\varphi_k(t,x,\eta)}\,b_k(t,x,\eta)\psi_k(\eta)\hatf(\eta)\,d\eta\,.
\end{equation}

\begin{theorem}\label{thm:bkest}
For $|t|\le 1$ the symbol $b_k(t,x,\eta)$ satisfies
$$
\bigl|\la\eta,\partial_\eta\ra^j\partial_\eta^\alpha\partial_x^\beta\partial_t^m
b_k(t,x,\eta)\bigr|\le 
C_{j,\alpha,\beta,m}\,\begin{cases} \,\tkhf^{|\alpha|}\,2^{\ts\frac k2\ts(|\beta|+m)}, & |t|\ge 2^{-k},\\
\,2^{-k\ts|\alpha|}\,2^{\ts\frac k2\ts(|\beta|+m)}, & |t|\le 2^{-k}\rule{0pt}{15pt}.
\end{cases}
$$
\end{theorem}
\begin{proof}
The symbol $b_k(t,x,\eta)$ is given by the oscillatory integral
\begin{equation}\label{bkform}
i\partial_t\varphi_k(t,x,\eta)
+\frac {i}{(2\pi)^n}\int e^{i\la x-y,\zeta\ra + i\varphi_k(t,y,\eta)-i\varphi_k(t,x,\eta)}
\ts\tp_k(x,\zeta)\,dy\,d\xi
\end{equation}
where recall that we assume $|\eta|\in\bigl[\frac 34\ts 2^{k},\frac 43\ts 2^{k+1}\bigr]$.
We write
$$
\varphi_k(t,y,\eta)-\varphi_k(t,x,\eta)=(y-x)\cdot V(t,x,y-x,\eta)\,,
$$
where
$$
V(t,x,h,\eta)=\int_0^1(\nabla_x\varphi_k)(t,x+sh,\eta)\,ds.
$$
Then 
\begin{equation*}
V(t,x,0,\eta)=\nabla_x\varphi_k(t,x,\eta)\,,\quad\partial_{h_i}V_j(t,x,0,\eta)=\thf\partial_{x_i}\partial_{x_j}\varphi_k(t,x,\eta).
\end{equation*}
We note $|V(t,x,h,\eta)-\eta|\le \frac 18|\eta|$ by \eqref{cond0}, and for $|\alpha|+|\beta|+m+|\gamma|\ge 1$ Corollary \ref{phikest''} yields
\begin{equation}\label{est:onehalf}
|\partial_\eta^\alpha\partial_x^\beta\partial_t^m\partial_h^\gamma V(t,x,h,\eta)|\le 
C_{\alpha,\beta,m,\gamma}\ts 2^{\frac k2(|\alpha|+|\beta|+m+|\gamma|-1)}\,|\eta|^{1-|\alpha|}.
\end{equation}
We make the change of variables
$y\rightarrow y+h$, followed by $\zeta\rightarrow V(t,x,h,\eta)+\zeta$, to write the integral term in \eqref{bkform} as
\begin{equation}\label{pkint}
\int e^{-i\la h,\zeta\ra}
\tp_k\bigl(x,V(t,x,h,\eta)+\zeta\bigr)\,dh\,d\zeta.
\end{equation}
We then decompose \eqref{pkint} using a smooth cutoff $\chi$, supported in $|\zeta|\le 2$, with
$\chi(\zeta)=1$ for $|\zeta|\le 1$. Specifically, we write
$$
1=\chi(2^{-k+4}\zeta)\bigl(1-\chi(h)\bigr)+\bigl(1-\chi(2^{-k+4}\zeta)\bigr)+\chi(h)\chi(2^{-k+4}\zeta).
$$
Since $\tp_k\in S^1_{1,\frac 12}$, the estimates \eqref{est:onehalf} imply that if $|\eta|\in\bigl[\frac 34\ts 2^{k},\frac 43\ts 2^{k+1}\bigr]$,
\begin{multline}\label{est:pkhalf}
\bigl|\partial_\zeta^\theta\partial_\eta^\alpha\partial_x^\beta\partial_t^m\partial_h^\gamma 
\tp_k\bigl(x,V(t,x,h,\eta)+\zeta\bigr)\ts\chi(2^{-k+4}\zeta)\bigr|\\
\le C_{\alpha,\beta,m,\gamma,\theta}\,
\begin{cases}
2^{k(1-|\alpha|-|\theta|)}\,2^{\frac k2(|\alpha|+|\beta|+m+|\gamma|-1)},&
|\alpha|+|\beta|+m+|\gamma| \ge 1,\\ 
2^{k(1-|\theta|)},& |\alpha|+|\beta|+m+|\gamma|=0.\rule{0pt}{13pt}
\end{cases}
\end{multline}

Consider first the term $r_1(t,x,\eta)$, defined by
\begin{multline*}
\int e^{-i\la h,\zeta\ra}
\tp_k\bigl(x,V(t,x,h,\eta)+\zeta\bigr)\,\chi(2^{-k+4}\zeta)\bigl(1-\chi(h)\bigr)\,dh\,d\zeta\\
=
\int e^{-i\la h,\zeta\ra}
\Delta_\zeta^N\Bigl(\tp_k\bigl(x,V(t,x,h,\eta)+\zeta\bigr)\,\chi(2^{-k+4}\zeta)\Bigr)\\
\times\bigl(1-\chi(h)\bigr)|h|^{-2N}\,dh\,d\zeta.
\end{multline*}
The estimates \eqref{est:pkhalf} show that the integrand is bounded by $2^{k(1-2N)}|h|^{-2N}$,
and it is supported where $|\zeta|\le 2^{k-3}$ and $|h|>1$. 
Similar estimates on its derivatives in $(x,\eta)$ yield that, for all $N$,
\begin{equation}\label{eqn:rkest}
|\partial_\eta^\alpha\partial_x^\beta\partial_t^m r_1(t,x,\eta)|\le C_{N,\alpha,\beta,m}\,2^{-kN},
\qquad |\eta|\in\bigl[\tfrac 34\ts 2^{k},\tfrac 43\ts 2^{k+1}\bigr].
\end{equation}

Next consider the term $r_2(t,x,\eta)$, defined by the integral
\begin{multline*}
\int e^{-i\la h,\zeta\ra}
\tp_k(x,V(t,x,h,\eta)+\zeta)\bigl(1-\chi(2^{-k+4}\zeta)\bigr)\,dh\,d\zeta=\\
\int e^{-i\la h,\zeta\ra}
\bigl(1-\Delta_\zeta\bigr)^n\Delta_h^N\Bigl(\tp_k(x,V(t,x,h,\eta)+\zeta)
\bigl(1-\chi(2^{-k+4}\zeta)\bigr)|\zeta|^{-2N}\Bigr)\\
\times(1+|h|^2)^{-n}\,dh\,d\zeta.
\end{multline*}
The estimates \eqref{est:pkhalf} show that the integrand is bounded by a constant times
$2^{k(N+\hf)}|\zeta|^{-2N}\bigl(1+|h^2|\bigr)^{-n}$, and it is supported where $|\zeta|\ge 2^{k-4}$. It follows that
$r_2(t,x,\eta)$ also satisfies the estimates \eqref{eqn:rkest}.

Thus, up to rapidly decreasing terms, the symbol $b_k(t,x,\eta)$ is equal to
$$
i\partial_t\varphi_k(t,x,\eta)+\frac{i}{(2\pi)^d}\int e^{-i\la h,\zeta\ra}
\tp_k(x,V(t,x,h,\eta)+\zeta)\,\chi(2^{-k+4}\zeta)\,\chi(h)\,dh\,d\zeta.
$$
We take a Taylor expansion in $\zeta$ of $\tp_k$ about $\zeta=0$ to write the integral as
\begin{multline*}
\sum_{|\gamma|< 2N}\frac 1{\gamma !}\int e^{-i\la h,\zeta\ra}
D_h^\gamma\Bigl((\partial_\xi^\gamma \tp_k)(x,V(t,x,h,\eta))\ts\chi(h)\Bigr)\ts\chi(2^{-k+4}\zeta)\,dh\,d\zeta\\
+r(t,x,\eta),
\end{multline*}
where $r(t,x,\eta)$ is given by
\begin{multline*}
\sum_{|\gamma|=2N}\int_0^1(1-s)^{N-1}\!\!\int e^{-i\la h,\zeta\ra}D_h^\gamma 
\Bigl((\partial_\xi^\gamma \tp_k)(x,V(t,x,h,\eta)+s\zeta)\ts\chi(h)\Bigr)\\
\times\chi(2^{-k+4}\zeta)\,dh\,d\zeta\,ds.
\end{multline*}
The estimates \eqref{est:pkhalf} show that $|\partial_x^\beta\partial_\eta^\alpha r(t,x,\eta)|\le C_{N,\alpha,\beta}\,2^{k(d+1-\frac 12|\alpha|+\frac12|\beta|-N)}$, provided that $|\eta|\in\bigl[\frac 34\ts 2^{k},\frac 43\ts 2^{k+1}\bigr]$.

To handle the terms with $|\gamma|<2N$, let $\phi(h)=2^{-4d}\widehat{\chi}(2^{-4}h)$, which has integral $(2\pi)^n$ and vanishing moments of all non-zero order, and write the $\gamma$ term in the sum as
$$
\int e^{-i\la h,\zeta\ra}
D_h^\gamma\Bigl((\partial_\xi^\gamma \tp_k)(x,V(t,x,h,\eta))\,\chi(h)\Bigr)\,2^{nk}\phi(2^kh)\,dh\,d\zeta.
$$
We Taylor expand $\tp_k\bigl(x,V(t,x,h,\eta)\bigr)\chi(h)$ to order $N$ about $h=0$. The $N$-th order remainder term will lead to a term bounded by $2^{k(1-\frac 12|\gamma|-N)}$, with similar estimates on derivatives in $(x,\eta)$. All terms with $h^\theta$ with $\theta\ne 0$ integrate to $0$ by the moment condition. Therefore, since $\partial_t\varphi_k(t,x,\eta)=-p_k(x,\nabla_x\varphi_k(t,x,\eta))$, we can write $b_k(t,x,\eta)$ as $r(t,x,\eta)$ plus
\begin{equation}\label{eqn:bkasympt}
i\Bigl(-p_k(x,\nabla_x\varphi_k(t,x,\eta))+\!\!\sum_{|\gamma|<2N}\frac 1{\gamma!}D_h^\gamma
(\partial_\xi^\gamma \tp_k)(x,V(t,x,h,\eta))\bigr|_{h=0}\Bigr)
\end{equation}
If $|\eta|\in\bigl[\frac 34\ts 2^{k},\frac 43\ts 2^{k+1}\bigr]$ then $(\psi_{k-1}+\psi_k+\psi_{k+1})(\nabla_x\varphi_k(t,x,\eta))=1$,
so by \eqref{pkerror} the $\gamma=0$ term combines with $-p_k(t,x,\nabla_x\varphi_k(t,x,\eta))$ to give
\begin{multline}\label{pkterm}
\sum_{j=k-1}^{k+1}(p_k-p_j)(x,\nabla_x\varphi_k(t,x,\eta))\,\psi_j(\nabla_x\varphi_k(t,x,\eta))\\
+
\sum_{j=k-1}^{k+1}q_j(x,\nabla_x\varphi_k(t,x,\eta))\,\beta_j(\nabla_x\varphi_k(t,x,\eta)).
\end{multline}
We will estimate this term similar to the term $|\gamma|=1$, using the following estimate, which is a consequence of \eqref{gkdiffbounds},
\begin{equation}\label{pktermest}
|\partial_x^\beta\partial_\xi^\alpha (p_k-p_j)(x,\xi)|\le C_{\alpha,\beta}\,2^{\ts k(\frac 12 |\beta|-1)}\,|\xi|^{1-|\alpha|}.
\end{equation}
The same estimates hold for the term $q_j(x,\xi)\in S^0_{1,\frac 12}$ when $|\xi|\approx 2^k$.

We now examine the terms in the sum when $|\gamma|\ge 1$. Observe that
$$
\partial_h^\theta V(t,x,h,\eta)\bigr|_{h=0}=\frac{1}{1+|\theta|}\,\partial_x^\theta\nabla_x\varphi_k(t,x,\eta).
$$
The $\gamma$ term in \eqref{eqn:bkasympt} is then a finite linear combination of terms of the form
$$
(\partial_\xi^{\gamma+\sigma}\tp_k)\bigl(x,\nabla_x\varphi_k(t,x,\eta)\bigr)\,
\bigl(\partial_x^{\theta_1}\nabla_x\varphi_k(t,x,\eta)\bigr)\,\cdots\,\bigl(\partial_x^{\theta_l}\nabla_x\varphi_k(t,x,\eta)\bigr),
$$
where $\theta_1+\cdots+\theta_l=\gamma$, each $\theta_i\ne 0$, and $l=|\sigma|\ge 1$.

By Corollary \ref{phikest''}, when $\theta_i\ne 0$, $|\eta|\in\bigl[\frac 34\ts 2^{k},\frac 43\ts 2^{k+1}\bigr]$, and $2^{-k}\le t\le 1$,
\begin{multline*}
\bigl|\la \eta,\partial_\eta\ra^j\partial_\eta^\alpha\partial_x^\beta\partial_t^m
\bigl(\partial_x^{\theta_i}\nabla_x\varphi_k(t,x,\eta)\bigr)\bigr|\\
\le C_{j,\alpha,\beta,m,\theta}\,
2^{\frac k2(|\theta_i|+1)}\,\tkhf^{|\alpha|}\,2^{\frac k2 (|\beta|+m)}.
\end{multline*}
A recursion argument and \eqref{eqn:pkest} then show that, for $2^{-k}\le t\le 1$,
\begin{multline*}
\Bigl|\la \eta,\partial_\eta\ra^j\partial_\eta^\alpha\partial_x^\beta\partial_t^m
\Bigl((\partial_\xi^{\gamma+\sigma}\tp_k)\bigl(x,\nabla_x\varphi_k\bigr)\,
\partial_x^{\theta_1}\nabla_x\varphi_k\,\cdots\,\partial_x^{\theta_l}\nabla_x\varphi_k
\Bigr)\Bigr|\\
\le C_{j,\alpha,\beta,m,\gamma,\sigma}\,
2^{\frac k2(2-|\gamma|-|\sigma|)}\,\tkhf^{|\alpha|}\,2^{\frac k2 (|\beta|+m)}.
\end{multline*}
The expression for $b_k(t,x,\eta)$ involves an asymptotic sum over $|\gamma|\ge 1$,  where also $|\sigma|\ge 1$ in all terms, and the sum thus satisfies the statement of the theorem in case $2^{-k}\le t\le 1$. The estimate for $0\le t\le 2^{-k}$ follows similarly.

It remains to consider the term \eqref{pkterm}. Using \eqref{pktermest} and a similar recursion argument, we obtain for the case $2^{-k}\le t\le 1$, 
\begin{multline*}
\Bigl|\la \eta,\partial_\eta\ra^j\partial_\eta^\alpha\partial_x^\beta\partial_t^m
\Bigl(\sum_{i=k-1}^{k+1}(p_k-p_i)\bigl(x,\nabla_x\varphi_k(t,x,\eta)\bigr)\psi_i(\nabla_x\varphi_k(t,x,\eta))\Bigr)\Bigr|\\
\le C_{j,\alpha,\beta,m}\,\tkhf^{|\alpha|}\,2^{\frac k2 (|\beta|+m)},
\end{multline*}
and the proof for $0\le t\le 2^{-k}$ is similar.
\end{proof}

Repeating the proof of Corollary \ref{phikest'}, we obtain the following.
\begin{corollary}\label{cor:linsymbol}
The following estimates hold for $\eta\in \Omega^\nu_{k,t}$,
$$
\bigl|\la\nu,\partial_\eta\ra^j\partial_\eta^\alpha\partial_x^\beta\partial_t^m
b_k(t,x,\eta)\bigr|\le C_{j,\alpha,\beta,m}\,2^{-kj}\,\tkhf^{|\alpha|}\,2^{\frac k2 (|\beta|+m)}.
$$
For $0\le t\le 2^{-k}$ these hold for $\eta$ in the dyadic shell $\frac 23 2^{k-1}\le|\eta|\le \frac 32 2^{k+2}$ if $t$ is replaced by $2^{-k}$ on the right hand side.
\end{corollary}

\section{Energy flow estimates}\label{sec:energyflow}
In this section we construct the exact wave group $\exp(-itP)$ via a convergent iteration based on the approximate wave group
\begin{equation}\label{wdef}
W(t)=\sum_{k=2}^\infty W_k(t)+\psi_0(D)+\psi_1(D).
\end{equation}
Recall \eqref{wkdef}--\eqref{bkdef} that $(\partial_t+i\tP_k)W_k(t)=B_k(t)$ is of order $0$. To show that $(\partial_t+iP)W(t)$ is of order $0$ we will show that, for $|t|\le 1$, $W_k(t)f$ remains localized in frequency to an appropriate dyadic shell at scale $2^k$, modulo smoothing errors. This will yield
$$
(\partial_t+iP)W(t)=\sum_{k=2}^\infty B_k(t)+R(t).
$$
with $R(t)$ a smoothing error. Denoting the right hand side by $B(t)$, since $W(0)=\I\,$ the wave group can be obtained by convergent iteration of $W(t)$ and $B(t)$, using Sobolev mapping bounds for both. Dispersive estimates will then depend on showing that composite terms
$$
W(t-s_1)B(s_1-s_2)\cdots B(s_{n-1}-s_n) B(s_n),\quad t\ge s_1\ge\cdots\ge s_n\ge 0,
$$
have similar microlocal mapping properties to $W(t)$ and $B(t)$. For a fixed $n$ we could show that this term has an oscillatory integral representation similar to that for $B(t)$, but at frequency scale $2^k$ we will need consider $n$ up to $n\sim 2^{k\sigma}$, for some $\sigma>0$. To prove preservation of dyadic localization of the energy we then need to microlocalize the energy mapping of each term $B_k(s)$ to within $2^{k(1-\sigma)}$ of the Hamiltonian flow. For convenience we fix $\sigma=\frac 14$, though any $\sigma\in (0,\frac 12)$ would work. We then consider frequency cutoffs with symbols $a(\eta)\in S^0_{\frac 34}$, that is
\begin{equation}\label{psikjest}
\bigl|\partial_\eta^\alpha a(\eta)\bigr|\le C_\alpha\,2^{-\frac 34k|\alpha|}\,,\qquad \supp(a)\subset \bigl\{\eta:\tfrac 45\,2^{k-1}<|\eta|<\tfrac 54\,2^{k+2}\bigr\}.
\end{equation}

Given any compact set $K\subset \{\eta:\tfrac 78\,2^{k-1}<|\eta|<\tfrac 87\,2^{k+2}\bigr\}$ and 
$\delta>0$, there exists a cutoff $a$ satisfying \eqref{psikjest} such that $\supp(a)$ is contained in the 
$\delta2^{\frac{3k}4}$ neighborhood of $K$ and $a=1$ on the $\frac 12\ts\delta\ts 2^{\frac{3k}4}$ neighborhood of 
$K$, and such that the constants $C_\alpha$ depend on $\delta$ but are independent of $K$. Such
an $a(\eta)$ can be obtained, for example, by convolving the support function of the $\frac 34\ts\delta\ts 2^{\frac{3k}4}$ neighborhood of $K$ with an approximation to the identity supported in the $\frac 18\delta$ ball.

\begin{lemma}\label{lem:eprop}
Suppose that $a_1$ and $a_2$ are cutoffs satisfying \eqref{psikjest}, and let
$K$ be the projection onto $\eta$ of the image of $\Rd\times\supp(a_1)$ under the Hamiltonian flow of $p_k$ at time $t$. Assume that
$a_2=1$ on the $\delta\,2^{\frac 34k}$ neighborhood of $K$. Then for all $N$,
$$
\bigl\|\bigl(1-a_2(D)\bigr)B_k(t)a_1(D)f\bigr\|_{H^N}\le C_N\,2^{-kN}\,\|f\|_{H^{-N}},
$$
where the constant $C_N$ depends only on $N$, the constants $C_\alpha$ in \eqref{psikjest}, and $\delta$. The same holds with $B_k(t)$ replaced by $W_k(t)$.
\end{lemma}
\begin{proof} We prove this using a modification of the C\'ordoba-Fefferman wave packet transform introduced in \cite{CF}. We use the particular transform from \cite{Sm0}, which is based on a Schwartz function with Fourier transform of compact support, instead of a Gaussian. Fix $g$ a radial, real Schwartz function with $\|g\|_{L^2}=(2\pi)^{\frac d2}$ and $\supp(\hat g)\subset\{|\zeta|<\frac 14\}$, and set
$$
g_{x,\xi}(z)=2^{\frac{kd}4}\,e^{i\la \xi,z-x\ra}g\bigl(2^{\frac k2}(z-x)\bigr).
$$
For $f\in L^2(\Rd)$ define
$$
(T_kf)(x,\xi)=\int f(z)\,\overline{g_{x,\xi}(z)}\,dz.
$$
Then $T_k$ is an isometry, with adjoint given by
$$
\bigl(T_k^*F\bigr)(z)=\int F(x,\xi)\,g_{x,\xi}(z)\,dx\,d\xi.
$$
Since $|\eta|\approx 2^k$ on the support of $a_1(\eta)$, it suffices to show that for all $N$,
$$
\|T_k\,\la D\ra^N\bigl(1-a_2(D)\bigr)B_k(t)a_1(D)T_k^*F\|_{L^2(\re^{2d})}\le C_N\,2^{-2kN}\|F\|_{L^2(\re^{2d})}.
$$
The operator on the left is given by the following integral kernel,
$$
K_k(t,x',\xi';x,\xi)=\int_{\Rd} \bigl(B_k(t)\,a_1(D)\,g_{x,\xi}\bigr)(z)\,\overline{\la D\ra^N\bigl(1-a_2(D)\bigr)g_{x',\xi'}(z)}\,dz.
$$
Let $(x_t,\xi_t)=\chi_t(x,\xi)$, with $\chi_t$ the Hamiltonian flow for $p_k$. 
A simple integration by parts argument, using Lemma \ref{lem:FBI} below, shows that
for all $N$,
\begin{align}\label{cdfkernel}
|K_k(t,x',\xi';x,\xi)|&\le C_N2^{kN}\bigl(1+2^{\frac k2}|x'-x_t|+2^{-\frac k2}|\xi'-\xi_t|\bigr)^{-8N-2d-1}\\
&\le C'_N2^{-kN}\bigl(1+2^{\frac k2}|x'-x_t|+2^{-\frac k2}|\xi'-\xi_t|\bigr)^{-2d-1},\nonumber
\end{align}
where in deducing the second bound we used that the integrand vanishes unless 
$|\xi'-\xi_t|\ge \delta\,2^{\frac {3k}4}-2^{\frac k2}$. 
The desired $L^2$ bound then follows by the Schur test, using the fact that 
$(x,\xi)\rightarrow (x_t,\xi_t)$ is a volume preserving diffeomorphism, which is 
homogeneous in $\xi$ and bilipshitz on the cotangent bundle (uniformly over $k$).
\end{proof}

\begin{lemma}\label{lem:FBI}
Let $f_{x,\xi}(y)=2^{\frac{kd}4}e^{i\la\xi,y-x\ra}f\bigl(2^{\frac k2}(y-x)\bigr)$.
Assume that $f$ is a Schwartz function and $|\xi|\in[2^{k-1},2^{k+2}]$, and let $(x_t,\xi_t)=\chi_t(x,\xi)$. Then $$
\bigl(B_k(t)f_{x,\xi}\bigr)(z)=2^{\frac{kd}4}e^{i\la \xi_t,z-x_t\ra}h\bigl(t,2^{\frac k2}(z-x_t)\bigr),
$$
where for all $N,j,\gamma$,
$$
\bigl|\partial_t^j\partial_z^\gamma h(t,z)\bigr|\le 
C_{N,j,\gamma}\,2^{\frac k2 j}\bigl(1+|z|)^{-N}.
$$
For each $N,j,\gamma$, the constant $C_{N,j,\gamma}$ is bounded by a Schwartz seminorm of $f$, but is uniform over $k,x,\xi$.
\end{lemma}
\begin{proof} Up to a factor of $(2\pi)^d$, the function $h(t,z)$ is given by the integral
$$
\int e^{i\Phi(t,z,\eta)}\,b_k(t,x_t+2^{-\frac k2}z,\xi+2^{\frac k2}\eta)\,
\psi_k(\xi+2^{\frac k2}\eta)\,\hatf(\eta)\,d\eta,
$$
where
$$
\Phi(t,z,\eta)=\varphi_k(t,x_t+2^{-\frac k2}z,\xi+2^{\frac k2}\eta)-\la x,\xi+2^{\frac k2}\eta\ra-\la\xi_t,2^{-\frac k2}z\ra,
$$
Since $\varphi_k$ is the homogeneous generating function for $\chi_t$, this equals
\begin{multline*}
\varphi_k(t,x_t+2^{-\frac k2}z,\xi+2^{\frac k2}\eta)-\varphi_k(t,x_t,\xi)\\
-2^{\frac k2}\eta\cdot(\nabla_\eta\varphi_k)(t,x_t,\xi)
-2^{-\frac k2}z\cdot(\nabla_x\varphi_k)(t,x_t,\xi).
\end{multline*}
By Corollary \ref{cor:geodflow}, Theorem \ref{phikest}, Corollary \ref{phikest''}, and \eqref{dxdetaphi}, the following estimates hold on the support of the integrand,
$$
|\partial_t^j\partial_z^\beta\partial_\eta^\alpha\Phi(z,\eta)|\le C_{\alpha,\beta,j}\,2^{\frac k2 j}\quad\text{if}\quad |\alpha|+|\beta|\ge 2.
$$
As $\Phi$ vanishes to second order at $z=\eta=0$, then on the region of integration
\begin{align*}
|\partial_t^j\Phi(z,\eta)|&\le C_j\,2^{\frac k2 j}(\ts 1+|z|+|\eta|\ts)^2,\\
|\partial_t^j\nabla_{z,\eta}\Phi(z,\eta)|&\le C_j\,2^{\frac k2 j}(\ts 1+|z|+|\eta|\ts).
\end{align*}
By Theorem \ref{thm:yest} we have $|\ts\nabla_z\nabla_\eta\Phi(z,\eta)-\I\ts|\lesssim c_d$, and since $|\nabla_\eta^2\Phi(y,\eta)|\le C$, we deduce that 
$|z|\le C\bigl( \ts|\nabla_\eta\Phi(z,\eta)|+|\eta|\ts\bigr)$ and thus
$$
\frac{1}{1+|\nabla_\eta\Phi(y,\eta)|^2}\le C\,\frac{1+|\eta|^2}{1+|z|^2}.
$$
By Corollary \ref{cor:geodflow}, Theorem \ref{thm:bkest}, and \eqref{psikjest}, since $|\xi+2^{\frac k2}\eta|\approx 2^k$  we have
\begin{multline*}
\Bigl|\partial_t^j\partial_z^\beta\partial_\eta^\alpha\Bigl( 
b_k(t,x_t+2^{-\frac k2}z,\xi+2^{\frac k2}\eta)\,
\psi_k(\xi+2^{\frac k2}\eta)\,\hatf(\eta)\Bigr)\Bigr|\\
\le C_{N,\alpha,\beta,j}\,2^{\frac k2 j}\bigl(1+|\eta|\bigr)^{-N}.
\end{multline*}
Integrating by parts with respect to the vector field
$$
L=\frac{1-i\ts\nabla_\eta\Phi(z,\eta)\cdot\nabla_\eta}{1+|\nabla_\eta\Phi(z,\eta)|^2\rule{0pt}{11pt}}
$$
then leads to the bounds on $\partial_t^j\partial_z^\gamma h$ in the statement.
\end{proof}

The same argument also shows that the kernel of $T_kB_k(t)T_k^*$ satisfies \eqref{cdfkernel} with $N=0$, and in particular $B_k(t)$ is bounded on $L^2(\Rd)$, uniformly over $k$ and $|t|\le 1$.
By applying Lemma \ref{lem:eprop} with $a_1(\eta)=1$ on the support of $\beta_k(\eta)$, and $a_2(\eta)$ supported in the annulus $|\eta|\in [2^{k-1},2^{k+2}]$, we then obtain the following by an orthogonality argument.
\begin{lemma}
For all $s\in\re$ we have $\bigl\|\sum_{k=2}^\infty B_k(t)f\bigr\|_{H^s}\le C_s\|f\|_{H^s}$, uniformly over $|t|\le 1$.
\end{lemma}

We can now show that $W(t)$ defined above is an approximate evolution operator for $P$.
\begin{lemma}
Let $W(t)$ be defined by \eqref{wdef}. Then
$$
\bigl(\partial_t+iP\bigr)W(t)=\sum_{k=2}^\infty B_k(t)+R(t),
$$
where $R(t)$ is an integral kernel operator with kernel $K$ satisfying 
$$
|\partial_x^\alpha\partial_y^\beta K(t,x,y)|\le C_{N,\alpha,\beta}\,(1+|x-y|)^{-N}.
$$
In particular, $\|R(t)f\|_{H^N}\le C_N\,\|f\|_{H^{-N}}$ for all $N$, uniformly over $|t|\le 1$.
\end{lemma}
\begin{proof}
We take $a_k(\eta)$ supported in $\bigl\{\frac 45\ts 2^k<|\eta|<\frac 54\ts 2^{k+1}\bigr\}$, and equal to $1$ where $\bigl\{\frac 78\ts 2^k<|\eta|<\frac 87\ts 2^{k+1}\bigr\}$, satisfying \eqref{psikjest} with constants $C_\alpha$ independent of $k$. For $c_d$ small enough, the condition of Lemma \ref{lem:eprop} with $a_1=\psi_k$ and $a_2=a_k$ is satisfied for all $t$ with $|t|\le 1$. We need show that the operator
$$
\sum_{k=2}^\infty(P-\tP_k)W_k(t)
=\sum_{k=2}^\infty(P-\tP_k)\bigl(1-a_k(D)\bigr)W_k(t).
$$
satisfies the conditions for $R(t)$, since $P\circ\bigl(\psi_0(D)+\psi_1(D)\bigr)$ does.
It suffices to show we can write
\begin{equation*}
(P-\tP_k)W_k(t)=\text{Op}(R_k)\circ\psi_k(D)
\end{equation*}
with $R_k(t,x,y)$ an integral kernel satisfying, for all $N$,
$$
|\partial_x^\alpha\partial_y^\beta R_k(t,x,y)|\le C_{N,\alpha,\beta}\,2^{-kN}(1+|x-y|)^{-N}.
$$
Observe that $R_k=T_k^*\,\text{Op}(K_k)\,T_k$, where $K_k$ satisfies \eqref{cdfkernel}, and vanishes unless $|\xi|\in[\frac 45 2^k,\frac 54 2^{k+1}]$ and $|\xi'|\notin [\frac 78 2^k,\frac 87 2^{k+1}]$. For $c_d$ small this implies $|\xi_t|\in[\frac 56 2^k,\frac 65 2^{k+1}]$, hence $|\xi'-\xi_t|\ge 2^{k-4}$. Since $|x-x_t|\le 2$ for $|t|\le 1$, we have for all $N$
$$
|K_k(t,x',\xi';x,\xi)|\le C_N2^{-kN}\bigl(1+2^{-\frac k2}|\xi'|\bigr)^{-N}\bigl(1+|x-x'|\bigr)^{-N}.
$$
The operator $T_k$ is given by a kernel satisfying for all $N$
$$
|\partial_y^\alpha T_k(x,\xi;y)|\le C_{N,\alpha}2^{k\bigl(\frac{|\alpha|}2+\frac d4\bigr)}(1+2^{\frac k2}|x-y|)^{-N}.
$$
Since the volume of integration in $\xi$ is less than $C_d\,2^{kd}$, the estimate for $R_k(t,x,y)$ follows by composition.
\end{proof}

We now write
$$
\int R_k(t,x,y)\,(\psi_k(D)f)(y)\,dy=\frac 1{(2\pi)^n}\int e^{i\varphi_k(t,x,\eta)}r_k(t,x,\eta)\,\psi_k(\eta)\hatf(\eta)\,d\eta,
$$
with
$$
r_k(t,x,\eta)=e^{-i\varphi_k(t,x,\eta)}\int R_k(t,x,y)\,e^{-i\la y,\eta\ra}\,dy.
$$
Then for all $N$
$$
|\partial_x^\beta\partial_\eta^\alpha r_k(t,x,\eta)|\le C_{\alpha,\beta,N}2^{-kN},\qquad 2^{k-1}\le|\eta|\le 2^{k+2},
$$
and we can incorporate $r_k$ into $b_k$, and hence $R_k(t)$ into $B_k(t)$. Thus
we can write
$$
(\partial_t+iP)W(t)=\sum_{k=2}^\infty B_k(t)+P\circ\bigl(\psi_0(D)+\psi_1(D))\;\equiv B(t).
$$


We now can generate the exact wave group $E(t)$ for $\partial_t+iP$ by iteration,
$$
E(t)=W(t)-\int_0^t W(t-s)B(s)\,ds+\int_0^t\int_0^s W(t-s)B(s-r)B(r)\,dr\,ds-\cdots
$$
To write the iteration more concisely,
let $\Lambda^m\subset\re_+^{m+1}$ be the $m$-simplex, consisting of $\bfr=(r_1,\ldots,r_{m+1})$ with $r_j> 0$ for all $j$, and with $r_1+\cdots+r_{m+1}=1$. Let $d\bfr$ be the measure on $\Lambda^m$ induced by projection onto $(r_1,\ldots,r_m)$. Then
\begin{equation}\label{Esumdef}
E(t)=\sum_{m=0}^\infty\; (-t)^m\int_{\Lambda^m}W(tr_{m+1})B(tr_m)\cdots B(tr_1)\,d\bfr\,.
\end{equation}
If $C_s$ is an upper bound for the $H^s(\Rd)$ operator norm of both $W(t)$ and $B(t)$ for all $|t|\le 1$, then
the $m$-th term has $H^s(\Rd)$ operator norm at most $C_s^{m+1} t^m/m!$, and the following theorem holds.
\begin{theorem}
The expansion \eqref{Esumdef} converges uniformly over $|t|\le 1$, in the operator norm topology on $H^s(\Rd)$ for all $s\in\re$. The limit $E(t)$ is a one parameter group of $L^2$-unitary operators, and for $f\in H^s$, $F\in L^1([-1,1],H^s)$, the solution to
$(\partial_t +iP)u=F$, $u(0,\cdot)=f$ is given by $$u(t,\cdot)=E(t)f+\int_0^t E(t-s)F(s,\cdot)\,ds.$$
\end{theorem}


Our next two results show that if we localize $E(t)$ on the right to frequency scale $2^k$, then modulo a smoothing operator error one can localize each of the terms $W(tr_j)$ and $B(tr_j)$ in \eqref{Esumdef} to frequencies of scale $2^k$.
We use the notation $\tilde\psi_k=\psi_{k-1}+\psi_k+\psi_{k+1}$, and define
\begin{equation}\label{twkdef}
\begin{split}
\tW_k(t)&=\;\tpsi_k(D)(W_{k-1}+W_k+W_{k+1})(t),\\
\tB_k(t)&=\;\tpsi_k(D)(B_{k-1}+B_k+B_{k+1})(t). 
\end{split}
\end{equation}

\begin{lemma}\label{energyloc'}
If $m+1\le 2^{\frac k4}$, then for all $N\ge 0$ the operator
\begin{multline*}
R_{k,\bfr}(t)=W(tr_{m+1})B(tr_m)\cdots B(tr_1)\psi_k(D)\\
-\tW_k(tr_{m+1})\tB_k(tr_m)\cdots \tB_k(tr_1)\psi_k(D)
\end{multline*}
satisfies the following, with constant $C_N$ independent of $m$, $t$, $k$, and $\bfr$, 
$$
\|R_{k,\bfr}(t)f\|_{H^N}\le\,C_{N}\,2^{-kN}\,\|f\|_{H^{-N}}.
$$
\end{lemma}
\begin{proof}
Fix $t$ and $\bfr$, and without loss of generality assume $t\ge 0$.
We introduce a family of intermediate cutoffs $\psi_{k,j}(D)$ for $1\le j\le m$, which depend on $t\bfr$.
Define points $\frac{10}9\le p_{j-1}'<p_j<p'_j\le \frac 54$ as follows. Take $c_0$ and $c_1$ such that $p_0=e^{c_0}=\frac{10}9$, and
$e^{c_0+2c_1}=\frac 54$. For $j\ge 0$ we set
$$
p_j=e^{c_0+c_1(r_1+\cdots +r_j)t+c_1j2^{-\frac k4}},\qquad 
p'_j=e^{c_0+c_1(r_1+\cdots +r_j)t+c_1(j+\frac 12)2^{-\frac k4}}
$$
Thus $\psi_k$ is supported where $|\eta|\in[p_0^{-1}2^k,p_02^{k+1}]$, and
$\tpsi_k(\eta)=1$ on the set $\bigl\{\eta:|\eta|\in[\, {p'_m}^{-1}2^k,p'_m2^{k+1}]\bigr\}$.
Also,
$$
|p'_j-p_j|\ge \tfrac 12\,c_1 2^{-\frac k4}\,,\qquad
|p_{j+1}-p'_j|\ge c_1\ts r_{j+1}t+\tfrac 12\,c_1 2^{-\frac k4}\,.
$$

Let $\psi_{k,0}=\psi_k$. By the comments following \eqref{psikjest} we can construct functions 
$\psi_{k,j}(\xi)$ for $j\ge 1$ that satisfy \eqref{psikjest}, with constants $C_\alpha$ that depend only on the dimension $d$, such that
\begin{align*}
\supp\bigl(\psi_{k,j}\bigr)\subset\bigl\{\eta:|\eta|\in [\, {p'_j}^{-1}2^k, p'_j 2^{k+1}\ts]\bigr\}&,\quad j\ge 0,\\
\psi_{k,j}(\eta)=1\;\;\,\text{if}\,\;\;|\eta|\in [\, p_j^{-1}2^k, p_j 2^{k+1}]&,\quad j\ge 1\,.\rule{0pt}{15pt}
\end{align*}
Let $c_d'=\sup_{x,\xi}\bigl(|\xi|^{-1}|\nabla_xp_k(x,\xi)|\bigr)\lesssim c_d$. 
Then for solutions to the Hamiltonian flow,
$$
\exp(-c_d'\, t\ts r_j)\,|\xi(s)|\le |\xi(s+tr_j)|\le \exp(c_d'\, t\ts r_j)\,|\xi(s)|.
$$
Then if $c'_d\le c_1$, the condition of Lemma \ref{lem:eprop} with $\delta=\frac 14 c_1$ is satisfied for $a_2=\psi_{k,j}$ and $a_1=\psi_{k,j-1}$.
Thus Lemma \ref{lem:eprop} yields, for $j\ge 1$,
$$
\|(1-\psi_{k,j}(D))B(tr_j)\psi_{k,j-1}(D)\|_{H^s\rightarrow H^s}\,\le\,C_{s,N}\,2^{-kN},\qquad\forall \,s,N.
$$
Since $B(t)\psi_{k,j}(D)=\tB_k(t)\psi_{k,j}(D)$, and the number of terms is at most $m\le 2^{\frac k4}$,
we can apply this repeatedly to write
\begin{multline*}
W(tr_{m+1})B(tr_m)\cdots B(tr_1)\psi_k(D)\\
=\tW_k(tr_{m+1})\psi_{k,m}(D)\tB_k(tr_m)\cdots \psi_{k,1}(D)\tB_k(tr_1)\psi_k(D)+R_{k,\bfr}(t)\,,
\end{multline*}
where $\;\|R_{k,\bfr}(t)\|_{H^s\rightarrow H^s}\,\le\,C_{s,N}\,2^{-Nk}$ for all $s,N.$
We then prove Lemma \ref{energyloc'} by observing that the same steps let us write
\begin{multline*}
\tW_k(tr_{m+1})\psi_{k,m}(D)\tB_k(tr_m)\cdots\psi_{k,1}(D)\tB_k(tr_1)\psi_k(D)\\
=
\tW_k(tr_{m+1})\tB_k(tr_m)\cdots \tB_k(tr_1)\psi_k(D)+R_{k,\bfr}(t)\,,
\end{multline*}
for a similar $R_{k,\bfr}(t)$. Since $R_{k,\bfr}(t)$ is localized on the right at frequency $2^k$, it follows that
$\|R_{k,\bfr}(t)\|_{H^{-N}\rightarrow H^N}\le C_N\,2^{-kN}$ for all $N$.
\end{proof}

\begin{corollary}\label{cor:eprop}
One can write
\begin{equation*}
E(t)=
\sum_{k=0}^\infty\sum_{m=0}^{2^{\frac k4}} \;(-t)^m
\int_{\Lambda^m}\tW_k(tr_{m+1})\tB_k(tr_m)\cdots\tB_k(tr_1)\psi_k(D)\,d\bfr+R(t),
\end{equation*}
where for all $N$ we have $\|R(t)f\|_{H^N}\le C_N\,\|f\|_{H^{-N}}$, uniformly over $|t|\le 1$.
\end{corollary}
\begin{proof}
Consider
$$
\sum_{m=2^{\frac k4}}^\infty\;(-t)^m
\int_{\Lambda^m}W(tr_{m+1})B(tr_m)\cdots B(tr_1)\psi_k(D)\,d\bfr.
$$
For $|t|\le 1$ and all $N$, the $H^N\rightarrow H^N$ operator norm of this sum is bounded by the sum
$\sum_{m\ge 2^{\frac k4}}C_N^{m+1}/m!\le C_N2^{-3kN}$. It is localized on the right at frequency $2^k$, and thus maps
$H^{-N}\rightarrow H^N$ with norm $\le C_N 2^{-kN}$.
\end{proof}

The arguments leading to Lemma \ref{energyloc'} apply equally well to conic localization. 
We take a finite partition of unity on $\Rd\backslash\{0\}$,
$$
1=\sum_{\omega\in\Xi}a_\omega(D)\,,\qquad\supp(a_\omega(\eta))\subset 
\Bigl\{\eta:\Bigl|\omega-\frac{\eta}{|\eta|}\Bigr|\le \frac 1{32}\Bigr\}.
$$
Let $\ta_\omega(\eta)$ be a smooth, homogeneous cutoff such that
$$
\ta_\omega(\eta)=1\;\;\;\text{if}\;\;\;\Bigl|\omega-\frac{\eta}{|\eta|}\Bigr|\le \frac 1{24},\qquad
\supp(\ta_\omega)\subset \Bigl\{\eta:\Bigl|\omega-\frac{\eta}{|\eta|}\Bigr|\le \frac 1{16}\Bigr\}.
$$
We define an angularly localized version of $\tW_k$, recalling \eqref{twkdef},
\begin{equation}\label{twomegakdef}
\begin{split}
\tW^\omega_k(t)=\ta_\omega(D)\tW_k(t)\ta_\omega(D),\\
\tB^\omega_k(t)=\ta_\omega(D)\tB_k(t)\ta_\omega(D).
\end{split}
\end{equation}
Then, 
\begin{multline}\label{eqn:energyloc'}
R^\omega_{k,\bfr}(t)=W(tr_{m+1})B(tr_m)\cdots B(tr_1)a_\omega(D)\psi_k(D)\\
-\tW_k^\omega(tr_{m+1})\tB_k^\omega(tr_m)\cdots \tB_k^\omega(tr_1)a_\omega(D)\psi_k(D)
\end{multline}
satisfies the conclusion of Lemma \ref{energyloc'},
and consequently, with $R(t)$ as in Corollary \ref{cor:eprop},
\begin{equation*}
E(t)=
\sum_{k=0}^\infty\sum_{\omega\in\Xi} \tE_k^\omega(t)+R(t),
\end{equation*}
where we define
\begin{equation}\label{angloc}
\tE_k^\omega(t)=
\sum_{m=0}^{2^{\frac k4}} (-t)^m
\int_{\Lambda^m}
\tW^\omega_k(tr_{m+1})\tB^\omega_k(tr_m)\cdots \tB^\omega_k(tr_1)a_\omega(D)\psi_k(D).
\end{equation}

\begin{lemma}\label{lem:tEkernel}
Let $f_{x,\xi}(y)=2^{\frac{kd}4}e^{i\la\xi,y-x\ra}f\bigl(2^{\frac k2}(y-x)\bigr)$.
Assume that $f$ is a Schwartz function, and let $(x_t,\xi_t)=\chi_t(x,\xi)$. Then one can write
$$
\bigl(\tE^\omega_k(t)f_{x,\xi}\bigr)(z)=2^{\frac{kd}4}e^{i\la \xi_t,z-x_t\ra}h\bigl(t,2^{\frac k2}(z-x_t)\bigr)\equiv h(t,\cdot)_{x_t,\xi_t},
$$
where for all $N$,
$$
\bigl|\partial_t^j\partial_z^\gamma h(t,z)\bigr|\le 
C_{N,j,\gamma}\,2^{\frac k2 j}\bigl(1+|z|)^{-N}.
$$
For each $N,j,\gamma$, the constant $C_{N,j,\gamma}$ is bounded by a Schwartz seminorm of $f$, but is uniform over $k,x,\xi$.
\end{lemma}
\begin{proof}
Let $K(s,y,\eta;x,\xi)$ denote the integral kernel of $T_k\tB_k^\omega(s)T_k^*$.
Following the proof of \eqref{cdfkernel} we can bound, with $C_N$ uniform over $k$ and $s$,
\begin{align*}
|K(s,y,\eta;x,\xi)|&\le C_N\bigl(1+2^{\frac k2}|y-x_s|+2^{-\frac k2}|\eta-\xi_s|\bigr)^{-N}.
\end{align*}
Furthermore, this kernel vanishes unless $2^{k-2}\le |\eta|,|\xi|\le 2^{k+3}$.
By the bilipschitz property of the Hamiltonian flow, for such $\eta,\xi$ we have
\begin{equation}\label{eqn:bilipham}
|y_s-x_s|+2^{-k}|\eta_s-\xi_s|\le A\ts|y-x|+2^{-k}A\ts|\eta-\xi|.
\end{equation}
For $N\ge 2d+1$, we then bound the kernel of 
$T_k\tB_k^\omega(tr_1)\tB_k^\omega(tr_2)T_k^*$ $=$ $\bigl(T_k\tB_k^\omega(tr_1)T_k^*\bigr)\bigl(T_k\tB_k^\omega(tr_2)T_k^*\bigr)$ by $C_N^2$ multiplied by the quantity
\begin{multline*}
\int \bigl(1+2^{\frac k2}|y-z_{tr_1}|+2^{-\frac k2}|\eta-\zeta_{tr_1}|\bigr)^{-N}
\bigl(1+2^{\frac k2}|z-x_{tr_2}|+2^{-\frac k2}|\zeta-\xi_{tr_2}|\bigr)^{-N}dz\,d\zeta\\
\le
A_N\bigl(1+2^{\frac k2}|y-x_{t(r_1+r_2)}|+2^{-\frac k2}|\eta-\xi_{t(r_1+r_2)}|\bigr)^{-N}.
\end{multline*}
Similarly, for $\bfr\in\Lambda^m$ the operator $T_k\tW_k^\omega(tr_{m+1})\tB_k^\omega(tr_m)\cdots\tB_k^\omega(tr_1)T_k^*$ has kernel bounded by
$$
C_N\bigl(A_NC_N\bigr)^m\bigl(1+2^{\frac k2}|y-x_t|+2^{-\frac k2}|\eta-\xi_t|\bigr)^{-N},
$$
and summing over $m$ gives the following bounds for the kernel of $T_k\tE_k^\omega(t)T_k^*$,
$$
|\tK_k^\omega(t,y,\eta;x,\xi)|\le C_N\,e^{tA_NC_N}\bigl(1+2^{\frac k2}|y-x_t|+2^{-\frac k2}|\eta-\xi_t|\bigr)^{-N}.
$$
Let $F=T_k(f_{x,\xi})$. Then
$$
|F(\barx,\barxi)|\le C_N\bigl(1+2^{\frac k2}|\barx-x|+2^{-\frac k2}|\barxi-\xi|\bigr)^{-N}.
$$
Then 
$(\tE^\omega_k(t)f_{x,\xi}\bigr)(z)$ is equal to
$$
2^{\frac{kd}4}\int \tK_k^\omega(t,y,\eta;\barx,\barxi)\,F(\barx,\barxi)\,e^{i\la \eta,z-y\ra}g\bigl(2^{\frac k2}(z-y)\bigr)\,d\barx\,d\barxi\,dy\,d\eta.
$$
The change of variables 
\begin{align*}
&(y,\eta)\rightarrow (x_t+2^{-\frac k2}y,\xi_t+2^{\frac k2}\eta)\\
&(\barx,\barxi)\rightarrow (x+2^{-\frac k2}\barx,\xi+2^{\frac k2}\barxi)
\end{align*}
shows that $h(t,z)=2^{-\frac{kd}4}e^{-i2^{-\frac k2}\la\xi_t,z\ra}\bigl(\tE^\omega_k(t)f_{x,\xi}\bigr)(x_t+2^{-\frac k2}z)$ is equal to
\begin{multline*}
\int 
\tK_k^\omega(t,x_t+2^{-\frac k2}y,\xi_t+2^{\frac k2}\eta;x+2^{-\frac k2}\barx,\xi+2^{\frac k2}\barxi)
F(x+2^{-\frac k2}\barx,\xi+2^{\frac k2}\barxi)\\
\times e^{-i2^{-\frac k2}\la\xi_t,y\ra}e^{i\la \eta,z-y\ra}g(z-y)\,d\barx\,d\barxi\,dy\,d\eta.
\end{multline*}
By the bilipschitz property \eqref{eqn:bilipham} of $\chi_t$ we have
\begin{multline*}
|y|+|\eta|\le A\ts|\barx|+A\ts|\barxi|
+2^{\frac k2}\bigl|(x_t+2^{-\frac k2}y)-(x+2^{-\frac k2}\barx)_t\bigr|\\
+2^{-\frac k2}\bigl|(\xi_t+2^{\frac k2}\eta)-(\xi+2^{\frac k2}\barxi)_t\bigr|,
\end{multline*}
and conclude that
\begin{multline*}
\bigl|\tK_k^\omega(t,x_t+2^{-\frac k2}y,\xi_t+2^{\frac k2}\eta;x+2^{-\frac k2}\barx,\xi+2^{\frac k2}\barxi)\bigr|\\
\le 
C_N\,\bigl(1+|y|+|\eta|\bigr)^{-N}\bigl(1+|\barx|+|\barxi|\bigr)^{N}.
\end{multline*}
Together with the bound
$$
\bigl|F(x+2^{-\frac k2}\barx,\xi+2^{\frac k2}\barxi)\bigr|\le
C_N\,\bigl(1+|\barx|+|\barxi|\bigr)^{-2N}, 
$$
this leads to the following estimates on $\partial_z^\gamma h(t,z)$, which is the case $j=0$,
\begin{equation}\label{prevresult}
|\partial_z^\gamma h(t,z)|\le C_{N,\gamma}\bigl(1+|z|\bigr)^{-N}.
\end{equation}
The constant $C_{N,\gamma}$ is seen to be bounded by a Schwartz seminorm of $f$, but uniform over $k,x,\xi$.

To handle time derivatives we proceed by induction, and assume the estimates on $\partial_t^i\partial_z^\gamma h(t,z)$ hold for $0\le i\le j$, and all $\gamma$. We write
$$
\tE_k^\omega(t)f_{x,\xi}=\tW_k^\omega(t)f_{x,\xi}+
\int_0^t \tW_k^\omega(t-s)\tE_k^\omega(s)f_{x,\xi}\,ds,
$$
where on the right the term $\tE_k^\omega(t)$, defined in \eqref{angloc}, has upper summation limit reduced by 1. This does not affect the validity of \eqref{prevresult}, since the proof of \eqref{prevresult} is done separately for each value of $m$. By Lemma \ref{lem:FBI} the first term satisfies the conditions of the statement, since the proof of that lemma works equally well for 
$B_k(t)$ replaced by $\tW_k^\omega(t)$. The desired estimates on $h$ are then a consequence of the following, for the given value of $j$ and all $\gamma$,
\begin{multline}\label{newhest}
\Bigl|\bigl(\partial_z-i\xi_t\bigr)^\gamma\bigl(\partial_t+ip_k(x_t,\xi_t)\bigr)^{j+1}\int_0^t \tW_k^\omega(t-s)\tE_k^\omega(s)f_{x,\xi}\,ds\Bigr|\\
\le C_{N,j+1,\gamma}2^{\frac{kd}4}2^{\frac k2 (|\gamma|+j+1)}\bigl(1+2^{\frac k2}|z-x_t|\bigr)^{-N}.
\end{multline}
This is seen by noting that
\begin{multline*}
e^{-i\la\xi_t,z-x_t\ra}\bigl(\partial_t+ip_k(x_t,\xi_t)\bigr) \Bigl(e^{i\la\xi_t,z-x_t\ra}h(t,2^{\frac k2}(z-x_t))\Bigr)
=(\partial_t h)(t,2^{\frac k2}(z-x_t))\\
-\Bigl(i\nabla_x p_k(x_t,\xi_t)\cdot(z-x_t)h+2^{\frac k2}(\nabla_\xi p_k)(x_t,\xi_t)\cdot\nabla_z h\Bigr)(t,2^{\frac k2}(z-x_t)).
\end{multline*}
The latter terms are controlled by the spatial derivative bounds on $h$, and their time derivatives controlled by the bounds
$$
|\partial_t^i (\nabla_x p_k)(x_t,\xi_t)|\le C_i \, 2^{k+\frac k2 i},
\qquad
|\partial_t^i(\nabla_\xi p_k)(x_t,\xi_t)|\le C_i\,2^{\frac k2 i},
$$
which follow by Corollary \ref{cor:geodflow} and \eqref{eqn:pkest}.

To establish \eqref{newhest} we expand
\begin{multline*}
\bigl(\partial_t+ip_k(x_t,\xi_t)\bigr)^{j+1}\int_0^t \tW_k^\omega(t-s)\tE_k^\omega(s)f_{x,\xi}\,ds\\
=
\sum_{i=0}^j\bigl(\partial_t+ip_k(x_t,\xi_t)\bigr)^{j-i}
\Bigl[\bigl(\partial_r+ip_k(x_{t+r},\xi_{t+r})\bigr)^i
\tW_k^\omega(r)\,\tE_k^\omega(t)f_{x,\xi}\Bigr]_{r=0}\\
+\int_0^t \bigl(\partial_t+ip_k(x_t,\xi_t)\bigr)^{j+1}\tW_k^\omega(t-s)\tE_k^\omega(s)f_{x,\xi}\,ds.
\end{multline*}
The latter term on the right is handled by Lemma \ref{lem:FBI}, since we have already shown that
$\tE_k^\omega(s)f_{x,\xi}=f(s,\cdot)_{x_s,\xi_s}$ where $f(s,\cdot)$ is a bounded family of Schwartz functions. 
The first term on the right expands into a sum of terms
\begin{equation}\label{bdterms}
\Bigl[\partial_t^n\bigl(\partial_r+ip_k(x_{t+r},\xi_{t+r})\bigr)^i
\tW_k^\omega(r)\Bigr]_{r=0}\bigl(\partial_t+ip_k(x_t,\xi_t)\bigr)^{j-n-i}\tE_k^\omega(t)f_{x,\xi}.
\end{equation}
We can write $\Bigl[\partial_t^n\bigl(\partial_r+ip_k(x_{t+r},\xi_{t+r})\bigr)^i
\tW_k^\omega(r)\Bigr]_{r=0}$ as a sum of terms
$$
\bigl(\partial_t^{n_1} p_k(x_t,\xi_t)\bigr)\cdots
\bigl(\partial_t^{n_m} p_k(x_t,\xi_t)\bigr)
\Bigl[\bigl(\partial_r+ip_k(x_{t+r},\xi_{t+r})\bigr)^{l}
\tW_k^\omega(r)\Bigr]_{r=0}
$$
where $n_1+\cdots +n_m+m+l=n+i$, and each $n_j\ge 1$. By Lemma \ref{lem:FBI} and the induction assumption
we can write
\begin{multline*}
\Bigl[\bigl(\partial_r+ip_k(x_{t+r},\xi_{t+r})\bigr)^l
\tW_k^\omega(r)\Bigr]_{r=0}\bigl(\partial_t+ip_k(x_t,\xi_t)\bigr)^{j-n-i}\tE_k^\omega(t)f_{x,\xi}\\
=2^{\frac k2(l+j-n-i)}f(t,\cdot)_{x_t,\xi_t}
\end{multline*}
for a bounded family of Schwartz functions $f(t,\cdot)$. The estimates
$$
\bigl|\partial_t^{n_j} p_k(x_t,\xi_t)\bigr|\le C_{n_j}\,2^{\frac k2(n_j+1)}\,,\quad n_j\ge 1,
$$
then show that the term in \eqref{bdterms} is of the form $2^{\frac k2 j}f(t,\cdot)_{x_t,\xi_t}$
for a bounded family of Schwartz functions $f(t,\cdot)$, which implies \eqref{newhest}. 
\end{proof}


We use this to establish sideways energy estimates for $E(t)$, which state that 
if the initial data $f$ is microlocalized to frequencies within a small angle of the co-direction $\omega$, then the $L^2$ norm of the restriction of $E(t)f$ to space-time hyperplanes perpendicular to $\omega$ is dominated by the $L^2$ norm of $f$. By rotation and translation invariance it suffices to consider $\omega=e_1$ and the plane $x_1=0$.

\begin{theorem}\label{thm:sideways}
Suppose $\phi\in C_c^\infty\bigl((-\hf,\hf)\bigr)$. Then
$$
\bigl\|\phi(t)\bigl(a_{e_1}(D)\psi_k(D)E(t)f\bigr)\bigr|_{x_1=0}\bigr\|_{L^2_{x'}L^2_t}\le 
C\,\|f\|_{L^2}
$$
for a constant $C$ that is independent of $k$.
\end{theorem}
\begin{proof}
By Lemma \ref{energyloc'} and the comments following Corollary \ref{cor:eprop}, it suffices to show that
$$
\bigl\|\phi(t)\bigl(\tE_k^{e_1}(t)f\bigr)\bigr|_{x_1=0}\bigr\|_{L^2_{x'}L^2_t}\le 
C\,\|f\|_{L^2}.
$$
For $\xi\in\Rd$ with $|\angle(\xi,e_1)|\le \frac 12$ and $|s|\le 2$, the null bicharacteristic 
curve $\gamma(s)\in (\re^{d+1})^*$ of $\tau+p_k(y,\eta)$ that passes over $(x,\xi)$ at time $s=0$ satisfies $\frac 45\le |y_1'(s)|\le \frac 54$.
Consequently, if $|x_1|\le \frac 32$ and $|\angle(\xi,e_1)|\le \frac 12$ there is a unique value $s=s(x,\xi)$ in $\{s:|s|\le 2\}$ such that $\gamma(s(x,\xi))\in\{y_1=0\}$. We parameterize the cotangent bundle of $y_1=0$ by $(t,y',\tau,\eta')$, and let $T_k^0$ be the wave packet transform acting on this plane. Observe that
the integral kernel $\tK_k^{e_1}(t,y',\tau,\eta';x,\xi)$ of $T_k^0\bigl(\phi(t)\tE_k^{e_1}(t)T_k^*\bigr)$ vanishes unless $|\angle(\xi,e_1)|\le \frac 12$. 

We show that if $|x_1|\le \frac 32$, then
\begin{multline}\label{sideways}
\bigl|\tK_k^{e_1}(t,y',\tau,\eta';x,\xi)\bigr|\\
\le C_N\,\bigl(1+2^{\frac k2}|(t,y')-\Pi_{t,y'}\gamma(s(x,\xi))|
+2^{-\frac k2}|(\tau,\eta')-\Pi_{\tau,\eta'}\gamma(s(x,\xi))|\bigr)^{-N}.
\end{multline}
The Schur test, and the fact that $(x,\xi)\rightarrow \gamma(s(x,\xi))\bigr|_{y_1=0}$ is a bilipschitz symplectic map, shows $L^2$ boundedness of 
$T_k^0\,\phi(t)\tE_k^{e_1}(t)T_k^*\,\one_{|x_1|<\frac 32}$. We consider the case $|x_1|>\frac 32$ afterwards.

To prove \eqref{sideways}, we use Lemma \ref{lem:tEkernel} to express $\tK_k^{e_1}(t,y',\tau,\eta';x,\xi)$ as
\begin{multline*}
2^{\frac{kd}2}\int e^{i\la \xi_s,(0,z')-x_s\ra-i\tau (s-t)-i\la\eta',z'-y'\ra}\\
\times h\bigl(s,2^{\frac k2}((0,z')-x_s)\bigr)\,g\bigl(2^{\frac k2}(s-t,z'-y')\bigr)\,ds\,dz'.
\end{multline*}

Since $\gamma$ is null, we have $\Pi_\tau\gamma(s)=-p(x_s,\xi_s)=-\la\xi_s,\partial_sx_s\ra$.
We then note that
\begin{align*}
&\Bigl(\partial_{z'}+i(\eta'-\xi_s')\Bigr) e^{i\la \xi_s,(0,z')-x_s\ra-i\tau (s-t)-i\la\eta',z'-y'\ra}=0=\\
&\Bigl(\partial_s+i(\tau+p_k(x_s,\xi_s))-i\la \partial_s\xi_s,(0,z')-x_s\ra\Bigr)
 e^{i\la \xi_s,(0,z')-x_s\ra-i\tau (s-t)-i\la\eta',z'-y'\ra}.
\end{align*}
Applying each of $2^{-\frac k2}\partial_{z'}$, $2^{-\frac k2}\partial_s$, or $\la \partial_s\xi_s,(0,z')-x_s\ra$ to the amplitude term $h(\cdots)g(\cdots)$ preserves its form.
An integration by parts argument, together with Schwartz bounds on $h$ and $g$, then shows that the integral is dominated in absolute value by 
\begin{multline}\label{Ke1est}
C_N\,2^{\frac {kd}2}\int \Bigl(1+2^{-\frac k2}|\eta'-\xi_s'|+2^{-\frac k2}|\tau+p_k(x_s,\xi_s)|\\
+2^{\frac k2}|s-t|
+2^{\frac k2}|z'-y'|+2^{\frac k2}|(0,z')-x_s|\Bigr)^{-N}\,ds\,dz'.
\end{multline}
Note that $|(x_s)_1|\ge \frac 45|s-s(x,\xi)|$ as $|\partial_s(x_s)_1|\ge \frac 45$. Since $2^{-k}\xi_s$, $2^{-k}p_k(x_s,\xi_s)$, and $x_s$ are all uniformly Lipschitz in $s$, the integral is in turn bounded by
\begin{multline*}
C_{N+2n+1}\,2^{\frac {kd}2}\int \bigl(1+2^{\frac k2}|s-t|+2^{\frac k2}|z'-y'|\bigr)^{-2n-1}\,ds\,dz'\\
\times\Bigl(1+2^{-\frac k2}|(\tau,\eta')-\Pi_{\tau,\eta'}\gamma(s(x,\xi))|
+2^{\frac k2}|(t,y')-\Pi_{t,y'}\gamma(s(x,\xi))|\Bigr)^{-N}
\end{multline*}
which yields the estimate \eqref{sideways} for $|x_1|\le\frac 32$.

If $|x_1|\ge\frac 32$, $|t|\le 1$, we have $|(x_t)_1|\ge \frac 16|x_1|\ge\frac 14|t|$.
By a similar proof to above, \eqref{Ke1est} then leads to the following bounds,
\begin{multline*}
\bigl|\tK_k^{e_1}(t,y',\tau,\eta';x,\xi)\,\one_{|x_1|\ge \frac 32}\bigr|\\
\le
C_N\bigl(1+2^{-\frac k2}|\tau+p_k(x,\xi)|+2^{-\frac k2}|\eta'-\xi'|
+2^{\frac k2}|x_1|+2^{\frac k2}|y'-x'|\bigr)^{-N}.
\end{multline*}
Here we use, for example, that 
$$
|x_1|+|y'-x'|\lesssim |(x_1)_t|+|y'-x'|\lesssim |(x_1)_s|+|y'-x'_s|+|s-t|
$$
by the above. The Schur test, and the fact that $(x,\xi_1,\xi')\rightarrow (x,p_k(x,\xi),\xi')$ is a diffeomorphism on $|\angle(\xi,e_1)|\le\frac 12$, proves $L^2$ boundedness of the operator
$T_k^0\,\phi(t)\tE_k^{e_1}(t)T_k^*\,\one_{|x_1|\ge\frac 32}$.
\end{proof}


We now turn to the proof of \eqref{fundest} for the operator $E(t)$, that is
\begin{align*}
\|\la D\ra^{-s}E(t)f\|_{L^q_t L^r_x([0,1]\times\Rd)}&\le C\,\|f\|_{L^2(\Rd)}\\
\Bigl\|\la D\ra^{-s}\!\int_0^tE(t-s)F(s,\cdot)\,ds\Bigr\|_{L^q_t L^r_x([0,1]\times\Rd)}&\le 
C\,\|\la D\ra^{1-s} F\|_{L^{\tq'}_t L^{\tr'}_x([0,1]\times\Rd)}
\end{align*}
for $s,q,\tq,r,\tr$ satisfying the conditions of Theorem \ref{thm:strichartz}.
A consequence of Corollary \ref{cor:eprop} is that 
$$
a_\omega(D)\psi_k(D) E(t)=a_\omega(D)\psi_k(D)E(t)a'_\omega(D)\psi'_k(D)+a_\omega(D)\psi_k(D)R(t),
$$
with $R(t)$ a smoothing operator, and $a'_\omega(\eta)\psi'_k(\eta)$ a $S^0_{1,0}$ cutoff to a $\delta 2^{k}$ neighborhood of the support of $a_\omega(\eta)\psi_k(\eta)$. 
Since $q,r\ge 2\ge \tq',\tr'$, it suffices
by Littlewood-Paley theory to prove that, for a constant $C$ independent of $k$,
$$
\|a_\omega(D)\psi_k(D)E(t)f\|_{L^q_t L^r_x([0,1]\times\Rd)}\le C\,2^{ks}\|f\|_{L^2(\Rd)},
$$
and that
\begin{multline*}
\Bigl\|\int_0^ta_\omega(D)\psi_k(D)E(t-s)a'_\omega(D)\psi'_k(D)F(s,\cdot)\,ds\Bigr\|_{L^q_t L^r_x([0,1]\times\Rd)}\\
\le 
C\,2^k\|F\|_{L^{\tq'}_t L^{\tr'}_x([0,1]\times\Rd)}.
\end{multline*}
Since $E(t)E^*(s)=E(t-s)$, we can apply \cite[Theorem 1.2]{KT} with a scaling of $(t,x)$ by $2^k$ to conclude that these are implied by the estimate
\begin{multline*}
\|a'_\omega(D)\psi'_k(D)E(t-s)a'_\omega(D)\psi'_k(D)f\|_{L^\infty(\Rd)}\\
\le C\,2^{kd}\bigl(1+2^k|t-s|\bigr)^{-\frac{d-1}2}\|f\|_{L^1(\Rd)}.
\end{multline*}
By Corollary \ref{cor:eprop} and the comments following it, this estimate in turn is implied by proving the same estimate with $E(t-s)$ replaced by $\tE_k^\omega(t-s)$. Letting $\tK_k^\omega(t,x,y)$ be the integral kernel of $\tE_k^\omega(t)$, we need show that
$$
|\tK_k^\omega(t,x,y)|\le C\,2^{kd}\bigl(1+2^k|t|\bigr)^{-\frac{d-1}2},\qquad |t|\le 1.
$$

We in fact prove a stronger estimate, which captures the decay of the fundamental solution away from the light cone.
We will show in Section \ref{sec:wavepackets} that, for all $N$, with $S_t(y)$ the geodesic sphere of radius $|t|$ centered at $y$, and $\dist(x,S_t(y))$ the geodesic distance in $\g_k$ of $x$ to the set $S_t(y)$,
\begin{equation}\label{eqn:kbound}
|\tK_k^\omega(t,x,y)|\le C_N\,2^{kd}(1+2^k|t|)^{-\frac{d-1}2}\bigl(1+2^k \bigl|\dist(x,S_t(y))\bigr|\,\bigr)^{-N},
\end{equation}
which will imply \eqref{fundest} by the above.

By similar steps and duality, estimate \eqref{fundsqfnest} reduces to proving that, for $q_d$ and $s_d$ as in Theorem \ref{thm:sqfn}, and $\phi\in C_c^\infty\bigl((-\hf,\hf)\bigr)$,
$$
\Bigl\|\phi(t)\int \tE_k^\omega(t-s)\phi(s)F(s,\cdot)\,ds\Bigr\|_{L^{q_d}_xL^2_t}
\le C\,2^{2ks_d}
\|F\|_{L^{q_d'}_yL^2_s}.
$$
It suffices to prove this for $\omega=e_1$.
We deduce from \eqref{eqn:kbound} that
$$
|\tK_k^\omega(t,x,y)|\le C_N\,2^{kd}(1+2^k|x-y|)^{-\frac{d-1}2}\bigl(1+2^k \bigl|t-\dist(x,y)\bigr|\,\bigr)^{-N},
$$
which uses that $\dist(x,S_t(y))\ge |t-\dist(x,y)\bigr|$, and $\dist(x,y)\approx |x-y|$. As a consequence, letting $x=(x_1,x')$, we have
\begin{multline*}
\Bigl\|\phi(t)\int \tK_k^{e_1}(t-s,x_1,x',y_1,y')\phi(s)F(s,y_1,y')\,ds\,dy'\Bigr\|_{L^\infty_{x'}L^2_t}\\
\le C\,2^{k(d-1)}(1+2^k|x_1-y_1|)^{-\frac{d-1}2}
\|F(\cdot,y_1,\cdot)\|_{L^1_{y'}L^2_s}.
\end{multline*}
On the other hand, writing $E(t-s)=E(t)E(s)^*$, Theorem \ref{thm:sideways} and the comments surrounding \eqref{eqn:energyloc'} show that
\begin{multline*}
\Bigl\|\phi(t)\int \tK_k^{e_1}(t-s,x_1,x',y_1,y')\phi(s)F(s,y_1,y')\,ds\,dy'\Bigr\|_{L^2_{x'}L^2_t}\\
\le C\,
\|F(\cdot,y_1,\cdot)\|_{L^2_{y'}L^2_s}.
\end{multline*}
Interpolation then yields
\begin{multline*}
\Bigl\|\phi(t)\int \tK_k^{e_1}(t-s,x_1,x',y_1,y')\phi(s)F(s,y_1,y')\,ds\,dy'\Bigr\|_{L^{q_d}_{x'}L^2_t}\\
\le C\,2^{2ks_d}|x_1-y_1|^{-1+\frac 1{q_d'}-\frac 1{q_d}}
\|F(\cdot,y_1,\cdot)\|_{L^{q_d'}_{y'}L^2_s},
\end{multline*}
and an application of the Hardy-Littlewood inequality yields the desired bound.

\section{Wave packets and dispersive estimates}\label{sec:wavepackets}
This section is devoted to the proof of \eqref{eqn:kbound} for $|t|\le 1$. Without loss of generality we assume $0\le t\le 1$ throughout to simplify notation.

To motivate the proof we recall Fefferman's analysis in \cite{F} of $\exp(-i|D|)$, the wave group for the Euclidean laplacian at $t=1$. Consider
$$
K_k(x)=(2\pi)^{-n}\int e^{i\la x,\eta\ra-i|\eta|}\,\psi_k(\eta)\,d\eta.
$$
Following \cite{F}, decompose $\psi_k(\eta)=\sum_{\nu}\psi_k^\nu(\eta)$, where $\psi_k^\nu$ equals $\psi_k$ multiplied by a homogeneous cutoff to a conic neighborhood of angle $2^{-\frac k2}$ about the direction $\nu\in\sph^{d-1}$, and $\nu$ varies over a discrete set of directions separated by distance $2^{-\frac k2}$. The function $\psi_k^\nu$ behaves like a scaled cutoff to a rectangle of dimension $2^{k}\times (2^{\frac k2})^{d-1}$, in that
$$
\bigl|\la\nu,\partial_\eta\ra^m\partial_\eta^\alpha\psi_k^\nu(\eta)\bigr|\le 
C_{m,\alpha}2^{-k(m+\frac{|\alpha|}2)},
$$
with constants independent of $k$.
The angular width is selected since one can write
$$
e^{-i|\eta|}\,\psi_k^\nu(\eta)=e^{-i\la \nu,\eta\ra}\,a_k^\nu(\eta),
$$
where $a_k^\nu$ satisfies the same derivative estimates as $\psi_k^\nu$. This decomposes
$$
K_k(x)=\sum_\nu {f_k^\nu}(x-\nu),\quad\text{where}\;\; \widehat{f_k^\nu}(\eta)=a_k^\nu(\eta).
$$
The function $f_k^\nu(x-\nu)$ is concentrated in a rectangle centered at $\nu$, of dimension $2^{-k}$ along the $\nu$ direction and $2^{-\frac k2}$ in perpendicular directions. By the spacing of the indices $\nu$ these rectangles are essentially disjoint, and simple geometry shows that, for all $N$,
$$
|K_k(x)|\le C_N\,2^{k(\frac{d+1}2)}\bigl(1+2^k\bigl||x|-1\bigr|\bigr)^{-N}.
$$
If $2^{-k}\le t\le 1$, the above argument can be scaled by $t$ to decompose the kernel of $\exp(-it|D|)$. This gives a $t$-dependent splitting $\psi_k=\sum_\nu\psi_{k,t}^\nu$, where now $\psi_{k,t}^\nu$ is localized to a cone of angle $t^{-\frac 12} 2^{-\frac k2}$, and the $f_{k,t}^\nu(x-t\nu)$ are concentrated in a rectangle of dimensions $2^{-k}$ and $t^{\frac 12}2^{-\frac k2}$, centered at $t\nu$.
These rectangles are again mutually disjoint, leading to bounds
$$
|K_k(t,x)|\le C_N\,2^{k(\frac{d+1}2)}t^{-(\frac{d-1}2)}\bigl(1+2^k\bigl||x|-t\bigr|\bigr)^{-N}.
$$
For $0\le t\le 2^{-k}$, the symbol $e^{-it|\eta|}$ is a classical symbol, and the kernel has the same size as $\widehat \psi_k(-x)$, or as $K_k(t,x)$ at $t=2^{-k}$.

The decomposition of \cite{F} was used in Seeger-Sogge-Stein \cite{SSS} to estimate the kernel of oscillatory integral operators with nondegenerate phase functions, for example $\exp(-iP)$ for a smooth metric. The key ingredient is that the phase function $\varphi(x,\eta)$ can be linearized in $\eta$ over the support of each $\psi_k^\nu$, up to an error that behaves like an appropriate amplitude function.

To get the correct kernel estimates for $t\ll 1$ requires better estimates on the phase function for it to linearize over the support of $\psi_{k,t}^\nu$. The needed estimates are precisely those of \eqref{linphase}, and the corresponding estimates for amplitudes are those of \eqref{linsymbol}.

The proof of the estimates in \eqref{eqn:kbound} for a single term $\tW_k^\omega(t)$ or $\tB_k^\omega(t)$ would follow along the lines of \cite{SSS}, using the decomposition $\psi_{k,t}^\nu$, together with \eqref{linphase}--\eqref{linsymbol}. 
We need, however, prove these estimates for a product of arbitrarily many terms $\prod_j \tB_k^\omega(tr_j)$, where $\sum r_j=1$. It is still appropriate to use the partition $\psi_{k,t}^\nu$
for each term; however, we need a function space argument in order to handle a product of terms since there is no hope for controlling the operator product using a symbol calculus. 
We therefore work with a wave packet frame and function spaces using weighted norms in that frame that grow with the distance to a given point $(x_0,\nu_0)$ on the cosphere bundle. We prove that the operator $\tB_k^\omega(s)$ is bounded from the space weighted at $(x_0,\nu_0)$ to the space weighted at its time-$s$ flowout $(x_s,\nu_s)$. These function space estimates iterate and yield a convergent sum, which is sufficient to prove the bounds in \eqref{eqn:kbound}.


\subsection{The wave packet frame}
We will establish \eqref{eqn:kbound} for $2^{-k} \le t \le 1$; the proof for $0\le t\le 2^{-k}$ follows by using the same proof as for $t=2^{-k}$. We consider $t$ to be fixed for this section and suppress the dependence of the frame on $t$; however, we note that all constants are uniform over $t\in [0,1]$.

We prove the estimate by studying the behavior of $\tilde{E}_k^\omega(t)$ in a family of wave packets that form a frame for functions that are frequency localized at scale $2^k$. The wave packet frame that we use at scale $2^k$ is essentially a spatial dilation by $t^{-1}$ of the scale $t\ts 2^k$ parabolic wave packets of Smith \cite{Sm0}. The only difference is that our frame covers more than one dyadic region, but we provide the details here for completeness.

We will be expanding functions with Fourier transform supported in the annulus
$$
A_k = \{ \eta: \tfrac{4}{5}\,2^{k-1} \le |\eta| \le \tfrac{5}{4}\,2^{k+2}\}.
$$ 
Let 
$A_k' = \{\eta : \frac{2}{3} 2^{k-1} \le |\eta| \le \frac{3}{2}2^{k+2}\}$. 
We construct a partition of unity on $A_k$, supported in $A_k'$, of the form
$$
1= \sum_{\nu \in \Upsilon_{k,t}} \beta_{k,t}^{\nu}(\eta)^2 \text{ when } \eta \in A_k,\quad  \supp (\beta_{k,t}^{\nu}) \subset \Omega_{k,t}^{\nu},
$$
where $\Upsilon_{k,t}$ is a collection of unit vectors separated by $t^{-\frac{1}{2}}2^{-\frac{k}{2}}$, and $\beta_{k,t}^{\nu}(\eta)$ satisfies the following estimates
\begin{equation}\label{est:betakj}
\left| \la \nu, \partial_{\eta} \ra^j \partial_{\eta}^{\alpha} \beta_{k,t}^{\nu}(\eta) \right| \le C_{j,\alpha}\,2^{-kj}\,\bigl(t^{-\frac 12}2^{\frac k2}\bigr)^{-|\alpha|}.
\end{equation}
Observe that $\Omega_{k,t}^{\nu}$, defined in \eqref{rknudef}, is contained in a rectangle of dimension $2^{k+3}$ along the direction $\nu$, and $t^{-\frac 12}2^{\frac k2}$ along the directions orthogonal to $\nu$. For each $\nu$, let $\Xi_{k,t}^{\nu}$ be a rectangular lattice in $\mathbb{R}^n$ with spacing $2\pi \cdot 2^{-k-3}$ along the $\nu$ direction and spacing $2\pi\cdot t^{\frac{1}{2}} 2^{-\frac{k}{2}}$ in directions orthogonal to $\nu$. 
Let $\Gamma_{k,t} = \bigl\{(x,\nu): x \in \Xi_{k,t}^{\nu}\,,\, \nu \in \Upsilon_{k,t} \bigr\}$, which is a discrete subset of the cosphere bundle $S^*(\Rd)$. We use
$\gamma=(x,\nu)$ to denote a variable in $S^*(\Rd)$, and for $\gamma\in\Gamma_{k,t}$ we set 
$$
\hat{\phi}_{\gamma}(\eta) = 2^{-\frac 32}\,2^{-k(\frac{d+1}4)}t^{\frac{d-1}4}e^{-i\la x, \eta\ra} \beta_{k,t}^{\nu}(\eta).
$$
Then, with $\la\nu^\perp,\partial_y\ra$ denoting derivatives in directions perpendicular to $\nu$,
\begin{multline}\label{phigammaest}
\bigl|\la\nu^\perp,\partial_y\ra^\alpha\partial_y^\beta\phi_\gamma(y)\bigr|\le
C_{N,\alpha,\beta}\,2^{k(\frac{d+1}4)}\,t^{-\frac{d-1}4}\times
\\
2^{k|\beta|}\,\bigl(t^{-\hf}2^{\frac k2}\bigr)^{|\alpha|}\,\bigl(1+2^k|\la\nu,y-x\ra|+t^{-1}2^k\,|y-x|^2\bigr)^{-N}.\rule{0pt}{13pt}
\end{multline}
Functions $f\in L^2(\mathbb{R}^n)$ with $\supp(\hat{f}) \subset A_k$ admit an expansion in 
$\{\phi_\gamma\}_{\gamma \in \Gamma_{k,t}}$,
$$
f = \sum_{\gamma\in\Gamma_{k,t}} c_\gamma \phi_\gamma,\qquad
c_\gamma = \int \overline{\phi_\gamma(y)}f(y)\, dy.
$$

We define a pseudodistance function on the cosphere bundle $S^*(\Rd)$ by
$$
d_t(x,\nu; x', \nu') = |\la \nu, x-x'\ra| + |\la \nu', x-x'\ra| + t\ts|\nu-\nu'|^2 + t^{-1}|x-x'|^2.
$$
This is the parabolic pseudodistance of Smith \cite{Sm0} scaled like the wave packet frame, and satisfies, for all $t>0$,
\begin{equation}\label{pseudist}
d_t(\gamma;\gamma'')\le 4\ts d_t(\gamma; \gamma')+4\ts d_t(\gamma';\gamma'').
\end{equation}
It is also approximately invariant under the Hamiltonian flow $\chi_s$ for $s\le t$. This was proven for $C^{1,1}$ metrics in \cite{Sm0}, we provide the proof here for metrics of bounded curvature.
\begin{lemma}\label{distinvar}
For some $C$ and all $0\le s\le t\le 1$, and $\chi_s$ the projected Hamiltonian flow map for any metric $\g_M$ satisfying \eqref{cond0}--\eqref{cond2}. Then
$$
C^{-1}\, d_t(\gamma;\gamma')\le d_t(\chi_s(\gamma);\chi_s(\gamma'))\le C\, d_t(\gamma;\gamma').
$$
\end{lemma}
\begin{proof}
Let $\eta=\nu$ and $\eta'=\nu'$. If $(x_s,\xi_s)$ is the (non-projected) Hamiltonian flow of $(x,\eta)$, then $\bigl||\xi_s|-1\bigr|\lesssim c_d$, so we can replace $\nu_s$ by $\xi_s$ in the distance function.
From Corollary \ref{cor:geodflow}, when $|\eta|=1$ we have the bound $|\partial_\eta x_s|\lesssim s$, 
$|\partial_x x_s|+|\partial_x \xi_s|+|\partial_\eta\xi_s|\lesssim 1$, and we deduce
$$
|x'_s-x_s|+t\ts|\xi'_s-\xi_s|\lesssim |x'-x|+t\ts|\eta'-\eta|.
$$
Applying this also to $\chi_{-s}$ we obtain
$$
t^{-1}|x'_s-x_s|^2+t\ts|\xi'_s-\xi_s|^2\approx t^{-1}|x'-x|^2+t\ts|\eta'-\eta|^2.
$$
By symmetry it thus suffices to show that
\begin{equation}\label{eqn:taylorerror}
|\la \eta,x'-x\ra-\la\xi_s,x'_s-x_s\ra|\lesssim t^{-1}|x'_s-x_s|^2+t\ts|\eta'-\eta|^2.
\end{equation}
Let $\varphi$ be the phase function for $\g_M$, and write $x=\nabla_\eta\varphi(s,x_s,\eta)$ and $\xi_s=\nabla_x\varphi(s,x_s,\eta)$. By homogeneity,
\begin{align*}
\la \eta,x'-x\ra-&\la\xi_s,x'_s-x_s\ra\\
&=\la\eta,\nabla_\eta\varphi(x,x'_s,\eta')-\nabla_\eta\varphi(s,x_s,\eta)\ra-
\la\nabla_x\varphi(s,x_s,\eta),x'_s-x_s\ra
\rule{0pt}{11pt}\\
&=\varphi(s,x'_s,\eta')-\varphi(s,x_s,\eta)-\la x'_s-x_s,\nabla_x\varphi(s,x_s,\eta)\ra
\rule{0pt}{11pt}\\
&\hspace{2.4in}-\la\eta'-\eta,\nabla_\eta\varphi(s,x'_s,\eta')\ra.
\end{align*}
Observe that, by Theorem \ref{thm:yest}, 
\begin{align*}
\bigl|\la\eta'-\eta,\nabla_\eta\varphi(s,x'_s,\eta')-\nabla_\eta\varphi(s,x_s,\eta)\ra\bigr|
&\lesssim |\eta'-\eta|\bigl(\ts|x'_s-x_s|+t\ts|\eta'-\eta|\ts\bigr)\\
&\lesssim t^{-1}|x'_s-x_s|^2+t\ts|\eta'-\eta|^2.
\end{align*}
Consequently, it suffices to show that the error bound for the first order Taylor expansion of $\varphi(s,x'_s,\eta')-\varphi(s,x_s,\eta)$ is bounded by the right hand side of \eqref{eqn:taylorerror}. The estimates \eqref{dxphi}--\eqref{dxdetaphi} give $|\partial_x^2\varphi_k|\lesssim 1$, $|\partial_x\partial_\eta\varphi_k|\lesssim 1$, $|\partial_\eta^2\varphi_k|\lesssim |s|$, and hence the remainder is dominated by
$$
|x'_s-x_s|^2+|x'_s-x_s|\,|\eta'-\eta|+t\ts|\eta'-\eta|^2\le \tfrac 32\,t^{-1}|x'_s-x_s|^2+\tfrac 32\,t\ts|\eta'-\eta|^2
$$
giving the desired bound.
\end{proof}

For any given integer $M \geq 0$ and point $\gamma_0\in S^*(\Rd)$, we define a weighted norm space
$$
\|f\|_{M,\gamma_0}^2 = \sum_{\gamma} \big( 1+2^k d_t(\gamma;\gamma_0)\big)^{2M} |c_\gamma(f)|^2,\qquad
c_\gamma(f) = \int \overline{\phi_\gamma(y)}f(y)\, dy.
$$
For dyadically localized $f$, this norm roughly measures how far $f$ is from being a wave packet centered at $\gamma_0$. In the next subsection we will prove the following theorem. 

\begin{theorem}\label{thm:fundest}
Suppose that $0\le s\le t\le 1$, $\gamma_0\in\Gamma_{k,t}$, and $\chi_s(\gamma_0)=(x_s,\nu_s)$, where $\chi_s$ is the projected Hamiltonian flow for $\g_k$. Then for all $l,\beta,N$, there are constants $C_{l,\beta,N}$ so that
\begin{multline}\label{derivativeestoff}
\bigl| \la \nu_s^\perp, \partial_x \ra ^{\alpha}\partial_x^{\beta}
\bigl(\tB_k^\omega(s)\phi_{\gamma_0} \bigr)(x)\bigr| 
\le C_{N,\alpha,\beta} \,2^{k(\frac{d+1}4)}t^{-\frac{d-1}4} \\
\times 2^{k|\beta|}\tkhfinv^{|\alpha|}\bigl(1+2^k|\la \nu_s,x - x_s  \ra| + t^{-1}2^{k}|x-x_s|^2\bigr)^{-N}. 
\end{multline}
\end{theorem}

In the remainder of this subsection we deduce \eqref{eqn:kbound} from Theorem \ref{thm:fundest}. First we deduce $\|\cdot\|_{M,\chi_s(\gamma)}$ mapping properties for $\tB_k^\omega(s)$ from \eqref{derivativeestoff}. The left hand side of \eqref{derivativeestoff} vanishes unless $\angle(\omega,\gamma)\le \frac 14$, so
we may assume $\angle(\omega,\gamma_s)\le \frac 12$.

\begin{lemma}\label{lem:derivtoweight}
Suppose that $\hatf$ is supported in the set $\{n:\angle(\eta,\nu_0)\le \frac 12\}$, and for all $N,\alpha,\beta$ we have
\begin{multline*}
\bigl|\la\nu_0^\perp,\partial_y\ra^\alpha\partial_y^\beta f(y)\bigr|\le
C_{N,\alpha,\beta}\,2^{k(\frac{d+1}4)}\,t^{-\frac{d-1}4}
\\
\times 2^{k|\beta|}\,\bigl(t^{-\hf}2^{\frac k2}\bigr)^{|\alpha|}\,
\bigl(1+2^k|\la\nu_0,y-x_0\ra|+t^{-1}2^k\,|y-x_0|^2\bigr)^{-N}.\rule{0pt}{13pt}
\end{multline*}
Let $\gamma_0=(x_0,\nu_0)$. Then for all $M\ge 0$ we have $\|f\|_{M,\gamma_0}\le C_M$, where $C_M$ depends on only a finite number of the $C_{N,\alpha,\beta}$.
\end{lemma}
\begin{proof}
Without loss of generality we assume that $\nu_0=e_1$. By the derivative estimates we have
$$
|\hatf(\eta)|\le C_N\,2^{-k(\frac{d+1}4)}\,t^{\frac{d-1}4}\bigl(1+2^{-k}|\eta_1|+2^{-k}t|\eta'|^2\bigr)^{-N},
$$
where for each $N$ the value of $C_N$ depends on only a finite number of $C_{N,\alpha,\beta}$.
Since $\widehat\phi_\gamma$ is supported where $|\eta'|\ge 2^{k-4}|\nu-e_1|$, by Plancherel's theorem we obtain for all $N$, and similar $C_N$,
\begin{equation}\label{eqn:anglebound}
|c_\gamma(f)|\le C_N\,\bigl(1+2^kt|\nu-e_1|^2\bigr)^{-2N},\qquad c_\gamma(f)=\int\overline{\phi_\gamma(y)}\,f(y)\,dx.
\end{equation}
By the pointwise estimates on $f(y)$ and $\phi_\gamma(y)$, we have
\begin{multline*}
|c_\gamma(f)|\le C_N \,2^{k(\frac{d+1}2)}t^{-\frac{d-1}2}\int
\bigl(1+2^kd_t(y,e_1;\gamma_0)\bigr)^{-2N-d}\\
\times
\bigl(1+2^kd_t(y,\nu;\gamma)\bigr)^{-2N-d}\,dy.
\end{multline*}
By \eqref{pseudist}, noting that $d_t(y,e_1;y,\nu)=t|\nu-e_1|^2$, we have
$$
\tfrac 1{16} d_t(\nu,\gamma_0)\le d_t(y,e_1;\gamma_0)+d_t(y,\nu;\gamma)+t|\nu-e_1|^2.
$$
Together with \eqref{eqn:anglebound}, this implies 
$|c_\gamma(f)|\le C_N\,\bigl(1+2^k d_t(\gamma;\gamma_0)\bigr)^{-N}$. The lemma then follows from
the bound
\begin{equation}\label{eqn:schurbound}
\sup_{\gamma'}\sum_{\gamma\in\Gamma_{k,t}}\bigl(1+2^k\,d_t(\gamma;\gamma')\bigr)^{-d-1}\le C_d,
\end{equation}
which follows from estimate $(2.3)$ in \cite{Sm0} after rescaling.
\end{proof}

The converse to Lemma \ref{lem:derivtoweight} also holds; we need it only for $\alpha=\beta=0$, and prove that version in the proof of Corollary \ref{cor:tekest} below.

An immediate consequence of Theorem \ref{thm:fundest} and Lemma \ref{lem:derivtoweight} is decay estimates on the matrix coefficients of $\tB_k^\omega(s)$. Precisely, for all $N$ we have
\begin{equation}\label{tBkcoeffest}
\Bigl|\int\overline{\phi_{\gamma}(y)}\bigl(\tB_k^\omega(s)\phi_{\gamma'}\bigr)(y)\,dy\Bigr|\le
C_N\bigl(1+2^k d_t(\gamma;\chi_s(\gamma'))\bigr)^{-N}.
\end{equation}
We then use this to prove boundedness of $\tB_k^\omega(s)$ in the weighted norm spaces via the following lemma.

\begin{lemma}\label{lem:matrixbound}
Suppose that $M\ge 0$, $0\le s\le t\le 1$, and $T:\mathcal{S}(\Rd)\rightarrow \mathcal{S}'(\Rd)$ is a linear map such that the matrix coefficients
$$
a(\gamma,\gamma')=\int\overline{\phi_\gamma(y)}\bigl(T\phi_{\gamma'})(y)\,dy
$$
satisfy the bound
$$
|a(\gamma,\gamma')|\le \bigl(1+2^k\,d_t(\gamma;\chi_s(\gamma'))\bigr)^{-(M+d+1)}.
$$
Then, uniformly over $\gamma_0\in S^*(\Rd)$, we have $\|Tf\|_{M,\chi_s(\gamma_0)}\le C_M\|f\|_{M,\gamma_0}$.
\end{lemma}
\begin{proof}
It follows from \eqref{eqn:schurbound} that
$$
\sup_{\gamma'}\sum_{\gamma}|a(\gamma,\gamma')|\le C,
$$
and, since $d_t(\gamma;\chi_s(\gamma'))\approx d_t(\chi_{-s}(\gamma);\gamma')$ by Lemma \ref{distinvar}, we also have
$$
\sup_{\gamma}\sum_{\gamma'}|a(\gamma,\gamma')|\le C,
$$
where $C$ is independent of $s, t$ and $k$. By Schur's lemma we conclude
$$
\|\tB_k^\omega(s)f\|_{0,\chi_s(\gamma_0)} \leq C\,\|f\|_{0,\gamma_0}.
$$
The weighted case $M\ge 1$ follows by noting that
$$
\bigl(1+2^kd_t(\gamma;\chi_s(\gamma_0))\bigr)\lesssim 
\bigl(1+2^kd_t(\gamma;\chi_s(\gamma'))\bigr)\bigl(1+2^kd_t(\gamma';\gamma_0) \bigr),
$$
which follows from 
$$
d_t(\gamma;\chi_s(\gamma_0)) \le 
4\ts d_t(\gamma;\chi_s(\gamma'))+4\ts d_t(\chi_s(\gamma'); \chi_s(\gamma_0)),
$$
and the fact that $d_t(\chi_s(\gamma');\chi_s(\gamma_0)) \approx d_t(\gamma',\gamma_0)$.
\end{proof}

\begin{corollary}\label{cor:tekest}
Let $\chi_s$ denote the time $s$ projected Hamiltonian flow for $\g_k$. Then for $0\le s \le t \le 1$, and all $M\ge 0$,
$$
\| \tE_k^\omega(s)f \|_{M,\chi_s(\gamma_0)} \leq C_M \|f\|_{M,\gamma_0}
$$
with constant $C_M$ independent of $s$, $t$, $\gamma_0$, $\omega$, and $k$.
\end{corollary}
\begin{proof}
By Lemma \ref{lem:matrixbound} and the estimate \eqref{tBkcoeffest}, which holds also for $\tW_k^\omega(s)$ by the same proof, we have
$$
\|\tB_k^\omega(s)f\|_{M,\chi_s(\gamma_0)}\le C_M\|f\|_{M,\gamma_0}.
$$
The formula \eqref{angloc} for $\tE_k^\omega(t)$ and the group property of $\chi_s$ then show that
\begin{align*}
\| \tE_k^\omega(s)f \|_{M, \chi_s(\gamma)}
&\le\; \sum_{m=0}^{\infty}\frac{s^m\,C_M^{m+1}}{m!}\,\|f\|_{M,\gamma}\\
&=\; C_M\,e^{sC_M}\,\|f\|_{M,\gamma}\,.
\end{align*}
\end{proof}

We conclude this section by deriving the bound \eqref{eqn:kbound} from Corollary \ref{cor:tekest}.
Write $\tK_k^\omega(t,x,y)=\bigl(\tE_k^\omega(t)\delta_y\bigr)(x)$. Since $\tE_k^\omega(t)$ has the factor $\psi_k(D)$ on the right, we may write
$$
\bigl(\tE_k^\omega(t)\delta_y\bigr)(x)=\sum_{\nu \in \Upsilon_{k,t}} \bigl(\tE_k^\omega(t)\beta_{k,t}^\nu(D)^2\delta_y\bigr)(x).
$$
The function $\beta_{k,t}^\nu(D)^2\delta_y$ has Fourier transform $e^{-i\la y,\eta\ra}\beta_{k,t}^\nu(\eta)^2$. Up to a normalization factor, this behaves like the frame element $\phi_\gamma$ at $\gamma=(y,\nu)$, and it is easy to verify that for all $M$
$$
\|\beta_{k,t}^\nu(D)^2\delta_y\|_{M,\gamma}\le C_M\,2^{k(\frac{d+1}4)}t^{-\frac{d-1}4}.
$$
By Theorem \ref{cor:tekest}, letting $\gamma_t\equiv (x_t,\nu_t)=\chi_t(y,\nu)$ we have
\begin{equation}\label{termest}
\|\tE_k^\omega(t)\beta_{k,t}^\nu(D)^2\delta_y\|_{M,\gamma_t}\le 
C_M\,2^{k(\frac{d+1}4)}t^{-\frac{d-1}4}.
\end{equation}
This implies that, for all $N$,
\begin{multline*}
\bigl|\bigl(\tE_k^\omega(t)\beta_{k,t}^\nu(D)^2\delta_y\bigr)(x)\bigr|
\\
\le 
C_N2^{k(\frac{d+1}2)}t^{-\frac{d-1}2}
\bigl(1+2^k\ts|\la\nu_t,x-x_t\ra|+2^k\ts t^{-1}\ts|x-x_t|^2\bigr)^{-N}.
\end{multline*}
We see this using \eqref{termest}, that the frame coefficients $\{c_{\gamma'}\}$ of $\tE_k^\omega(t)\beta_{k,t}^\nu(D)^2\delta_y$ satisfy for all $M$
$$
|c_{\gamma'}|\le C_M\,2^{k(\frac{d+1}4)}t^{-\frac{d-1}4}\,\bigl(1+2^k d_t(\gamma';\gamma_t)\bigr)^{-M}.
$$
From estimates \eqref{phigammaest} on $|\phi_\gamma(x)|$, we follow the proof of \cite[Lemma 2.5]{Sm0} with $\gamma'=(x',\nu')$ to bound
$\bigl|\bigl(\tE_k^\omega(t)\beta_{k,t}^\nu(D)^2\delta_y\bigr)(x)\bigr|$ by
\begin{align*}
C_M2^{k(\frac{d+1}2)}&t^{-\frac{d-1}2}\!\!\sum_{\gamma'\in\Gamma_{k,t}}
\bigl(1+2^k d_t(\gamma';\gamma_t)\bigr)^{-M}\bigl(1+2^k d_t((x,\nu');\gamma')\bigr)^{-M}\\
&\le C_M 2^{k(\frac{d+1}2)}t^{-\frac{d-1}2}\sum_{\nu'\in\Upsilon_{k,t}}\bigl(1+2^k d_t((x,\nu');\gamma_t)\bigr)^{-M}\\
&\le C_M 2^{k(\frac{d+1}2)}t^{-\frac{d-1}2}
\bigl(1+2^k\ts|\la\nu_t,x-x_t\ra|+2^k\ts t^{-1}\ts|x-x_t|^2\bigr)^{-M+\frac d2}.
\end{align*}

Deriving \eqref{eqn:kbound} from Corollary \ref{cor:tekest} then reduces to showing that
\begin{multline*}
\sum_{\nu\in\Upsilon_{k,t}}\bigl(1+2^k\ts|\la\nu_t,x-x_t\ra|+2^k\ts t^{-1}\ts|x-x_t|^2\bigr)^{-N-d}\\
\le C_N\bigl(1+2^k\,\dist(x,S_t(y))\bigr)^{-N}.
\end{multline*}
If $(\tilde x_t,\tilde\nu_t)=\chi_t(y,\tilde\nu)$ and $(x_t,\nu_t)=\chi_t(y,\nu)$,
then by Corollary \ref{cor:geodflow}
\begin{equation*}
\tfrac 45 \,t\le\frac{|\tilde x_t-x_t|}{|\tilde\nu-\nu|}\le \tfrac 54\,t\,.
\end{equation*}
Consequently, the points $x_t$ are separated by $t^{\frac 12}2^{-\frac k2}$ for $\nu\in\Upsilon_{k,t}$, and thus
$$
\sum_{\nu\in\Upsilon_{k,t}}\bigl(1+2^k\ts t^{-1}\ts|x-x_t|^2\bigr)^{-d}\le C.
$$
It therefore suffices to show that, for $c_d$ small enough, and for each $\nu\in \sph^{d-1}$,
\begin{equation}\label{distbound}
|\la\nu_t,x-x_t\ra|+ t^{-1}\ts|x-x_t|^2\ge \tfrac 1 4\,\dist(x,S_t(y)).
\end{equation}
Here, $x_t\in S_t(y)$ for each $\nu$, and $\nu_t$ is the unit normal to $S_t(y)$ at the point $x_t$, in that $\la\nu_t,\partial_{\eta_j} x_t|_{\eta=\nu}\ra=0$, which follows by homogeneity.

We observe that, by scaling, it suffices to prove \eqref{distbound} in the case $t=1$.
Precisely, $(t^{-1}x_t,\nu_t)$ is the image at time 1 of $(t^{-1}y,\nu)$ under the projected Hamiltonian flow for the metric $\g_k(t\,\cdot)$, 
and $t^{-1}S_t(y)$ is the corresponding unit geodesic sphere centered at $t^{-1}y$, hence the two sides of \eqref{distbound} dilate by the same factor $t$. 
Furthermore, the metric $\g_k(t\,\cdot)$ satisfies conditions \eqref{cond0}--\eqref{cond2} with $M=t\ts 2^{\frac k2}\le 2^{\frac k2}$. 

Without loss of generality we assume $\nu=e_1$ and $y=0$.
We introduce the notation $(x(\omega),n(\omega))=\chi_1(0,\omega)$ to denote the mapping of the 
unit sphere $\sph^{d-1}$ onto the unit conormal bundle of $S_1(0)$.
By Corollary \ref{cor:geodflow}, this map is $C^1$-close to the map 
$\omega\rightarrow(\omega,\omega)$; precisely
$$
|x(\omega)-\omega|+|\nabla_\omega x(\omega)-\Pi^\perp_\omega|+|n(\omega)-\omega|
+|\nabla_\omega n(\omega)-\Pi^\perp_\omega|\lesssim c_d.
$$

As a consequence we may parameterize $S_1(0)\cap\{x_1>0,\,|x'|\le \frac 12\}$ as a graph $x_1=F(x')$, where
\begin{equation}\label{C2close}
\bigl|\partial_x^\alpha\bigl(F(x')-\sqrt{1-|x'|^2}\;\bigr)\bigr|\lesssim c_d,\qquad |\alpha|\le 2,
\quad|x'|\le\frac 12.
\end{equation}
This holds for $|\alpha|\le 1$ by $C^1$ closeness of $x(\omega)$ to $\omega$, and for $|\alpha|=2$ since 
$\nabla_{x'} F(x')=-n'(\omega(x'))/n_1(\omega(x'))$ is $C^1$ close to $-x'/\sqrt{1-|x'|^2}$.

The bound \eqref{distbound} is equivalent to proving, for $x=(x_1,x')\in\re^d$,
$$
\min_\omega|x-x(\omega)|\le 4 \Bigl(|\la n(e_1),x-x(e_1)\ra|+|x-x(e_1)|^2\Bigr).
$$
We assume that $|x-x(e_1)|\le \frac 14$, hence $|x'|\le\frac 12$, as the bound is immediate otherwise. The left hand side is bounded above by $|x_1-F(x')|$, and the bound then follows by the Taylor expansion of $F(x')$ about $x'(e_1)$,
\begin{align*}
|x_1-F(x')|&\le \bigl|x_1-F(x'(e_1))-\la x'-x'(e_1),\nabla_{x'}F(x'(e_1))\ra\bigr|+|x'-x'(e_1)|^2\\
&= n_1(e_1)^{-1}\,|\la n(e_1),x-x(e_1)\ra|+|x'-x'(e_1)|^2
\end{align*}
where we use that $\|\nabla_{x'}^2 F\|\le 2$ for $|x'|\le\frac 12$ by \eqref{C2close} if $c_d$ is small, and $F(x'(e_1))=x_1(e_1)$.


\subsection{Proof of Theorem \ref{thm:fundest}}
We follow the key idea of \cite{SSS}, that the action of a Fourier integral operator on a function $f$ whose Fourier transform is suitably localized can be decomposed as a pseudodifferential operator acting on $f$, followed by a change of coordinates. Suitably localized means that the phase function can be written as a phase that is linear in $\eta$ plus a term that satisfies the estimates of a zero-order symbol on the support of $\hatf(\eta)$. Here we take $f=\phi_\gamma$, with $\hatf$ supported in the set $\Omega_{k,t}^\nu$ defined by \eqref{rknudef}, and the zero-order symbol estimates are those of Corollary \ref{cor:linsymbol}. The estimates of Corollary \ref{phikest'} will be used to establish the linearization of $\varphi_k$ on $\Omega_{k,t}^\nu$.

We prove Theorem \ref{thm:fundest} with $\tB_k^\omega(s)\varphi_\gamma$ replaced by $B_k(s)$; recall the definition \eqref{twomegakdef} and \eqref{twkdef}. The operators $\ta_\omega(D)\tpsi_k(D)$ is a mollifier on spatial scale $2^{-k}$ and commutes with differentiation, hence preserves the estimates of Theorem \ref{thm:fundest}, and $\ta_\omega(D)\phi_\gamma$ satisfies the same conditions as $\phi_\gamma$. The terms $B_{k\pm 1}(s)$ will follow the same proof as for $B_k(s)$.

Without loss of generality we assume $\gamma_0=(0,e_1)$. We need establish the bounds of Theorem \ref{thm:fundest} for the function
$$
\bigl(B_k(s)\phi_{\gamma_0}\bigr)(x) = 2^{-\frac 32} 2^{-k(\frac{d+1}4)} t^{\frac{d-1}4}
\int e^{i\varphi_k(s,x,\eta)} b_k(s,x,\eta) \beta_{k,t}^{e_1}(\eta)\,d\eta.
$$
We can express this in the form
$$
\bigl(B_k(s)\phi_{\gamma_0}\bigr)(x) = 2^{-\frac 32} 2^{-k(\frac{d+1}4)} t^{\frac{d-1}4} 
\int e^{i\la y(s,x),\eta\ra} e^{ih(s,x,\eta)} b_k(s,x,\eta) \beta_{k,t}^{e_1}(\eta)\,d\eta
$$
where $y(s,x) =\nabla_{\eta} \varphi_{k}(s,x,e_1)$, and where by \eqref{linphase'} on the support of $\beta_{k,t}^{e_1}$ the function $h(s,x,\eta) =\varphi_k(s,x,\eta)-\eta\cdot\partial_{\eta} \varphi_{k}(s,x,e_1)$
satisfies
$$
\bigl|\partial_{\eta_1}^j\partial_{\eta'}^\alpha\partial_x^\beta
h(s,x,\eta)\bigr|\le C_{j,\alpha,\beta}\,2^{-kj}\,\tkhf^{|\alpha|}\,2^{\frac k2|\beta|}.
$$
This, together with Corollary \ref{cor:linsymbol} and \eqref{est:betakj}, leads to the estimates
\begin{equation}\label{eqn:linbumpest}
\Bigl|\partial_{\eta_1}^j\partial_{\eta'}^\alpha\partial_x^\beta
\Bigl(e^{ih(s,x,\eta)} b_k(s,x,\eta) \beta_{k,t}^{e_1}(\eta)\Bigr)\Bigr|
\le C_{j,\alpha,\beta}\,2^{-kj}\,\tkhf^{|\alpha|}\,2^{\frac k2 |\beta|}.
\end{equation}
We now express
$$
\bigl(B_k(s)\phi_{\gamma_0}\bigr)(x)=2^{k(\frac{d+1}4)} t^{-\frac{d-1}4}F\bigl(x,y(s,x)\bigr),
$$
where
$$
F(x,y)=2^{-\frac 32} 2^{-k(\frac{d+1}2)} t^{\frac{d-1}2} \int e^{i\la y,\eta \ra} e^{ih(s,x,\eta)} b_k(s,x,\eta) \beta_{k,t}^{e_1}(\eta)\,d\eta.
$$
The estimates \eqref{eqn:linbumpest} and integration by parts leads to the bounds
$$
\bigl|\partial_{y_1}^j\partial_{y'}^\alpha\partial_x^\beta F(x,y)\bigr|
\le 
C_{N,j,\alpha,\beta}\,2^{\frac k2|\gamma|}\,2^{k|\beta|}\,(t^{-\frac 12}2^{\frac k2})^{|\alpha|}
\bigl(1+2^k|y_1| + t^{-1}2^k|y|^2\bigr)^{-N}.
$$

We now use the chain rule to express $x$-derivatives of the composition of $F(x,y)$ with $y=y(s,x)$ as a sum of terms,
$$
\partial_{x_i}F(x,y(s,x))=(\partial_{x_i}F)(x,y(s,x))+(\nabla_yF)(x,y(s,x))\cdot\partial_{x_i}y(s,x).
$$
The $\partial_{x_i}$ in first term on the right counts as a factor of $2^{\frac k2}$ in the derivative estimates, which is better than the conclusion of Theorem \ref{thm:fundest}. Similar considerations apply to terms in the expansion of higher order derivatives. Since we will estimate individually each term arising in such an expansion, we therefore can consider functions $F$ that are functions of only $y$. That is, we assume for all $N$ that
\begin{equation}\label{eqn:Fbound}
\bigl|\partial_{y_1}^j\partial_{y'}^\alpha F(y)\bigr|
\le 
C_{N,\alpha,\beta}\,2^{kj}\,(t^{-\frac 12}2^{\frac k2})^{|\alpha|}
\bigl(1+2^k|y_1| + t^{-1}2^k|y|^2\bigr)^{-N},
\end{equation}
and prove that the composition with $y(s,x)$ satisfies for all $N$
\begin{multline}\label{eqn:Fcomp}
\bigl|\parsl_x^\alpha\partial_x^\beta F(y(s,x))\bigr|\\
\le
C_{N,\alpha,\beta}\, 2^{k|\beta|}\tkhfinv^{|\alpha|}\bigl(1+2^k|\la \nu_s,x - x_s  \ra| + t^{-1}2^{k}|x-x_s|^2\bigr)^{-N},
\end{multline}
where $\parsl_x\equiv\la\nu_s^\perp,\partial_x\ra$ denotes derivatives in directions perpendicular to $\nu_s$.

Since $y(s,x_s)=0$, and the map $x\rightarrow y(s,x)$ is a globally bi-Lipschitz map of $\Rd$, with uniform bounds on the map and its inverse, we have 
$$
|y(s,x)|^2\approx |x-x_s|^2,
$$
with the ratio of the two sides close to 1 for $c_d$ small. For a constant $c$ close to $1$, we also have
$$
c\ts\nu_s=(\nabla_x\varphi_k)(s,x_s,e_1)=(\nabla_x\partial_{\eta_1}\varphi_k)(s,x_s,e_1)=(\nabla_x y_1)(s,x_s).
$$
We also have the equality $y_1(s,x)=\varphi_k(s,x,e_1)$ by homogeneity, which by \eqref{dxphi} implies
\begin{equation}\label{eqn:dxy1}
|\partial_x^\beta y_1(s,x)|\le 
\begin{cases}C,&|\beta|=1\\ C\,2^{\frac k2(|\beta|-2)},&|\beta|\ge 2.\end{cases}
\end{equation}
Together with a first order Taylor expansion these imply that, for $0<t\le 1$,
\begin{equation}\label{eqn:xtoy}
|y_1(s,x)|+t^{-1}|y(s,x)|^2\approx |\la\nu_s,x-x_s\ra|+t^{-1}|x-x_s|^2
\end{equation}
with uniform bounds on the ratios. Together with \eqref{eqn:Fbound} this gives \eqref{eqn:Fcomp} for $j=\alpha=0$.

To bound derivatives, we use the chain rule to express $\displaystyle\parsl_x^\alpha\partial_x^\beta F(y(s,x))$ as a sum of terms of the form
$$
(\partial_{y_1}^m\partial_{y'}^\theta F)
(\parsl_x^{\alpha_1}\partial_x^{\beta_1}y_1)\cdots
(\parsl_x^{\alpha_m}\partial_x^{\beta_m}y_1)
(\parsl_x^{\alpha_{m+1}}\partial_x^{\beta_{m+1}}y')\cdots
(\parsl_x^{\alpha_{m+|\theta|}}\partial_x^{\beta_{m+|\theta|}}y')
$$
where
$$
\alpha=\sum_{j=1}^{m+|\theta|}\alpha_j\,,\qquad
\beta=\sum_{j=1}^{m+|\theta|}\beta_j\,,\qquad
m+|\theta|\le |\alpha|+|\beta|.
$$
The estimate \eqref{eqn:Fcomp} then follows from \eqref{eqn:Fbound} and \eqref{eqn:xtoy}, together with the following bounds for the derivatives of $y(s,x)$ for $|\alpha|+|\beta|\ge 1$, and where $2^{-k}\le t\le 1$,
\begin{align*}
|\parsl_x^\alpha\partial^\beta_x y_1(s,x)|&\le C_{\alpha,\beta}\,2^{k(|\beta|-1)}\tkhfinv^{|\alpha|}
\bigl(\ts 1+ t^{-\hf}2^{\frac k2}|x-x_s|\ts\bigr)\\
|\parsl_x^\alpha\partial^\beta_x y'(s,x)|&\le C_{\alpha,\beta}\,2^{k|\beta|}\tkhfinv^{|\alpha|-1}.
\end{align*}
The second of these holds by the stronger bound of $C_{\alpha,\beta}\,2^{\frac k2(|\alpha|+|\beta|-1)}$ from Theorem \ref{thm:yest}, where if $|\alpha|=0$ we use that $2^{-k}\le\tkhfinv^{-1}$ and $|\beta|\ge 1$.
For the first, if $|\alpha|=1$ and $|\beta|=0$, we use \eqref{eqn:dxy1} and that $(\parsl_x y_1)(s,x_s)=0$ to see that $|\parsl_x y_1(s,x)|\le C\,|x-x_s|$. If $|\alpha|\ge 2$ or $|\beta|\ge 1$ then the estimate follows directly from \eqref{eqn:dxy1}.
\qed

\bibliographystyle{plain}
\bibliography{BCpaper}

\end{document}